\documentclass{article}

\usepackage{amsfonts} 
\usepackage{amsmath} 
\usepackage{amssymb} 
\usepackage{amsthm} 
	\newtheorem{thm}{Theorem}[section]

	\newtheorem{cor}[thm]{Corollary}
	\newtheorem{defn}[thm]{Definition}
	\newtheorem{lem}[thm]{Lemma}
	
	\newtheorem{prop}[thm]{Proposition}
	
\usepackage{authblk}
\usepackage{arydshln} 

\usepackage[margin=3cm]{geometry} 
\usepackage{hyperref} 
\usepackage{makecell} 
\usepackage{multicol}

\usepackage{graphicx} 
\usepackage{epstopdf}

\usepackage{tikz}
\usetikzlibrary{calc, arrows, decorations.markings, decorations.pathmorphing, positioning, decorations.pathreplacing,
knots 
}

\newcommand{\HH}{\mathcal{H}}

\newcommand{\M}{\mathcal{M}}

\renewcommand{\P}{\mathcal{P}}

\newcommand{\QQ}{\mathcal{Q}}
\newcommand{\R}{\mathbb{R}}

\newcommand{\s}{\mathfrak{s}}
\renewcommand{\S}{\mathcal{S}}

\newcommand{\To}{\longrightarrow}
\newcommand{\Tr}{\mathfrak{T}}
\newcommand{\TT}{\mathcal{T}}

\newcommand{\x}{{\bf x}}

\newcommand{\Z}{\mathbb{Z}}

\DeclareMathOperator{\Conv}{Conv}

\DeclareMathOperator{\Spin}{Spin}
\DeclareMathOperator{\Supp}{Supp}

\DeclareMathOperator{\Tight}{Tight}

\begin{document}

\title{Polytopes, dualities, and Floer homology} 

\author{Daniel V. Mathews} 
\affil{School of Mathematical Sciences,
Monash University,
VIC 3800, Australia
\texttt{daniel.mathews@monash.edu}}

\maketitle

\begin{abstract}

This article is an exposition of a body of existing results, together with an announcement of recent results. We discuss a theory of polytopes associated to bipartite graphs and trinities, developed by K\'{a}lm\'{a}n, Postnikov and others. This theory exhibits a variety of interesting duality and triality relations, and extends into knot theory, 3-manifold topology and Floer homology. In recent joint work with K\'{a}lm\'{a}n, we extend this story into contact topology and contact invariants in sutured Floer homology.

\end{abstract}

\tableofcontents

\section{Introduction}

\subsection{The story}

In this article we give an exposition of a body of mathematical results which, we think, tells an interesting story. This story begins with bipartite graphs and ends with contact topology, passing through several distinct fields along the way, including planar graph theory, Tutte matchings, trees and arborescences, knots and links, 3-manifolds and Floer homology, among others. Running throughout this story is a common thread of underlying polytopes with duality and triality relations --- polytopes studied by by K\'{a}lm\'{a}n and Postnikov, among others \cite{Kalman13_Tutte, K-Murakami, K-Postnikov, Postnikov09}. And we will see that a particular number appears repeatedly throughout, in counting all these objects.

The results we discuss are the work of many authors, including Friedl, Juh\'{a}sz, K\'{a}lm\'{a}n, Murakami, Postnikov, Rasmussen, Tutte, and others. We refer to their work as we proceed. None of the results presented here are new, or our own, except in the final section, where we announce some recent results in joint work with K\'{a}lm\'{a}n \cite{Kalman-Mathews_Trinities}. These results add a new chapter to the story and connect it to 3-dimensional contact topology.

This story crosses several areas of mathematics, and is aimed at a general mathematical audience. We have therefore tried not to assume particular expertise in any of these fields, and have provided background all along the way, which may be skipped by readers knowledgeable in the respective fields. We provide proofs (or sketches thereof) for some of the results discussed, when we understand them and we think they add to the exposition. But in general we refer to original papers for full proofs and details.

Our discussion of this story is, so far as possible, accompanied by running examples. These examples are sufficiently simple that the associated structures can be calculated by hand and visualised in 3 dimensions. This makes the examples perhaps too simple for some purposes, but we hope that the advantages of accessibility and manageability outweigh the disadvantages of possible over-simplicity.

We hope this exposition helps to illustrate some interesting ideas, and to bring them to a wider audience. The story so far covers several fields of mathematics, but may yet cover even more.

\subsection{Plan of the paper}

Our approach in this paper is to tell the story ``from the ground up", starting from the beginning and incrementally adding layers of extra structure. The advantage of this approach is that it makes logical sense and allows the reader to build up their understanding step by step; the disadvantage is that the reader will not have an overview of the story until the very end. In an attempt to provide the reader with a vague overview in advance, we offer the following preview.

We begin in section \ref{sec:hypergraphs_polytope_dualities} by introducing some polytopes associated to bipartite graphs, using ideas from the theories of hypergraphs and spanning trees. We illustrate these polytopes with examples, and discuss some duality relationships between them.

In section \ref{sec:plane_graphs_dualities_trinities} we extend the story to \emph{planar} bipartite graphs, which adds an additional type of duality, namely planar duality. These considerations naturally give rise to \emph{trinities}, which contain \emph{triples} of plane bipartite graphs, and exhibit a relationship of \emph{triality}. These relationships go back to Tutte and his tree trinity theorem.

Having seen polytopes associated to bipartite graphs, and the triples of bipartite graphs that arise when they are embedded in the plane, in section \ref{sec:trinities} we consider and compare the polytopes associated to these triples of bipartite graphs. We find that the polytopes also satisfy duality and triality relationhips.

In section \ref{sec:knots_links} we extend the story to knots and links, which can be constructed from plane bipartite graphs by a procedure called the \emph{median construction}. As it turns out, this procedure is quite a general one, and many \emph{special alternating} links can be cosnstructed this way. The story of graphs, polytopes and trinities extends to the world of knots and links in several ways. For instance, one of the polytopes associated to a plane bipartite graph is closely related to the HOMFLY-PT polynomial of the associated link.

Next, in section \ref{sec:sutured} we press further into 3-dimensional topology, introducing \emph{sutured manifolds} into the story. A slight extension of the median construction produces not only a link but a sutured 3-manifold. We may then consider \emph{sutured Floer homology}, a powerful invariant of sutured manifolds based on pseudoholomorphic curves. This turns out to contain polytopes as well, and the three sutured Floer homology groups associated to a trinity exhibit similar triality relations.

Finally, in section \ref{sec:contact} we add the final (for now) ingredient to the story, \emph{contact geometry}, announcing our recent joint work. We discuss contact structures on sutured 3-manifolds. We find that the sets of tight contact structures on the three sutured manifolds associated to a trinity also obey triality relations, parallel to those between associated polytopes.

\subsection{Acknowledgments}

I would like to thank Tam\'{a}s K\'{a}lm\'{a}n for introducing me to this wonderful area of mathematics. This work was supported by Australian Research Council grant DP160103085.

\section{Hypergraphs and polytope dualities}
\label{sec:hypergraphs_polytope_dualities}

We begin by discussing some interesting constructions of polytopes associated to bipartite graphs, following Postnikov \cite{Postnikov09} and K\'{a}lm\'{a}n \cite{Kalman13_Tutte}. Our conventions partly follow both these references, and we refer to them for further details.

\subsection{Bipartite graphs and hypergraphs}
\label{sec:bipartite_hypergraphs}

Recall a graph $G$ is \emph{bipartite} if its vertices can be partitioned into two sets $U \sqcup V$ so that each edge connects a vertex of $U$ to a vertex of $V$. We call $U, V$ the \emph{vertex classes} of $G$.

A \emph{hypergraph} is a generalisation of a graph. We can think of a graph $G$ as consisting of a set of vertices $V$, and a set $E$ of edges, where each edge is a pair of vertices. A hypergraph generalises this notion of edge to \emph{hyperedge}, where a hyperedge is now a subset of vertices of arbitrary size.
\begin{defn}
A \emph{hypergraph} $\mathcal{H} = (V,E)$ consists of
\begin{itemize}
\item a finite set $V$ (the \emph{vertices}), and 
\item a finite multiset $E$ of non-empty subsets of $V$ (the \emph{hyperedges}).
\end{itemize}
\end{defn}
When each element of $E$ has cardinality $2$, a hypergraph reduces to a graph.

Hypergraphs and bipartite graphs are closely related. From a hypergraph $\mathcal{H} = (V,E)$ we can form a bipartite graph $G$ (without multiple edges) by taking 
$V$ and $E$ as the vertex classes, and connecting $v \in V$ to $e \in E$ by an edge of $G$ if and only if $v \in e$. 

Conversely, a bipartite graph $G$ with vertex classes $U, V$ 
can be viewed as a hypergraph $\HH$. We may take the vertices of $\HH$ as $V$. For each $u \in U$, let $e(u)$ be the set of vertices connected to $u$ by an edge. Then each $e(u) \subseteq V$ and we take the hyperedges of $\mathcal{H}$ to be 
$\{ e(u) \mid u \in U \}$.

In fact, reversing the roles of $U$ and $V$, the bipartite graph $G$ can be viewed as a hypergraph in two distinct ways. These two viewpoints are related by reversing the roles of vertices and hyperedges and these two hypergraphs are called \emph{abstract duals} of each other. We denote them $\mathcal{H}$ and $\overline{\mathcal{H}}$.

\subsection{Polytopes}

Recall a \emph{polytope} is the convex hull of a finite set of points in the standard Euclidean $\R^n$.

Throughout this section, let $G$ be a bipartite graph with vertex classes $U = \{1, 2, \ldots, m\}$ and $V = \{\bar{1}, \bar{2}, \ldots, \bar{n}\}$, so $|U| = m$ and $|V| = n$. Let $e(1), \ldots, e(m)$ be subsets of $V$ defined as above:
\[
e(u) = \{ v \; \mid \; (u, v) \text{ is an edge of } G \}.
\]
In other words, $e(u) \subseteq V = \{\bar{1}, \bar{2}, \ldots, \bar{n} \}$ consists of all vertices connected to $u \in U$ by an edge in $G$. Similarly, we can define $e(\bar{1}), \ldots, e(\bar{n})$ by $e(v) = \{ u \; \mid \; (u, v) \text{ is an edge of } G \} \subseteq U$. That is, $e(v)$ consists of everything connected to $v$ in $G$. Let $\mathcal{H}$ be the corresponding hypergraph, with vertices $V$, and hyperedges 
$\{ e(u) \; \mid \; u \in U \}$. The abstract dual hypergraph $\overline{\HH}$ has vertices $U$ and hyperedges 
$\{ e(v) \; \mid \; v \in V \}$.

\begin{defn} \
\begin{enumerate}
\item
The \emph{GP polytope} of $\HH$ is 
\[
\P_\HH = \Delta_{e(1)} + \Delta_{e(2)} + \cdots + \Delta_{e(m)} 
= \sum_{u \in U} \Delta_{e(u)} \subset \R^n = \R^V.
\]
\item
The \emph{trimmed GP polytope} of $\HH$ is
\[
\P_\HH^- = \P_\HH - \Delta_{V} = \left( \sum_{u \in U} \Delta_{e(u)} \right) - \Delta_V \subset \R^{n} = \R^V.
\]
\end{enumerate}
\end{defn}
Postnikov called $\P_\HH$ and $\P_\HH^-$ the \emph{generalised permutohedron} and \emph{trimmed generalised permutohedron} associated to the bipartite graph $G$ \cite[sec. 9]{Postnikov09}. K\'{a}lm\'{a}n called $\P_\HH^-$ the \emph{hypertree polytope}.  For our purposes we do not need permutohedra, or their specialisations or generalisations, and we introduce hypertrees later, so we use the abbreviated name ``GP". 

Some remarks on these definitions are in order. Throughout, we identify $\R^V$ with $\R^n$, with basis vectors ${\bf i}_{\bar{1}}, \ldots, {\bf i}_{\bar{n}}$ corresponding to the elements of $V$, and similarly $\R^U = \R^m$, with basis ${\bf i}_1, \ldots, {\bf i}_n$. For a subset $I \subseteq V = \{ \bar{1}, \ldots, \bar{n} \}$, we denote by $\Delta_I$ the convex hull of $\{ {\bf i}_v \; \mid \; v \in I\}$ in $\R^V = \R^n$. Thus each $\Delta_I$ is a standard simplex in the appropriate coordinates, and $\Delta_V$ is the convex hull of all of ${\bf i}_{\bar{1}}, \ldots, {\bf i}_{\bar{n}}$.
In the definition of $\P_\HH$ the additions are \emph{Minkowski sums}, and in the definition of $\P_\HH^-$ the subtraction is a \emph{Minkowski difference}. For two sets $A,B \subseteq \R^n$, their Minkowski sum is
\[
A + B = \{ a + b \; \mid \; a \in A, \; b \in B \},
\]
and their Minkowski difference is
\[
A - B = \{ x \in \R^n \; \mid \; x + B \subseteq A \}.
\]
In particular, if $A$ and $B$ are polytopes, then $(A+B)-B = A$. 

Thus, the GP polytope $\P_\HH$ is given by taking Minkowski sums of simplices in $\R^V$ corresponding to hyperedges of $\HH$, and the trimmed version is then obtained by subtracting a full-dimensional simplex. 

Taking the abstract dual $\overline{\HH}$, with vertices $U$ and hyperedges
$\{ e(v) \; \mid \; v \in V \}$,
the corresponding polytopes are
\begin{align*}
\P_{\overline{\HH}} &= \Delta_{e({\bar{1}})} + \cdots + \Delta_{e({\bar{n}})} = \sum_{v \in V} \Delta_{e(v)} \subset 
\R^U \\
\P_{\overline{\HH}}^- &= \P_{\overline{\HH}} - \Delta_U = \left( \sum_{v \in V} \Delta_{e(v)} \right) - \Delta_U \subset 
\R^U
\end{align*}

We also consider polytopes based on \emph{subgraphs} and their degrees in the bipartite graph $G$; in particular we consider \emph{spanning trees}, of $G$.  For a subgraph $T$ of $G$, define the \emph{$U$-degree} vector $\deg_U T \in \Z_{\geq 0}^U$ as the vector of degrees of $T$ at the vertices $1, \ldots, m$ of $U$, and similarly the \emph{$V$-degree vector}:
\[
\deg_U T = \left( \deg_1 T, \ldots, \deg_m T \right) \in 
\Z_{\geq 0}^U, \quad
\deg_V T = \left( \deg_{\bar{1}} T, \ldots, \deg_{\bar{n}} T \right) \in 
\Z_{\geq 0}^V
\]
Further, for $u \in U$, let $T(u) = \{ v \in V \; \mid \; (u,v) \text{ is an edge of } T \}$, so $T(u) \subseteq e(u) \subseteq V$; that is, $T(u)$ consists of the vertices connected to $u$ by an edge in $T$. Note $T$ is uniquely specified by the sets $T(u)$, for $u \in U$; indeed, $T$ corresponds to a sub-hypergraph $\TT$ of $\HH$ with vertices $V$ and hyperedges $\{ T(u) \; \mid \; u \in U \}$. Similarly, $T$ is determined by $T(v) = \{u \in U \; \mid \; (u,v) \text{ is an edge of } T \} \subseteq e(v)$ over $v \in V$, which corresponds to a sub-hypergraph $\overline{\TT}$ of $\overline{\HH}$ with vertices $U$ and hyperedges $\{ T(v) \; \mid \; v \in V \}$.  

Regarding a subgraph $T$ of $G$ as a sub-hypergraph $\TT$ of $\HH$, we have the GP polytope
\[
\P_\TT = \Delta_{T(1)} + \Delta_{T(2)} + \cdots + \Delta_{T(m)}
= \sum_{u \in U} \Delta_{T(u)} \subset 
\R^V.
\]
Since each $T(u) \subseteq e(u)$, we have $\Delta_{T(u)} \subseteq \Delta_{e(u)}$ for each $u \in U$. Thus, when all $T(u)$ are nonempty, $\P_\TT \subseteq \P_\HH \subset \R^V$. In particular, this is the case when $T$ is a spanning tree. (However if $T$ has an isolated vertex in $U$, but $G$ does not, then $\Delta_{T(u)}$ will be empty, but $\Delta_{e(u)}$ will not, and $\P_\TT$ will not be contained in $\P_\HH$.) Similar considerations apply to $\P_{\overline{\TT}}$ and $\P_{\overline{\HH}}$. 

When $T$ is a spanning tree of $G$, then we will also say $\TT$ is a spanning tree of $\HH$, and $\overline{\TT}$ is a spanning tree of $\overline{\HH}$. In this case, $T$ has degree at least $1$ at every vertex; following \cite{Juhasz-Kalman-Rasmussen12, Kalman13_Tutte, Kalman-Mathews_Trinities}, we make the following definition.
\begin{defn} \
\label{def:hypertree}
\begin{enumerate}
\item
The \emph{hypertree} of $\TT$ is the vector
\[
f_\TT = \deg_U T - (1, \ldots, 1) \in \Z_{\geq 0}^U.
\]
\item
The \emph{hypertree} of $\overline{\TT}$ is the vector
\[
f_{\overline{\TT}} = \deg_V T - (1, \ldots, 1) \in \Z_{\geq 0}^V.
\]
\end{enumerate}
\end{defn}
Thus for a spanning tree $\TT$ of $\HH$, the hypertree of $\TT$ is the vector whose coordinates describe how many elements are selected from each hyperedge of $\HH$ (minus $1$).

We can form a polytope out of hypertrees as follows.
\begin{defn}
The \emph{hypertree polytope} of $\HH$ is
\[
\S_\HH = \Conv \left\{ f_\TT \; \mid \; \TT \text{ a spanning tree of } \HH \right\} \subset 
\R^U.
\]
\end{defn}
Here $\Conv (A)$ denotes the convex hull of $A$.
Similarly, the hypertree polytope of $\overline{\HH}$ is 
\[
\S_{\overline{\HH}}= \Conv \left\{ f_{\overline{\TT}} \; \mid \; \overline{\TT} \text{ a spanning tree of } \overline{\HH} \right\} \subset \R^V.
\]

So far all polytopes lie in $\R^U$ or $\R^V$. Our final polytope lies in the direct sum $\R^U \oplus \R^V$.
\begin{defn}
The \emph{root polytope} of $G$ is
\[
\QQ_G = \Conv \left\{ {\bf i}_u - {\bf i}_v \mid u \in U, \; v \in V, \; (u, v) \text{ is an edge of } G \right\} \subset \R^{m+n} = \R^U \oplus \R^V.
\]
\end{defn}
Note that the set of points ${\bf i}_u - {\bf i}_v$ with $u,v$ connected by an edge can alternatively be written as $\bigcup_{u \in U} \{u\} \times (-e(u))$ or $\bigcup_{v \in V} e(v) \times \{-v\}$. These points are effectively a plot of incidence relations in $\HH$ or $\overline{\HH}$, and the root polytope is their convex hull. 

We observe that if $T$ is a subgraph of $G$, then $\QQ_T \subseteq \QQ_G$, being the convex hull of a smaller collection of points.

Throughout, we use cursive letters to denote polytopes. We will also be interested in the integer points of these polytopes, which we denote by roman letters:
\[
P_\HH = \P_\HH \cap \Z^V, \quad
P_\HH^- = \P_\HH^- \cap \Z^V, \quad
S_\HH = \S_\HH \cap \Z^V, \quad
Q_G = \QQ_G \cap \Z^V, \quad \text{etc.}
\]

\subsection{An example}
\label{sec:polytope_example}

We now proceed through the calculations of the above polytopes for a specific example. The example is simple, and the calculations may sometimes appear tedious, but significantly many coincidences and patterns arise that we believe that the reader will find themselves rewarded for the effort.

Let $G$ be the bipartite graph $G$ shown in figure \ref{fig:graph_G}. We have $m=2$ and $n=3$, with vertices of $U$ drawn in green, and vertices of $V$ drawn in blue. Let $\HH$ and $\overline{\HH}$ be the corresponding abstract dual hypergraphs.

\begin{figure}
\begin{center}
\begin{tikzpicture}[scale = 1.5]
\coordinate (u1) at (0,0);
\coordinate (u2) at (0,-2);
\coordinate (v1) at (0,1);
\coordinate (v2) at (1,-1);
\coordinate (v3) at (-1,-1);

\draw [ultra thick, red] (u1) -- (v1);
\draw [ultra thick, red] (u1) -- (v2);
\draw [ultra thick, red] (u1) -- (v3);
\draw [ultra thick, red] (u2) -- (v2);
\draw [ultra thick, red] (u2) -- (v3);

\foreach \x/\word in {(u1)/{$1$}, (u2)/{$2$}}
{
\draw [green!50!black, ultra thick, fill=white] \x circle  (8pt);
\draw \x node {$\word$};
}

\foreach \x/\word in {(v1)/{$\overline{1}$}, (v2)/{$\overline{2}$}, (v3)/{$\overline{3}$}}
{
\draw [blue, ultra thick, fill=white] \x circle  (8pt);
\draw \x node {\word};
}
\end{tikzpicture}
\end{center}
\caption{The bipartite graph $G$.}
\label{fig:graph_G}
\end{figure}
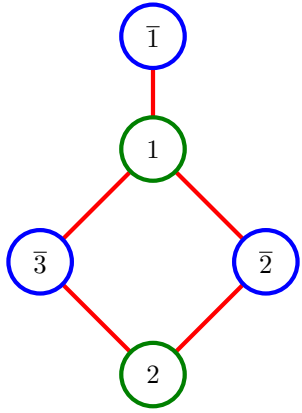

We first examine $\HH$. Its hyperedges are $e(1) = \{ \bar{1}, \bar{2}, \bar{3} \}$, $e(2)= \{ \bar{2}, \bar{3} \}$, and its GP polytope is
\begin{align*}
\P_\HH &= \Delta_{e(1)} + \Delta_{e(2)} = \Delta_{123} + \Delta_{23} 
= \Conv \{ {\bf i}_{\bar{1}}, {\bf i}_{\bar{2}}, {\bf i}_{\bar{3}} \} + \Conv \{ {\bf i}_{\bar{2}}, {\bf i}_{\bar{3}} \} \\
&= \Conv \{ {\bf i}_{\bar{1}} + {\bf i}_{\bar{2}}, {\bf i}_{\bar{1}} + {\bf i}_{\bar{3}}, 2{\bf i}_{\bar{2}}, 2{\bf i}_{\bar{3}} \} \\
&= \Conv \{ (1,1,0), (1,0,1), (0,2,0), (0,0,2) \} \subset \R^3.
\end{align*}
Here notation like $\Delta_{123}$ is shorthand for $\Delta_{\{\bar{1}, \bar{2}, \bar{3}\}}$.
(One may wonder why ${\bf i}_2 + {\bf i}_3 = (0,1,1)$ does not appear in the two lines above; this is because it is halfway from $(0,2,0)$ to $(0,0,2)$.) This is a quadrilateral lying in the plane $x+y+z = 2$ in $\R^3$. See figure \ref{fig:polytope_1}. It contains five integer points:
\[
P_\HH = \{ (1,1,0), (1,0,1), (0,2,0), (0,1,1), (0,0,2) \}.
\]
The trimmed GP polytope is then
\begin{align*}
\P_\HH^- 
= (\Delta_{123} + \Delta_{23}) - \Delta_{123} 
= \Delta_{23} 
= \Conv \{ (0,1,0), (0,0,1) \} \subset \R^3,
\end{align*}
which is just the interval from $(0,1,0)$ to $(0,0,1)$, so its integer points are $P^-_\HH = \{ (0,1,0), (0,0,1) \}$.

\begin{figure}
\begin{center}
\includegraphics[scale=0.7]{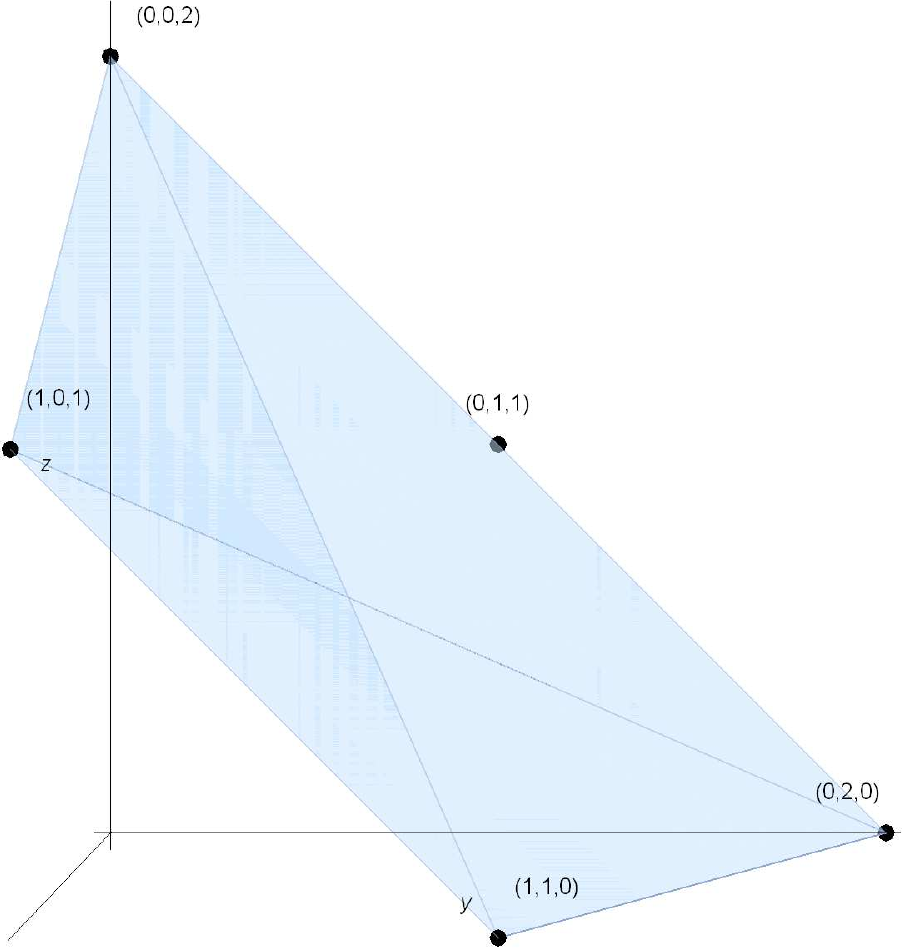}
\begin{tikzpicture}[scale = 1]
\coordinate [label=left:{$(1,1,0)$}] (110) at (0,0);
\coordinate [label=left:{$(1,0,1)$}] (101) at (0,2);
\coordinate [label=right:{$(0,2,0)$}] (020) at (2,-1);
\coordinate [label=right:{$(0,1,1)$}] (011) at (2,1);
\coordinate [label=right:{$(0,0,2)$}] (002) at (2,3);
\draw [ultra thick] (110) -- (101) -- (002) -- (020) -- cycle;
\end{tikzpicture}
\end{center}
\caption{The polytope $\P_\HH$, a 2-dimensional polytope (quadrilateral) in $\R^3$.}
\label{fig:polytope_1}
\end{figure}
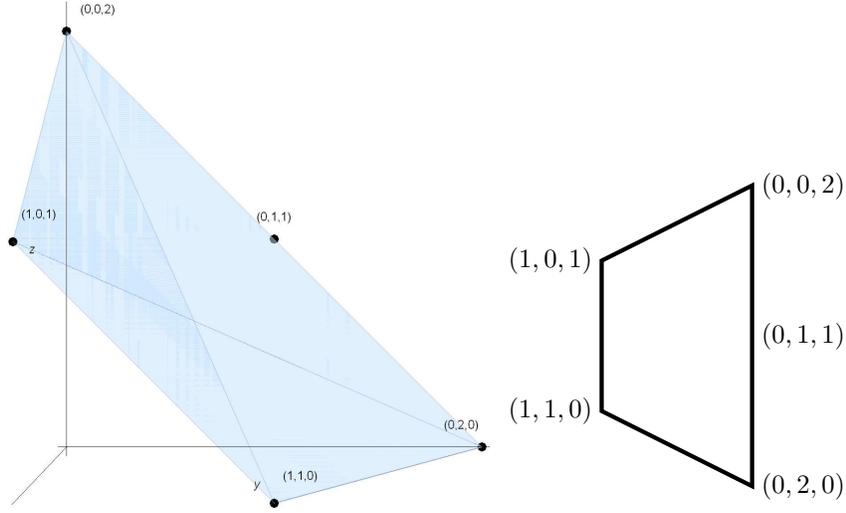

The spanning trees of $G$ are, in an obvious notation
\[
T_1 = \{ 1\bar{1}, 1\bar{2}, 1\bar{3}, 2\bar{2} \}, \quad
T_2 = \{ 1\bar{1}, 1\bar{2}, 1\bar{3}, 2\bar{3} \}, \quad
T_3 = \{ 1\bar{1}, 1\bar{2}, 2\bar{2}, 2\bar{3} \}, \quad
T_4 = \{ 1\bar{1}, 1\bar{3}, 2\bar{2}, 2\bar{3} \}.
\]
as shown in figure \ref{fig:spanning_trees_in_G}. The GP polytopes of the corresponding spanning trees $\TT_1, \ldots, \TT_4$ of $\HH$ are 
\begin{align*}
\P_{\TT_1} &= 
\Delta_{123} + \Delta_2 
= \Conv \{ (1,1,0), (0,2,0), (0,1,1) \} \\
\P_{\TT_2} &= \Delta_{123} + \Delta_3 = \Conv \{ (1,0,1), (0,1,1), (0,0,2) \} \\
\P_{\TT_3} &= \Delta_{12} + \Delta_{23} 
= \Conv \{ (1,1,0), (1,0,1), (0,2,0), (0,1,1) \} \\
\P_{\TT_4} &= \Delta_{13} + \Delta_{23} = \Conv \{ (1,1,0), (1,0,1), (0,1,1), (0,0,2) \}.
\end{align*}
These are all 2-dimensional polytopes in the plane $x+y+z = 2$ in $\R^3$, all triangles or parallelograms, containing 3 or 4 integer points respectively. They are all subsets of $\P_\HH$; We draw them in figure \ref{fig:tree_triangulation}. We observe, in fact, that $\P_\HH$ \emph{decomposes} nicely as a union of these tree polytopes in \emph{two different ways}:
\[
\P_\HH = \P_{\TT_2} \cup \P_{\TT_3} = \P_{\TT_1} \cup \P_{\TT_4}.
\]

The spanning trees have $U$-degrees
\[
\deg_U T_1 = (3,1), \quad
\deg_U T_2 = (3,1), \quad
\deg_U T_3 = (2,2), \quad
\deg_U T_4 = (2,2)
\]
hence hypertrees
\[
f_{\TT_1} = (2,0), \quad
f_{\TT_2} = (2,0), \quad
f_{\TT_3} = (1,1), \quad
f_{\TT_4} = (1,1)
\]
so the hypertree polytope is given by the interval
\[
\S_\HH = \Conv \{ (2,0), (1,1) \} \subset \R^2,
\quad \text{with integer points} \quad
S_\HH = \{ (2,0), (1,1) \}.
\]
Note that just as $T_1$ and $T_4$ have GP polytopes which cover $\P_\HH$, their hypertrees cover $S_\HH$; similarly for $T_2$ and $T_3$.

\begin{figure}
\begin{center}
\begin{tikzpicture}[scale = 1.4]
\coordinate (u1) at (0,0);
\coordinate (u2) at (0,-2);
\coordinate (v1) at (0,1);
\coordinate (v2) at (1,-1);
\coordinate (v3) at (-1,-1);

\draw [ultra thick, red] (u1) -- (v1);
\draw [ultra thick, red] (u1) -- (v2);
\draw [ultra thick, red] (u1) -- (v3);
\draw [ultra thick, red] (u2) -- (v2);

\foreach \x/\word in {(u1)/{$1$}, (u2)/{$2$}}
{
\draw [green!50!black, ultra thick, fill=white] \x circle  (8pt);
\draw \x node {$\word$};
}

\foreach \x/\word in {(v1)/{$\overline{1}$}, (v2)/{$\overline{2}$}, (v3)/{$\overline{3}$}}
{
\draw [blue, ultra thick, fill=white] \x circle  (8pt);
\draw \x node {\word};
}

\draw (0,-3) node {$T_1$};
\end{tikzpicture}
\begin{tikzpicture}[scale = 1.4]
\coordinate (u1) at (0,0);
\coordinate (u2) at (0,-2);
\coordinate (v1) at (0,1);
\coordinate (v2) at (1,-1);
\coordinate (v3) at (-1,-1);

\draw [ultra thick, red] (u1) -- (v1);
\draw [ultra thick, red] (u1) -- (v2);
\draw [ultra thick, red] (u1) -- (v3);
\draw [ultra thick, red] (u2) -- (v3);

\foreach \x/\word in {(u1)/{$1$}, (u2)/{$2$}}
{
\draw [green!50!black, ultra thick, fill=white] \x circle  (8pt);
\draw \x node {$\word$};
}

\foreach \x/\word in {(v1)/{$\overline{1}$}, (v2)/{$\overline{2}$}, (v3)/{$\overline{3}$}}
{
\draw [blue, ultra thick, fill=white] \x circle  (8pt);
\draw \x node {\word};
}

\draw (0,-3) node {$T_2$};
\end{tikzpicture}
\begin{tikzpicture}[scale = 1.4]
\coordinate (u1) at (0,0);
\coordinate (u2) at (0,-2);
\coordinate (v1) at (0,1);
\coordinate (v2) at (1,-1);
\coordinate (v3) at (-1,-1);

\draw [ultra thick, red] (u1) -- (v1);
\draw [ultra thick, red] (u1) -- (v2);
\draw [ultra thick, red] (u2) -- (v2);
\draw [ultra thick, red] (u2) -- (v3);

\foreach \x/\word in {(u1)/{$1$}, (u2)/{$2$}}
{
\draw [green!50!black, ultra thick, fill=white] \x circle  (8pt);
\draw \x node {$\word$};
}

\foreach \x/\word in {(v1)/{$\overline{1}$}, (v2)/{$\overline{2}$}, (v3)/{$\overline{3}$}}
{
\draw [blue, ultra thick, fill=white] \x circle  (8pt);
\draw \x node {\word};
}

\draw (0,-3) node {$T_3$};
\end{tikzpicture}
\begin{tikzpicture}[scale = 1.4]
\coordinate (u1) at (0,0);
\coordinate (u2) at (0,-2);
\coordinate (v1) at (0,1);
\coordinate (v2) at (1,-1);
\coordinate (v3) at (-1,-1);

\draw [ultra thick, red] (u1) -- (v1);
\draw [ultra thick, red] (u1) -- (v3);
\draw [ultra thick, red] (u2) -- (v2);
\draw [ultra thick, red] (u2) -- (v3);

\foreach \x/\word in {(u1)/{$1$}, (u2)/{$2$}}
{
\draw [green!50!black, ultra thick, fill=white] \x circle  (8pt);
\draw \x node {$\word$};
}

\foreach \x/\word in {(v1)/{$\overline{1}$}, (v2)/{$\overline{2}$}, (v3)/{$\overline{3}$}}
{
\draw [blue, ultra thick, fill=white] \x circle  (8pt);
\draw \x node {\word};
}

\draw (0,-3) node {$T_4$};
\end{tikzpicture}
\end{center}
\caption{Spanning trees in $G$.}
\label{fig:spanning_trees_in_G}
\end{figure}
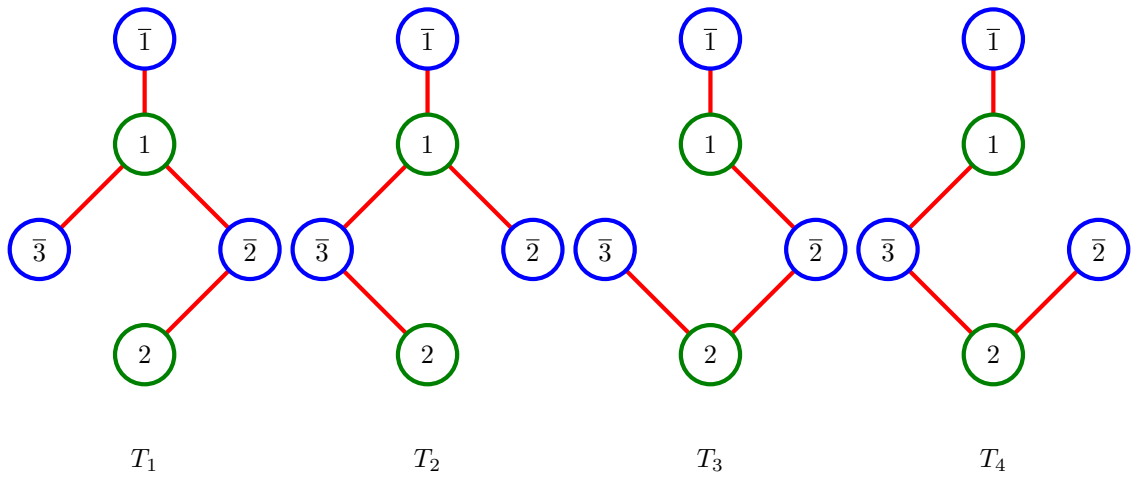

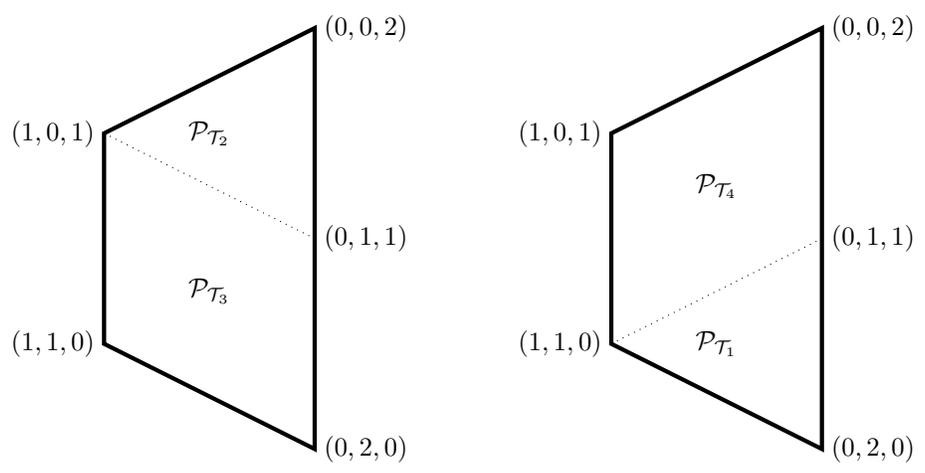
\begin{figure}
\begin{center}
\begin{tikzpicture}[scale = 1.4]
\coordinate [label=left:{$(1,1,0)$}] (110) at (0,0);
\coordinate [label=left:{$(1,0,1)$}] (101) at (0,2);
\coordinate [label=right:{$(0,2,0)$}] (020) at (2,-1);
\coordinate [label=right:{$(0,1,1)$}] (011) at (2,1);
\coordinate [label=right:{$(0,0,2)$}] (002) at (2,3);
\draw [ultra thick] (110) -- (101) -- (002) -- (020) -- cycle;
\draw [dotted] (101) -- (011);
\draw (1,2) node {$\P_{\TT_2}$};
\draw (1,0.5) node {$\P_{\TT_3}$};
\end{tikzpicture}
\hspace{1 cm}
\begin{tikzpicture}[scale = 1.4]
\coordinate [label=left:{$(1,1,0)$}] (110) at (0,0);
\coordinate [label=left:{$(1,0,1)$}] (101) at (0,2);
\coordinate [label=right:{$(0,2,0)$}] (020) at (2,-1);
\coordinate [label=right:{$(0,1,1)$}] (011) at (2,1);
\coordinate [label=right:{$(0,0,2)$}] (002) at (2,3);
\draw [ultra thick] (110) -- (101) -- (002) -- (020) -- cycle;
\draw [dotted] (110) -- (011);
\draw (1,1.5) node {$\P_{\TT_4}$};
\draw (1,0) node {$\P_{\TT_1}$};
\end{tikzpicture}
\end{center}
\caption{Tree polytopes for spanning trees.}
\label{fig:tree_triangulation}
\end{figure}

Let us now turn to the abstract dual $\overline{\HH}$. We have $e(\bar{1}) = \{1\}$, $e(\bar{2}) = (1,2)$ and $e(\bar{3}) = (1,2)$, so
\[
\P_{\overline{\HH}} = \Delta_1 + 2\Delta_{12}
= \Conv\{ (3,0), (1,2) \} \subset \R^2.
\]
This is an interval in $\R^2$ containing three integer points
\[
P_{\overline{\HH}} = \{ (3,0), (2,1), (1,2) \}.
\]
The trimmed GP polytope is
\[
\P^-_{\overline{\HH}} = \Delta_1 + \Delta_{12} = \Conv \{ (2,0), (1,1) \}
\quad \text{so} \quad
P^-_{\overline{\HH}} = \{ (2,0), (1,1) \}.
\]
We have seen this polytope before: it is also the hypertree polytope of $\HH$!
\[
\P^-_{\overline{\HH}} = \S_\HH.
\]
Next we can compute GP polytopes of spanning trees $\overline{\TT_1}, \ldots, \overline{\TT_4}$ of $\overline{\HH}$:
\begin{align*}
\P_{\overline{\TT_1}} &= \Delta_1 + \Delta_{12} + \Delta_1 = \Conv \{ (3,0), (2,1) \} \\
\P_{\overline{\TT_2}} &= \Delta_1 + \Delta_1 + \Delta_{12} = \Conv \{ (3,0), (2,1) \} \\
\P_{\overline{\TT_3}} &= \Delta_1 + \Delta_{12} + \Delta_2 = \Conv  \{ (2,1), (1,2) \} \\
\P_{\overline{\TT_4}} &= \Delta_1 + \Delta_2 + \Delta_{12} = \Conv \{ (2,1), (1,2) \}
\end{align*}
Unlike for $\HH$, we note that some of the tree polytopes coincide: $\P^{T_1}_{\overline{\HH}} = \P^{T_2}_{\overline{\HH}}$ and $\P^{T_3}_{\overline{\HH}} = \P^{T_4}_{\overline{\HH}}$. However, we do again find that the GP polytope $\P_{\overline{\HH}}$ decomposes into GP polytopes of spanning trees in a similar fashion as $\overline{\HH}$:
\[
\P_{\overline{\HH}} = \P_{\overline{\TT_1}} \cup \P_{\overline{\TT_4}}
= \P_{\overline{\TT_2}} \cup \P_{\overline{\TT_3}}.
\]
The spanning trees have 
hypertrees
\[
f_{\overline{\TT_1}} = (0,1,0), \quad
f_{\overline{\TT_2}} = (0,0,1), \quad
f_{\overline{\TT_3}} = (0,1,0), \quad
f_{\overline{\TT_4}} = (0,0,1)
\]
so the hypertree polytope is
\[
\S_{\overline{\HH}} = \Conv \left\{ (0,1,0), \; (0,0,1) \right\} \subset \R^3,
\quad \text{with integer points} \quad
S_{\overline{\HH}} = \{ (0,1,0), (0,0,1) \}.
\]
This polytope is again familiar: it is the trimmed GP polytope of $\HH$!
\[
\P^-_\HH = \S_{\overline{\HH}}.
\]
Yet again the hypertrees of $T_1$ and $T_4$ cover $S_{\overline{\HH}}$, as do the hypertrees of $T_2$ and $T_3$. 

We observe that both $\HH$ and $\overline{\HH}$ have precisely two hypertrees, and these are precisely the lattice points in the respective hypertree polytopes, so 
\[
| S_\HH | = | S_{\overline{\HH}} |.
\]

An observation which will be useful in the sequel is to compare the hypertrees $f_{\overline{\TT_1}}, f_{\overline{\TT_3}}$ and polytopes $\P_{\TT_1}, \P_{\TT_3}$.  Note that $\deg_V T_1 = \deg_V T_3 = (1,2,1)$ and $f_{\overline{\TT_1}} = f_{\overline{\TT_3}} = (0,1,0)$. Indeed we can obtain $T_1$ from $T_3$ by removing the edge $2\bar{3}$ and replacing it with the edge $1\bar{3}$; as both edges involve $\bar{3} \in V$ the $V$-degree does not change.

This fact is related to the fact that the GP polytopes of $\TT_1$ and $\TT_3$ (NB: \emph{not} $\overline{\TT}_1, \overline{\TT}_3$) are \emph{nested} one inside the other: $\P_{\TT_1} \subset \P_{\TT_3}$ (see figure \ref{fig:tree_triangulation}). Indeed, in replacing the edge $2\bar{3}$ with $1\bar{3}$, we remove $\bar{3}$ from $T(2)$ and add $\bar{3}$ to $T(1)$, so the subgraph polytopes change from
\[
\P_{\TT_3} = \Delta_{12} + \Delta_{23} 
\quad \text{to} \quad
\Delta_{123} + \Delta_2 = \P_{\TT_1}.
\]
Now $\Delta_{12} + \Delta_{23} \supset \Delta_{123} + \Delta_2$, which is an instance of the general fact that $\Delta_X + \Delta_Y \supseteq \Delta_{X \cup Y} + \Delta_{X \cap Y}$ (an easy exercise).

Finally, let us turn to the root polytope of $G$, which is given by
\begin{align*}
\QQ_G &= \Conv \{ {\bf i}_1 - {\bf i}_{\bar{1}}, {\bf i}_1 - {\bf i}_{\bar{2}}, {\bf i}_1 - {\bf i}_{\bar{3}}, {\bf i}_2 - {\bf i}_{\bar{2}}, {\bf i}_2 - {\bf i}_{\bar{3}} \} \\
&= \Conv \{ (1,0;-1,0,0), (1,0;0,-1,0), (1,0;0,0,-1), (0,1;0,-1,0), (0,1;0,0,-1) \} \subset \R^2 \oplus \R^3.
\end{align*}
Although $\QQ_G$ lives in $\R^5$, it is 3-dimensional. In order to visualise it, we make use of a linear projection $\pi: \R^5 \To \R^3$ which is injective on $\QQ_G$. Note $\QQ_G$ lies in the affine subspace of $\R^5$ given by $x_1 + x_2 = 1$ and $x_3 + x_4 + x_5 = -1$ (where $x_1, \ldots, x_5$ are coordinates on $\R^5$). If we introduce new coordinates $(x,y,z)$ given by $x = x_3 - x_4$, $y = x_4 - x_5$ and $z = x_1 - x_2$ then we can define $\pi$ by $(x_1, x_2, x_3, x_4, x_5) \mapsto (x,y,z)$, and obtain
\[
\pi(\QQ_G) = \Conv \{ (-1,0,1), (1,-1,1), (0,1,1), (1,-1,-1), (0,1,-1) \}
\]
which is the 3-dimensional polytope shown in figure \ref{fig:polytope_2}. Four of the vertices lie in the plane $2x+y=1$ and form a rectangle; together with the vertex $(-1,0,1)$ we see $\QQ_G$ is a rectangular pyramid.

\begin{figure}
\begin{center}
\includegraphics[scale=0.5]{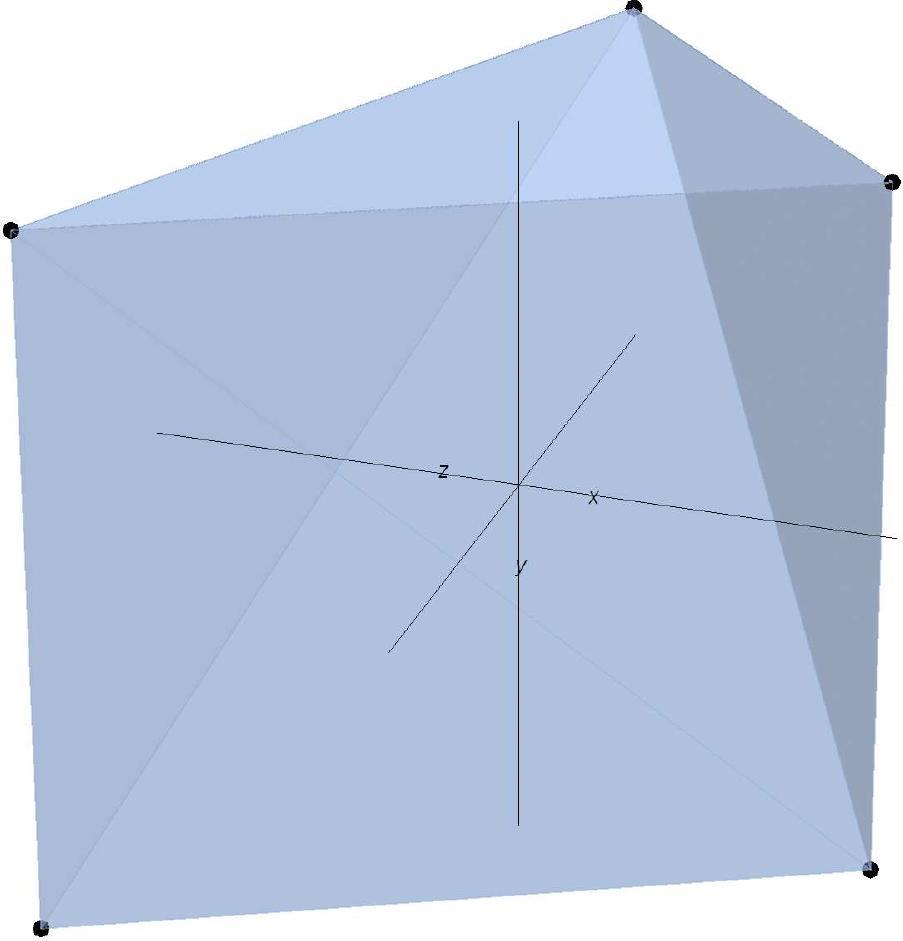}
\hspace{1 cm}
\includegraphics[scale=0.6]{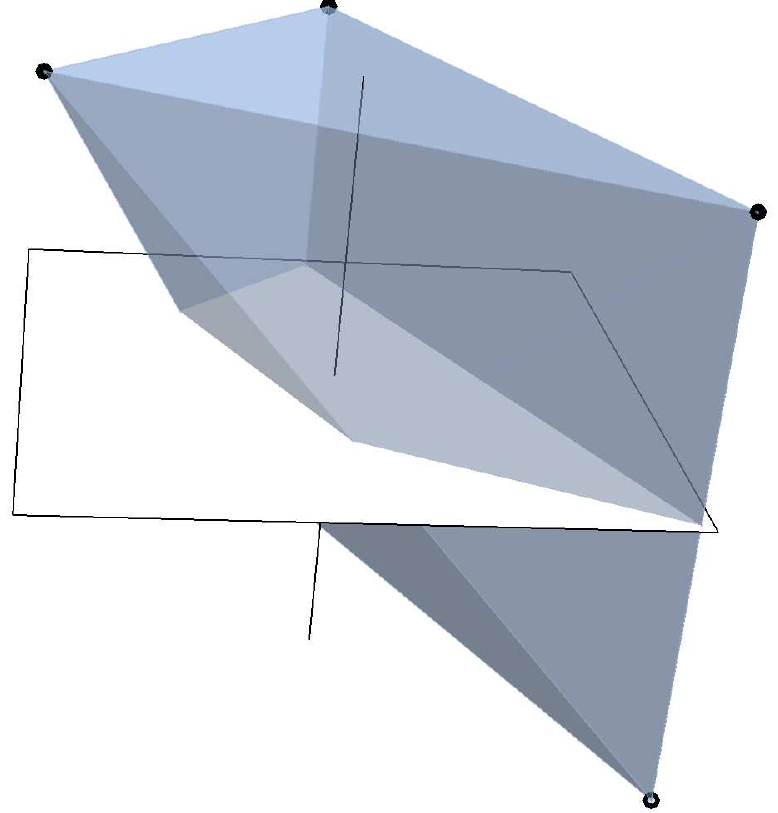}
\end{center}
\caption{The polytope $\QQ_G$, a 3-dimensional polytope in $\R^5$, projected to $\R^3$ by $\pi$. It is an irregular rectangular pyramid. Left: the base of the pyramid is the rectangle depicted at front. Right: the slice of $\QQ_G$ with $x_1 = x_2 = 1/2$ is a quadrilateral equivalent to $\P_\HH$.}
\label{fig:polytope_and_plane}
\label{fig:polytope_2}
\end{figure}

Now if we intersect $\QQ_G$ with $x_1 = x_2 = 1/2$ (corresponding to $z=0$ under projection by $\pi$), we obtain a quadrilateral as shown in figure \ref{fig:polytope_and_plane} (right). Its vertices in $\R^5 = \R^2 \times \R^3$ are
\begin{gather*}
\left( \frac{1}{2}, \frac{1}{2} \right) \times \left\{ 
(-\frac{1}{2},-\frac{1}{2},0), (-\frac{1}{2},0,-\frac{1}{2}), (0,0,-1), (0,-1,0) \right\} \\
= \left( \frac{1}{2}, \frac{1}{2} \right) \times -\frac{1}{2} 
\Bigg\{ (1,1,0), (1,0,1), (0,0,2), (0,2,0) \Bigg\}
\end{gather*}
Thus, the $(x_1, x_2)=(1/2, 1/2)$ slice of $\QQ_G$ is precisely $-1/2$ times the GP polytope of $\HH$:
\[
\QQ_G \cap \left( \left\{ \frac{1}{2}, \frac{1}{2} \right\} \times \R^3 \right) 
= \left\{ \frac{1}{2}, \frac{1}{2} \right\} \times -\frac{1}{2} \P_\HH.
\]

Similarly, if we intersect $\QQ_G$ with $x_3 = x_4 = x_5 = -1/3$ (corresponding to the $z$-axis in $\R^3$), we obtain an interval with endpoints $\frac{1}{3} \{ (3,0), (1,2) \} \times (-1/3, -1/3, -1/3)$, so $\P_{\overline{\HH}}$ is also a slice of $\QQ_G$!
\[
\QQ_G \cap \left( \R^2 \times \left\{ -\frac{1}{3}, - \frac{1}{3}, - \frac{1}{3} \right\} \right) = \frac{1}{3} \P_{\overline{\HH}} \times \left\{ -\frac{1}{3}, - \frac{1}{3}, - \frac{1}{3} \right\}.
\]

The root polytopes of the spanning trees of $G$ are polytopes inside $\S_G$, which we can calculate:
\begin{align*}
\QQ_{T_1} &= \Conv \{ (1,0,-1,0,0), (1,0,0,-1,0), (1,0,0,0,-1), (0,1,0,-1,0) \} \\
\QQ_{T_2} &= \Conv \{ (1,0,-1,0,0), (1,0,0,-1,0), (1,0,0,0,-1), (0,1,0,0,-1) \} \\
\QQ_{T_3} &= \Conv \{ (1,0,-1,0,0), (1,0,0,-1,0), (0,1,0,-1,0), (0,1,0,0,-1) \} \\
\QQ_{T_4} &= \Conv \{ (1,0,-1,0,0), (1,0,0,0,-1), (0,1,0,-1,0), (0,1,0,0,-1) \}
\end{align*}
These polytopes are in fact \emph{tetrahedra} and the root polytope of $G$ \emph{decomposes} into pairs of these tetrahedra, as shown in figure \ref{fig:polytope_slices}. Thus we have \emph{triangulations} of $\QQ_G$:
\[
\QQ_G = \QQ_{T_1} \cup \QQ_{T_4} = \QQ_{T_2} \cup \QQ_{T_3}.
\]
Indeed, if we slice this triangulation through $x_1 = x_2 = 1/2$, we obtain the decompositions $\P_\HH = \P_{\TT_1} \cup \P_{\TT_4} = \P_{\TT_2} \cup \P_{\TT_3}$ found earlier. (Precisely, each $\QQ_{\TT_i} $ intersects the slice $\{\frac{1}{2}, \frac{1}{2} \} \times \R^3$ in $\{ \frac{1}{2}, \frac{1}{2} \} \times -\frac{1}{2} \P_{\TT_i}$.) And if we slice through $x_3 = x_4 = x_5 = -1/3$ we obtain corresponding decompositions of $\P_{\overline{\HH}}$.


\begin{figure}
\begin{center}
\includegraphics[scale=0.5]{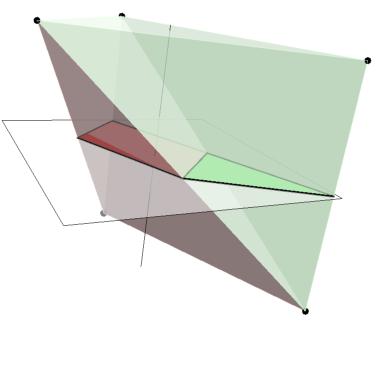}
\includegraphics[scale=0.5]{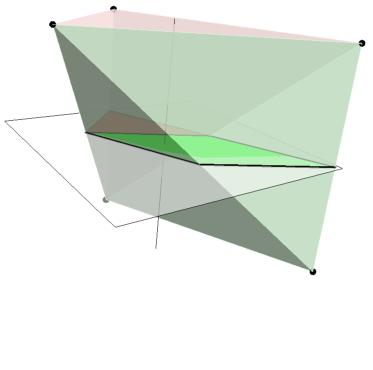}
\end{center}
\caption{The root polytope $\QQ_G$ is triangulated as $\QQ_{\TT_1} \cup \QQ_{\TT_4}$ (left) or $\QQ_{\TT_2} \cup \QQ_{\TT_3}$ (right). Slicing through $x_1 = x_2 = 1/2$ (corresponding to $z=0$ here) yields the decomposition of $\P_\HH$ as $\P_{\TT_1} \cup \P_{\TT_4}$ (left) or $\P_{\TT_2} \cup \P_{\TT_3}$ (right).}
\label{fig:polytope_slices}
\end{figure}

\subsection{Observations and statements}

In our example we observed some interesting phenomena, which of course are not coincidences.
\begin{itemize}
\item
Trimmed GP polytopes of hypergraphs coincide with hypertree polytopes of their abstract duals:
\[
\P^-_\HH = \S_{\overline{\HH}}
\quad \text{and} \quad
\P^-_{\overline{\HH}} = \S_\HH.
\]
\item
The slice of the root polytope $\QQ_G$ with equal coordinates $1/2 = 1/|U|$ on the $U$ factors (resp. $-1/3 = -1/|V|$ on the $V$ factors) is the scaled GP polytope $-\frac{1}{|U|} \P_\HH$ (resp. $\frac{1}{|V|} \P_{\overline{\HH}}$). In this sense, $\QQ_G$ provides a duality between the polytopes $\P_\HH$ and $\P_{\overline{\HH}}$.
\item
Root polytopes of spanning trees $\QQ_{\TT_i}$ form simplices in $\QQ_G$ and certain collections of spanning trees yield triangulations of $\QQ_G$. Slicing such a triangulation by setting $U$ coordinates to $1/|U|$ (resp. $V$ coordinates to $-1/|V|$) yields a decomposition of $\P_\HH$ into GP polytopes $\P_{\TT_i}$ (resp. $\P_{\overline{\HH}}$ into $\P_{\overline{\TT_i}}$). So $\QQ_G$ also provides a duality between decompositions of $\P_\HH$ and $\P_{\overline{\HH}}$.
\item
Both $\HH$ and $\overline{\HH}$ have the same number of hypertrees, i.e. $S_\HH = S_{\overline{\HH}}$.
\end{itemize}

We now give some brief sketches of the proofs. For full details we refer to \cite{Postnikov09} or \cite{Kalman13_Tutte}. We begin by considering polytopes inside the root polytope $\QQ_G$.
\begin{lem}
\label{lem:subgraph_polytope_isomorphism}
The set of subgraphs of $G$, partially ordered by inclusion, and the set of polytopes whose vertices are vertices of $\QQ_G$, partially ordered by inclusion, are isomorphic posets.
\end{lem}

\begin{proof}
We noted above that if $F$ is a subgraph of $G$ then $\QQ_F \subseteq \QQ_G$ is a polytope whose vertices are taken from $\QQ_G$. Conversely, a polytope $P$ with vertices from $\QQ_G$ has every vertex of the form ${\bf i}_u - {\bf i}_v$, where $(u,v)$ is an edge of $G$. Thus the vertices of $P$ corresponds to edges of $G$, and $P$ is the root polytope of the subgraph consisting of these edges.
\end{proof}

We now verify that a spanning tree $T$ yields a tetrahedron $\QQ_T \subseteq \QQ_G$, and more. This is lemma 12.5 of \cite{Postnikov09}.
\begin{prop}
Let $F$ be a subgraph of $G$, with root polytope $\QQ_F \subseteq \QQ_G$.
\begin{enumerate}
\item
$\QQ_F$ is a simplex if and only if $F$ is a forest.
\item
$\QQ_F$ is a simplex of the same dimension as $\QQ_G$ if and only if $F$ is a spanning tree of $G$.
\end{enumerate}
\end{prop}

\begin{proof}
It is not difficult to see that $F$ contains a cycle if and only if the vectors ${\bf i}_u - {\bf i}_v$, where $(u,v)$ is an edge of $F$, are linearly dependent. Thus $\QQ_F$ has linearly independent vertices, i.e. is a simplex, precisely when $F$ is a forest. The dimension of this simplex is maximal when $F$ has as many edges as possible, i.e. is a spanning tree.
\end{proof}

Returning to the GP polytope $\P_\HH$, the decompositions of $\P_\HH$ found in our example were decompositions of a specific type.
\begin{defn}
Let $P = P_1 + \cdots + P_m$ be a Minkowski sum of polytopes (hence a polytope). 
\begin{enumerate}
\item
A \emph{Minkowski cell} of $P$ is a polytope of the form $B_1 + \cdots + B_m$, of the same dimension as $P$, where each $B_i$ is the convex hull of some of the vertices of $P_i$.
\item
A \emph{mixed subdivision} of $P$ is a decomposition of $P$ as a union of Minkowksi cells, such that the intersection of any two cells, if nonempty, is a common face.
\item
A mixed subdivision of $P$ is \emph{fine} if it cannot be decomposed further as a mixed subdivision of $P$.
\end{enumerate}
\end{defn}
Note that a Minkowski cell of $P$ is always a polytope contained inside $P$, but need not be a simplex; similarly, a fine mixed subdivision need not be a triangulation. The decompositions of $\P_\HH$ and $\P_{\overline{\HH}}$ found in our example are fine mixed subdivisions.

The key idea in relating $\P_\HH$ and $\QQ_G$ is known as the \emph{Cayley trick} \cite{Huber-Rambau-Birkett00, Santos05, Sturmfels94}. Given polytopes $P_1, \ldots, P_m$ in $\R^n$, the Cayley trick relates decompositions of the Minkowski sum $P_1 + \cdots + P_m \subset \R^n$, to decompositions of the polytope
\[
\mathcal{C} (P_1, \ldots, P_m) = \Conv \{ {\bf i}_j - P_j \} \subset \R^m \times \R^n.
\]
Here ${\bf i}_j$ is a vector in $\R^m$, and $P_j \subset \R^n$. The trick is to consider the slice of $\mathcal{C} ( \Delta_{e(1)}, \ldots, \Delta_{e(m)} )$ along $x_1 = \cdots = x_m = \frac{1}{m}$ (where $x_1, \ldots, x_m$ are coordinates on $\R^m$). We find that the result is $-\frac{1}{m} (P_1 + \cdots + P_m) \subset \R^n$.

Observe that if we take $P_1, \ldots, P_m$ as the simplices $\Delta_{e(1)}, \ldots, \Delta_{e(m)}$ (i.e. $\Delta_{e(u)}$ over $u \in U$) in a hypergraph $\HH$, then
\[
\mathcal{C} ( \Delta_{e(1)}, \ldots, \Delta_{e(m)} ) = \QQ_G
\quad \text{and} \quad
\Delta_{e(1)} + \cdots + \Delta_{e(m)} = \P_\HH.
\]
Stated precisely in this context, the Cayley trick is as follows. This statement is taken from \cite[thm. 1.4]{Santos05}; see also 
\cite[prop. 14.5]{Postnikov09} and \cite[sec. 3]{Huber-Rambau-Birkett00}. 
\begin{thm}[Cayley trick] \
\label{thm:cayley_trick}
\begin{enumerate}
\item
The intersection of $\QQ_G$ with 
\begin{enumerate}
\item
$\{ \frac{1}{m}, \ldots, \frac{1}{m} \} \times \R^n$ 
is $\{ \frac{1}{m}, \ldots, \frac{1}{m} \} \times -\frac{1}{m} \P_\HH$;
\item
$\R^m \times \{ -\frac{1}{n}, \ldots, -\frac{1}{n} \}$ 
is $\frac{1}{n} \P_{\overline{\HH}} \times \{ -\frac{1}{n}, \ldots, -\frac{1}{n} \}$.
\end{enumerate}
\item
For any polyhedral subdivision of $\QQ_G$, its intersection with
\begin{enumerate}
\item $\{ \frac{1}{m}, \ldots, \frac{1}{m} \} \times \R^n$ yields a mixed subdivision of $-\frac{1}{m} \P_\HH$;
\item
$\R^m \times \{ -\frac{1}{n}, \ldots, -\frac{1}{n} \}$ yields a mixed subdivision of $\frac{1}{n} \P_{\overline{\HH}}$.
\end{enumerate}
\item
The map from polyhedral subdivisions of $\QQ_G$ to mixed subdivisions of $\P_\HH$ (or $\P_{\overline{\HH}}$) is an isomorphism of partially ordered sets, where subdivisions are partially ordered by refinement.
\item
This map restricts to a bijection between triangulations of $\QQ_G$ and fine mixed subdivisions of $\P_\HH$ (or $\P_{\overline{\HH}}$).
\end{enumerate}
\end{thm}

This statement makes precise the idea that the root polytope $\QQ_G$ provides a \emph{duality} between the GP polytopes $\P_\HH$ and $\P_{\overline{\HH}}$.

\begin{proof}[Proof sketch]
The proof of the first part is not difficult; the calculation in our example woks in the general case. For the second part, note by lemma
\ref{lem:subgraph_polytope_isomorphism}
that a polyhedron in a polyhedral subdivision of $\QQ_G$ is of the form $\QQ_F$ for some subgraph $F$ of $G$, corresponding to hypergraphs $\mathcal{F}, \overline{\mathcal{F}}$, and the desired intersections are just $-\frac{1}{m} \P_{\mathcal{F}}$ and $\frac{1}{n} \P_{\overline{\mathcal{F}}}$, which are Minkowski cells. To see that a mixed subdivision of $\P_\HH$ arises from a unique polyhedral subdivision of $\QQ_G$, one can consider scalings of Minkowski cells and regard these as points in $\R^m \times \R^n$. The map clearly preserves refinements, and the finest polyhedral subdivisions are triangulations.
\end{proof}

In fact, one can show that a mixed subdivision of $\P_\HH$ is fine if and only if in each cell $B = B_1 + \cdots + B_m$ of the subdivision, each $B_i$ is a simplex, and $\sum \dim B_i$ is the dimension of $\P_\HH$ \cite[lem. 14.2]{Postnikov09}, \cite[prop. 2.3]{Santos05}. Letting $B_i = \Delta_{T(i)}$ we obtain $B$ as the GP polytope of the hypergraph $\TT$ with hyperedges $T(i)$; the condition on dimensions means that $\TT$ is a spanning tree, corresponding to a simplex $\QQ_\TT$ in $\QQ_G$. In particular, a spanning tree $T$ yields a simplex of maximal dimension in $\QQ_G$, whose intersection with $\{1/m, \ldots, 1/m\} \times \R^n$ is a Minkowski cell of $\P_\HH$.

The key to the relationship between the hypertree polytope and the trimmed GP polytope is the following lemma. This is lemma 14.9 of \cite{Postnikov09}. Let $T$ be a spanning tree of $G$, with corresponding spanning trees $\TT, \overline{\TT}$ of $\HH$ and $\overline{\HH}$ so that $\P_\TT \subseteq \P_\HH$ is a Minkowski cell.  Recall the hypertree $f_{\overline{\TT}}$ is defined as $\deg_V T - (1,1, \ldots, 1) \in \R^V$, and the notation $\Delta_V = \Conv\{ {\bf i}_{v} \; \mid \; v \in V \}$.
\begin{lem} \
\begin{enumerate}
\item
$\Delta_V + f_{\overline{\TT}} \subseteq \P_\TT$. That is, the polytope $\P_\TT$ contains a copy of the top-dimensional simplex $\Delta_V$, translated by the hypertree $f_{\overline{\TT}} \in \R^n$ of $\overline{\TT}$.
\item
For any integer vector $(a_1, \ldots, a_n) \in \Z^V$ other than $f_{\overline{\TT}}$, the translated simplex $(a_1, \ldots, a_n) + \Delta_V$ shares no interior points with $\P_\TT$.
\end{enumerate}
\end{lem}
Thus, if we consider integer shifts of $\Delta_V$, only one such shift lies inside $\P_\TT$, namely the shift by $f_{\overline{\TT}} = \deg_V T - (1, \ldots, 1)$; all other shifts have no (interior) overlap with $\P_\TT$.

\begin{proof}
Recall $\P_\HH = \Delta_{T(1)} + \cdots + \Delta_{T(m)}$, where $T(u) = \{v \mid (u,v) \text{ is an edge in } G \}$. Generalising our discussion of $T_3$ and $T_1$ in our example, we will successively adjust $T$ to obtain a sequence of trees (not necessarily spanning trees of $G$), all with the same $V$-degree vectors, and with their GP polytopes shrinking at each stage, hence always contained in the original $\P_\TT$. 

As $T$ is connected, there exist $u_1 < u_2$ such that $T(u_1) \cap T(u_2)$ is nonempty; indeed, as $T$ is a tree, we have $T(u_1) \cap T(u_2) = \{v\}$ for some $v \in V$. We remove all edges of the form $(u_2, v')$ from $T$, for $v' \neq v$, and add corresponding edges $(u_1, v')$. In this way $V$-degree vectors are preserved, and in the corresponding polytope we replace $\Delta_{T(u_1)} + \Delta_{T(u_2)}$ with $\Delta_{v} + \Delta_{T(u_1) \cup T(u_2)}$. Since in general $\Delta_X + \Delta_Y \supseteq \Delta_{X \cap Y} + \Delta_{X \cup Y}$, the polytope shrinks.

Repeating this process, we can eventually arrange that $T(1) = V$ (i.e. the vertex $1 \in U$ is connected to every vertex in $V$), and every other $T(u)$ is a singleton. The multiset union of these singletons precisely describes $\deg_V T - (1,\ldots,1)$. Thus we obtain a spanning tree $T'$ of $G$ (and $\TT'$ of $\HH$) such that $\deg_V T' = \deg_V T$, with hypertree $f_{\TT'} = f_{\TT}$ and corresponding polytope $\P_{\TT'} = \Delta_V + \deg_V T - (1,\ldots,1) = \Delta_V + f_{\TT'}$. Since $\P_{\TT}' \subseteq \P_{\TT}$, the first claim follows.

If there is another translation of $\Delta_V$ which overlaps with $\P_\TT$ in its interior, then we have an interval in the interior of $\P_\TT$ whose endpoints differ by an element of $\Z^n$. But as $T$ is a tree, $\P_\TT$ is a Minkowski sum of independent simplices, and hence can be regarded as a product of simplices; projecting to one of these factors, we obtain an interval in the interior of a standard simplex, whose endpoints differ by an integer vector; this is a contradiction.
\end{proof}

Now $\Delta_V + f_{\overline{\TT}} \subseteq \P_\TT \subseteq \P_\HH$ implies by definition that $f_{\overline{\TT}} \in \P_\HH - \Delta_V = \P^-_\HH$. As this is true for all hypertrees $\overline{\TT}$ of $\overline{\HH}$, the hypertree polytope of $\overline{\HH}$ must be contained in the trimmed GP polytope, i.e. $\S_{\overline{\HH}} \subseteq \P^-_\HH$. On the other hand, if we take a point $(a_1, \ldots, a_n)$ in $\Z^n$ which is not a hypertree of any spanning tree of $\overline{\HH}$, then the above lemma says that $(a_1, \ldots, a_n) + \Delta_V$ shares no interior points with $\P_\TT$, for any spanning tree $\TT$ of $\HH$. But $\P_\HH$ has decompositions into Minkowski cells of the form $\P_\TT$, so such an $(a_1, \ldots, a_n) + \Delta_V$ must share no interior points with the entire GP polytope $\P_\HH$. In particular, $(a_1, \ldots, a_n)$ is not an element of the trimmed polytope $\P_\HH^-$. Hence the sets of lattice points $P^-_\HH$ and $S_{\overline{\HH}}$ coincide.

Moreover, if we take a fine mixed subdivision of $\P_\HH$ into $\bigcup_{i} \P_{\TT_i}$, for some spanning trees $\TT_i$, then any translate $(a_1, \ldots, a_n) + \Delta_V$ which lies in $\P_\HH$ must in fact lie in one of the $\P_{\TT_i}$. So the full set of hypertrees $S_{\overline{\HH}}$ must coincide with the set of hypertrees arising from the $\TT_i$.

Finally, as both $\S_{\overline{\HH}}$ and $\P_\HH^-$ have vertices with integer coordinates, they are the convex hulls of their lattice points, hence coincide. So we have proved $\S_{\overline{\HH}} = \P^-_\HH$, and more; compare \cite[thm. 12.9]{Postnikov09}.
\begin{thm} 
\label{thm:trimmed_GP_spanning_tree_polytopes}
Let $T_1, \ldots, T_s$ be a set of spanning trees such that $\P_{\overline{\TT_i}}$ form a fine mixed subdivision of $\P_{\overline{\HH}}$, or equivalently, such that $\QQ_{\TT_i}$ form a triangulation of $\QQ_G$.

Then the hypertrees $f_{\overline{\TT_1}}, \ldots, f_{\overline{\TT_s}}$ coincide precisely with the hypertrees $S_{\overline{\HH}}$ of $\overline{\HH}$, which in turn coincide with the set of lattice points $P^-_\HH$ of the trimmed GP polytope of $\HH$. Moreover 
\[
\S_{\overline{\HH}} = \P^-_\HH.
\]
\qed
\end{thm}
Applying duality to the entire argument a corresponding statement on the abstract duals; in particular $\S_\HH = \P^-_{\overline{\HH}}$.

It is not too difficult to show that if we have a triangulation $\QQ_{T_1}, \ldots, \QQ_{T_s}$ of $\QQ_G$ arising from spanning trees $T_1, \ldots, T_s$, then all the $T_i$ have distinct hypertrees in $\HH$ and $\overline{\HH}$ \cite[lem. 12.8]{Postnikov09}. Hence the equality of sets $\{ f_{\overline{\TT_1}}, \ldots, f_{\overline{\TT_s}} \} = S_{\overline{\HH}}$ in the above theorem involves no repeated elements. So the number of hypertrees in $\overline{\HH}$ is equal to the number of simplices arising in any triangulation of $\QQ_G$. Applying the same argument, the number of hypertrees in$\HH$ must also coincide with the number of simplices triangulating $\QQ_G$, and we have the following.
\begin{thm}
\label{thm:hypertrees_in_abstract_duals}
Any hypergraph $\HH$ and its abstract dual $\overline{\HH}$ have the same number of hypertrees: $|S_\HH | = | S_{\overline{\HH}} |$.
\qed
\end{thm}

Indeed, in our example, we found only two distinct hypertrees in $\HH$ and in $\overline{\HH}$; and these two hypertrees were realised in the decompositions seen of $\QQ_G$, $\P_\HH$ and $\P_{\overline{\HH}}$.

In general the hypertree polytopes $\S_\HH$ and $\S_{\overline{\HH}}$ may look very different; it is perhaps surprising that they should contain the same number of lattice points.

\section{Plane graphs, dualities and trinities}
\label{sec:plane_graphs_dualities_trinities}

\subsection{Plane graphs and planar duals}
\label{sec:planar_duals}

We now consider \emph{plane} graphs. For us a plane graph is a graph embedded in $\R^2$, up to isotopy; adding a point at infinity, we can also regard the graph as embedded in $S^2$. 

A plane graph $G$ has its own form of duality, distinct from abstract duality of hypergraphs: a \emph{planar dual} $G^*$. We place one vertex in each complementary region of $G$ in the plane; the set $R$ of these vertices is the vertex set of $G^*$. Each edge $e$ of $G$ has complementary regions on either side (possibly the same region), with corresponding vertices $r_1, r_2 \in R$ (possibly $r_1 = r_2$), so we may draw an edge $e^*$ from $r_1$ to $r_2$ through $e$. These edges $e^*$ form the edges of the planar dual $G^*$. So $G^*$ is a plane graph with vertices in bijection with the complementary regions of $G$, and edges in bijection with the edges of $G$. Observe that $G$ can be obtained from $G^*$ in the same way; $G$ and $G^*$ are planar duals of each other.

\subsection{Plane bipartite graphs and trinities}
\label{sec:plane_bipartite_graphs_trinities}

In this story we are concerned with plane graphs which are also bipartite. Plane bipartite graphs have notions of both \emph{abstract} and \emph{planar} duality, and the interaction of these distinct types of duality gives rise to interesting structure.

So, let $G$ be a plane bipartite graph with vertex classes $V$ and $E$. Let $R$ be a set of vertices, one in each complementary region of $G$. We can think of these three vertex classes as distinct \emph{colours}: the vertices in $V,E,R$ are called \emph{violet}, \emph{emerald} and \emph{red} respectively.

Each complementary region of $G$ then contains a single red vertex, and the boundary of this region consists of violet and emerald vertices, together with edges of $G$. Since $G$ is bipartite, the violet and emerald vertices alternate around the boundary of the region. We can draw an edge from each red vertex to all the violet and emerald vertices around the boundary of the corresponding region, so as to obtain a \emph{triangulation} of $S^2$ --- that is, a graph embedded in $S^2$ for which every complementary region is a triangle. Moreover, each triangle in the triangulation has vertices of three distinct colours.
\begin{defn}
A triangulation of $S^2$ with each vertex coloured violet, emerald or red, such that each triangle has one vertex of each colour, is called a \emph{trinity}.
\end{defn}
Since each edge of the triangulation has vertices of distinct colours, and we can colour each edge to have the unique colour distinct from its endpoints. Moreover, each triangle has a violet, emerald and red vertex, labelled in either a clockwise or anticlockwise order; and triangles sharing an edge have their vertices in opposite orders. Thus the triangles come in two types. Accordingly, we colour a triangle black if the order is clockwise, and white if the order is anticlockwise; any two triangles sharing an edge have different colours. For the bipartite graph of figure \ref{fig:graph_G}, the corresponding trinity is shown in figure \ref{fig:easy_trinity}. See figure \ref{fig:trinity} for another example.

We have seen that any bipartite plane graph yields a trinity; and conversely, by taking the violet and emerald vertices, together with the red edges joining them, a trinity yields a bipartite plane graph. But we could equally take the emerald and red vertices; or the red and violet vertices, and obtain bipartite plane graphs. So a trinity $\Tr$ naturally contains \emph{three} bipartite plane graphs. We define the \emph{violet graph} $G_V$ to have the violet edges, and vertex classes $(E,R)$; the \emph{emerald graph} $G_E$ to have emerald edges and vertex classes $(R,V)$; and the \emph{red graph} $G_R$ to have red edges and vertex classes $(V,E)$. (The red graph is our original graph, $G_R = G$.) 

In other words, bipartite plane graphs naturally come in threes, and have a natural form, not of duality, but of \emph{triality}. Figures \ref{fig:easy_triple_of_graphs} and \ref{fig:trinity} show triples of bipartite plane graphs associated to trinities.

\begin{figure}
\begin{center}
\begin{tikzpicture}[scale = 1.5]
\coordinate (e1) at (0,0);
\coordinate (e2) at (0,-2);
\coordinate (v1) at (0,1);
\coordinate (v2) at (1,-1);
\coordinate (v3) at (-1,-1);
\coordinate (r1) at (0,-1);
\coordinate (r2) at (0,2);

\fill [black!10!white] (r2) to [bend right=60] (e1) -- (v1) -- cycle;
\fill [black!10!white] (r2) .. controls (-3,2) and (-3,-2) .. (e2) -- (v3) to [bend left=60] (r2);
\fill [black!10!white] (r1) -- (e1) -- (v3) -- cycle;
\fill [black!10!white] (r1) -- (e2) -- (v2) -- cycle;
\fill [black!10!white] (r2) to [bend left=60] (e1) -- (v2) to [bend right=60] (r2);

\draw [ultra thick, red] (e1) -- (v1);
\draw [ultra thick, red] (e1) -- (v2);
\draw [ultra thick, red] (e1) -- (v3);
\draw [ultra thick, red] (e2) -- (v2);
\draw [ultra thick, red] (e2) -- (v3);

\draw [ultra thick, blue] (r1) -- (e1);
\draw [ultra thick, blue] (r1) -- (e2);
\draw [ultra thick, blue] (r2) to [bend right=60] (e1); 
\draw [ultra thick, blue] (r2) to [bend left=60] (e1);
\draw [ultra thick, blue] (r2) .. controls (-3,2) and (-3,-2) .. (e2);

\draw [ultra thick, green!50!black] (r1) -- (v2);
\draw [ultra thick, green!50!black] (r1) -- (v3);
\draw [ultra thick, green!50!black] (r2) -- (v1);
\draw [ultra thick, green!50!black] (r2) to [bend left=60] (v2);
\draw [ultra thick, green!50!black] (r2) to [bend right=60] (v3);

\foreach \x in {(e1), (e2)}
{
\draw [green!50!black, fill=green!50!black] \x circle  (3pt);
}

\foreach \x in {(v1), (v2), (v3)}
{
\draw [blue, fill=blue] \x circle  (3pt);
}

\foreach \x in {(r1), (r2)}
{
\draw [red, fill=red] \x circle  (3pt);
}

\end{tikzpicture}

\begin{tikzpicture}[scale = 1]
\coordinate (e1) at (0,0);
\coordinate (e2) at (0,-2);
\coordinate (v1) at (0,1);
\coordinate (v2) at (1,-1);
\coordinate (v3) at (-1,-1);
\coordinate (r1) at (0,-1);
\coordinate (r2) at (0,2);

\draw [ultra thick, blue] (r1) -- (e1);
\draw [ultra thick, blue] (r1) -- (e2);
\draw [ultra thick, blue] (r2) to [bend right=60] (e1); 
\draw [ultra thick, blue] (r2) to [bend left=60] (e1);
\draw [ultra thick, blue] (r2) .. controls (-3,2) and (-3,-2) .. (e2);

\foreach \x in {(e1), (e2)}
{
\draw [green!50!black, fill=green!50!black] \x circle  (3pt);
}

\foreach \x in {(v1), (v2), (v3)}
{
\draw [blue, fill=blue] \x circle  (3pt);
}

\foreach \x in {(r1), (r2)}
{
\draw [red, fill=red] \x circle  (3pt);
}

\end{tikzpicture}
\hspace{2 cm}
\begin{tikzpicture}[scale = 1]
\coordinate (e1) at (0,0);
\coordinate (e2) at (0,-2);
\coordinate (v1) at (0,1);
\coordinate (v2) at (1,-1);
\coordinate (v3) at (-1,-1);
\coordinate (r1) at (0,-1);
\coordinate (r2) at (0,2);

\draw [ultra thick, green!50!black] (r1) -- (v2);
\draw [ultra thick, green!50!black] (r1) -- (v3);
\draw [ultra thick, green!50!black] (r2) -- (v1);
\draw [ultra thick, green!50!black] (r2) to [bend left=60] (v2);
\draw [ultra thick, green!50!black] (r2) to [bend right=60] (v3);

\foreach \x in {(e1), (e2)}
{
\draw [green!50!black, fill=green!50!black] \x circle  (3pt);
}

\foreach \x in {(v1), (v2), (v3)}
{
\draw [blue, fill=blue] \x circle  (3pt);
}

\foreach \x in {(r1), (r2)}
{
\draw [red, fill=red] \x circle  (3pt);
}

\end{tikzpicture}
\hspace{2 cm}
\begin{tikzpicture}[scale = 1]
\coordinate (e1) at (0,0);
\coordinate (e2) at (0,-2);
\coordinate (v1) at (0,1);
\coordinate (v2) at (1,-1);
\coordinate (v3) at (-1,-1);
\coordinate (r1) at (0,-1);
\coordinate (r2) at (0,2);

\draw [ultra thick, red] (e1) -- (v1);
\draw [ultra thick, red] (e1) -- (v2);
\draw [ultra thick, red] (e1) -- (v3);
\draw [ultra thick, red] (e2) -- (v2);
\draw [ultra thick, red] (e2) -- (v3);

\foreach \x in {(e1), (e2)}
{
\draw [green!50!black, fill=green!50!black] \x circle  (3pt);
}

\foreach \x in {(v1), (v2), (v3)}
{
\draw [blue, fill=blue] \x circle  (3pt);
}

\foreach \x in {(r1), (r2)}
{
\draw [red, fill=red] \x circle  (3pt);
}

\end{tikzpicture}
\end{center}
\caption{The trinity derived from the bipartite plane graph from figure \ref{fig:graph_G} (with red vertices added in complementary regions), and the corresponding triple of bipartite plane graphs (including the original graph $G=G_R$).}
\label{fig:easy_triple_of_graphs}
\label{fig:easy_trinity}
\end{figure}
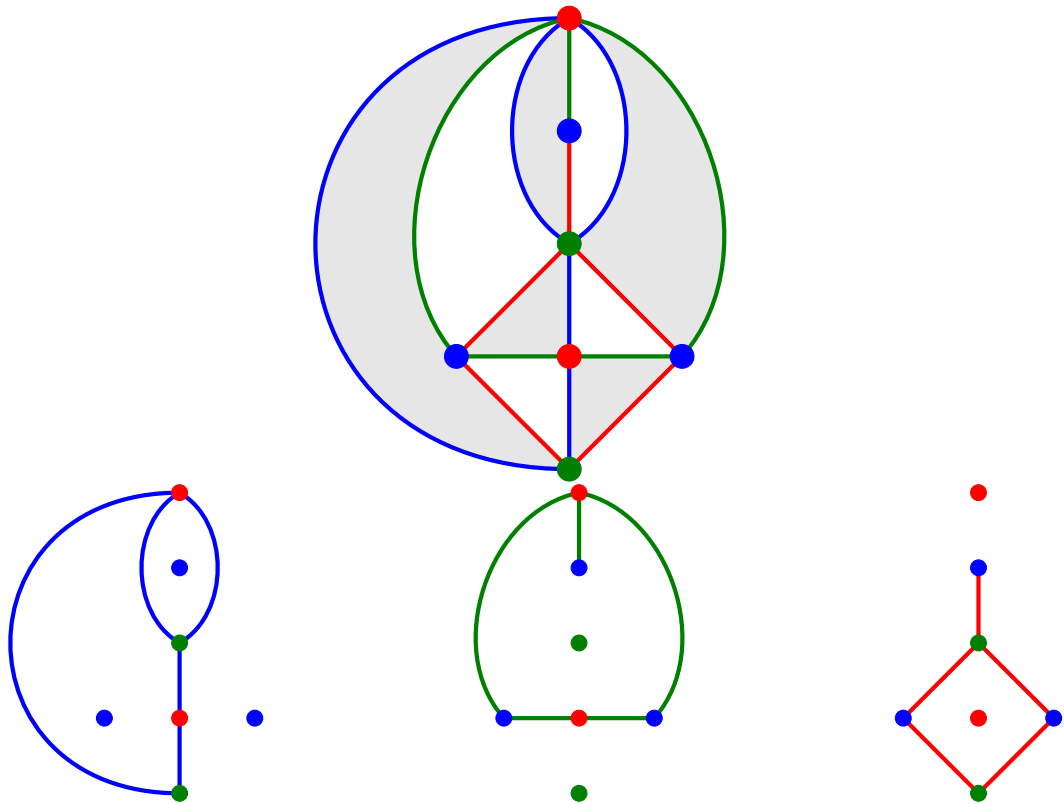

\begin{figure}
\begin{center}
\begin{tikzpicture}[scale = 0.25]
\coordinate (r3) at (1,6);
\coordinate (e1) at (-1,4); 
\coordinate (v1) at (-1.5, 2);
\coordinate (e2) at (-2,0); 
\coordinate (v2) at (-2.3,-2); 
\coordinate (e3) at (-2.5,-4); 
\coordinate (r0) at (-3,-7); 
\coordinate (v4) at (3,4);
\coordinate (r2) at (2.5,0);
\coordinate (v3) at (4,-3);
\coordinate (e0) at (7,3);
\coordinate (v0) at (-6,6); 
\coordinate (r1) at (-6, 0); 

\fill [black!10!white] (e0)
.. controls ($ (e0) + (60:7) $) and ($ (v0) + (60:7) $) .. (v0)
to [out=30, in=120] (r3)
to [out=30, in=120] (e0);

\fill [black!10!white] (r3)
to[out=210,in=70] (e1)
to [out=10, in=180] (v4)
to [out=120, in=300] (r3);

\fill [black!10!white] (v0)
.. controls ($ (v0) + (240:3) $) and ($ (r1) + (150:3) $) .. (r1)
to [out=90, in=190] (e1)
to [out=150, in=330] (v0);

\fill [black!10!white] (e1)
to [out = 250, in = 70] (v1)
to [out=0, in=120] (r2)
to [out=90, in=330] (e1);

\fill [black!10!white] (v1)
to (e2)
to [out=180, in=330] (r1)
to [out=60, in=180] (v1);

\fill [black!10!white] (e2)
to (v2)
to [out=350, in=210] (r2)
to [out=180, in=0] (e2);

\fill [black!10!white] (v2)
to (e3)
.. controls ($ (e3) + (120:2) $)  and ($ (r1) + (240:5) $) .. (r1)
to [out=270, in=170] (v2);

\fill [black!10!white] (v0) 
.. controls ($ (v0) + (210:8) $) and ($ (e3) + (200:8) $) .. (e3)
to [out=270, in=90] (r0)
.. controls ($ (r0) + (180:15) $) and ($ (v0) + (150:6) $) .. (v0);

\fill [black!10!white] (v4)
to [out=0, in=165] (e0)
to [out=240, in=0] (r2)
to [out=30, in=300] (v4);

\fill [black!10!white] (v3) 
to [out=90, in=300] (r2)
to [out=270, in=20] (e3)
to [out=300, in=210] (v3);

\fill [black!10!white] (v3) 
to [out=30, in=300] (e0)
.. controls ($ (e0) + (345:10) $)  and ($ (r0) + (270:8) $) .. (r0)
to [out=0, in=270] (v3);

\draw [ultra thick, red] (e1) to [out = 250, in = 70] (v1); 
\draw [ultra thick, red] (v1) to (e2); 
\draw [ultra thick, red] (e2) to (v2); 
\draw [ultra thick, red] (v2) to (e3); 
\draw [ultra thick, red] (e0) to [out=165, in=0] (v4); 
\draw [ultra thick, red] (v4) to [out=180, in=10] (e1); 
\draw [ultra thick, red] (v0)
.. controls ($ (v0) + (210:8) $) and ($ (e3) + (200:8) $) .. (e3); 
\draw [ultra thick, red] (e1) to [out=150, in=330] (v0); 
\draw [ultra thick, red] (e3) to [out=300, in=210] (v3); 
\draw [ultra thick, red] (v3) to [out=30, in=300] (e0); 
\draw [ultra thick, red] (e0) 
.. controls ($ (e0) + (60:7) $) and ($ (v0) + (60:7) $) .. (v0); 

\draw [ultra thick, blue] (e1) to [out=190, in=90] (r1); 
\draw [ultra thick, green!50!black] (r1) to [out=270, in=170] (v2); 
\draw [ultra thick, green!50!black] (v1) to [out=180, in=60] (r1); 
\draw [ultra thick, blue] (r1)
.. controls ($ (r1) + (240:5) $) and ($ (e3) + (120:2) $) .. (e3); 
\draw [ultra thick, green!50!black] (v0)
.. controls ($ (v0) + (240:3) $) and ($ (r1) + (150:3) $) .. (r1); 
\draw [ultra thick, blue] (r1) to [out=330, in=180] (e2); 

\draw [ultra thick, green!50!black] (v2) to [out=350, in=210] (r2); 
\draw [ultra thick, green!50!black] (r2) to [out=30, in=300] (v4); 
\draw [ultra thick, blue] (e3) to [out=20, in=270] (r2); 
\draw [ultra thick, blue] (r2) to [out=90, in=330] (e1); 
\draw [ultra thick, green!50!black] (v3) to [out=90, in=300] (r2); 
\draw [ultra thick, green!50!black] (r2) to [out=120, in=0] (v1); 
\draw [ultra thick, blue] (e2) to [out=0, in=180] (r2); 
\draw [ultra thick, blue] (r2) to [out=0, in=240] (e0); 

\draw [ultra thick, blue] (r3) to[out=210,in=70] (e1); 
\draw [ultra thick, green!50!black] (v4) to [out=120, in=300] (r3); 
\draw [ultra thick, green!50!black] (r3) to [out=120, in=30] (v0); 
\draw [ultra thick, blue] (e0) to [out=120, in=30] (r3); 

\draw [ultra thick, draw=none] (e3) to [out=270, in=90] (r0); 
\draw [ultra thick, draw=none] (r0)
.. controls ($ (r0) + (270:8) $) and ($ (e0) + (345:10) $) .. (e0); 
\draw [ultra thick, draw=none] (v0)
.. controls ($ (v0) + (150:6) $) and ($ (r0) + (180:15) $) .. (r0); 
\draw [ultra thick, draw=none] (r0) to [out=0, in=270] (v3); 

\draw [ultra thick, blue] (e3) to [out=270, in=90] (r0); 
\draw [ultra thick, blue] (r0)
.. controls ($ (r0) + (270:8) $) and ($ (e0) + (345:10) $) .. (e0); 
\draw [ultra thick, green!50!black] (v0)
.. controls ($ (v0) + (150:6) $) and ($ (r0) + (180:15) $) .. (r0); 
\draw [ultra thick, green!50!black] (r0) to [out=0, in=270] (v3); 

\foreach \x/\word in {(r0)/r0, (r1)/r1, (r2)/r2, (r3)/r3}
{
\draw [red, fill=red] \x circle  (10pt);
}

\foreach \x/\word in {(e0)/e0, (e1)/e1, (e2)/e2, (e3)/e3}
{
\draw [green!50!black, fill=green!50!black] \x circle  (10pt);
}

\foreach \x/\word in {(v0)/v0, (v1)/v1, (v2)/v2, (v3)/v3, (v4)/v4}
{
\draw [blue, fill=blue] \x circle (10pt);
}

\end{tikzpicture}

\begin{tikzpicture}[scale = 0.2]
\coordinate (r3) at (1,6);
\coordinate (e1) at (-1,4); 
\coordinate (v1) at (-1.5, 2);
\coordinate (e2) at (-2,0); 
\coordinate (v2) at (-2.3,-2); 
\coordinate (e3) at (-2.5,-4); 
\coordinate (r0) at (-3,-7); 
\coordinate (v4) at (3,4);
\coordinate (r2) at (2.5,0);
\coordinate (v3) at (4,-3);
\coordinate (e0) at (7,3);
\coordinate (v0) at (-6,6); 
\coordinate (r1) at (-6, 0); 

\draw [ultra thick, blue] (r3) to[out=210,in=70] (e1); 
\draw [ultra thick, blue] (e3) to [out=270, in=90] (r0); 
\draw [ultra thick, blue] (r0)
.. controls ($ (r0) + (270:8) $) and ($ (e0) + (345:10) $) .. (e0); 
\draw [ultra thick, blue] (e1) to [out=190, in=90] (r1); 
\draw [ultra thick, blue] (e3) to [out=20, in=270] (r2); 
\draw [ultra thick, blue] (r2) to [out=90, in=330] (e1); 
\draw [ultra thick, blue] (r1)
.. controls ($ (r1) + (240:5) $) and ($ (e3) + (120:2) $) .. (e3); 
\draw [ultra thick, blue] (e0) to [out=120, in=30] (r3); 

\draw [ultra thick, blue] (r1) to [out=330, in=180] (e2); 
\draw [ultra thick, blue] (e2) to [out=0, in=180] (r2); 
\draw [ultra thick, blue] (r2) to [out=0, in=240] (e0); 

\foreach \x/\word in {(r0)/r0, (r1)/r1, (r2)/r2, (r3)/r3}
{
\draw [red, fill=red] \x circle  (10pt);
}

\foreach \x/\word in {(e0)/e0, (e1)/e1, (e2)/e2, (e3)/e3}
{
\draw [green!50!black, fill=green!50!black] \x circle  (10pt);
}

\foreach \x/\word in {(v0)/v0, (v1)/v1, (v2)/v2, (v3)/v3, (v4)/v4}
{
\draw [blue, fill=blue] \x circle (10pt);
}

\end{tikzpicture}
\begin{tikzpicture}[scale = 0.2]
\coordinate (r3) at (1,6);
\coordinate (e1) at (-1,4); 
\coordinate (v1) at (-1.5, 2);
\coordinate (e2) at (-2,0); 
\coordinate (v2) at (-2.3,-2); 
\coordinate (e3) at (-2.5,-4); 
\coordinate (r0) at (-3,-7); 
\coordinate (v4) at (3,4);
\coordinate (r2) at (2.5,0);
\coordinate (v3) at (4,-3);
\coordinate (e0) at (7,3);
\coordinate (v0) at (-6,6); 
\coordinate (r1) at (-6, 0); 

\draw [ultra thick, green!50!black] (r1) to [out=270, in=170] (v2); 
\draw [ultra thick, green!50!black] (v2) to [out=350, in=210] (r2); 
\draw [ultra thick, green!50!black] (r2) to [out=30, in=300] (v4); 
\draw [ultra thick, green!50!black] (v4) to [out=120, in=300] (r3); 
\draw [ultra thick, green!50!black] (r3) to [out=120, in=30] (v0); 
\draw [ultra thick, green!50!black] (v0)
.. controls ($ (v0) + (150:6) $) and ($ (r0) + (180:15) $) .. (r0); 
\draw [ultra thick, green!50!black] (r0) to [out=0, in=270] (v3); 
\draw [ultra thick, green!50!black] (v3) to [out=90, in=300] (r2); 
\draw [ultra thick, green!50!black] (r2) to [out=120, in=0] (v1); 
\draw [ultra thick, green!50!black] (v1) to [out=180, in=60] (r1); 

\draw [ultra thick, green!50!black] (v0)
.. controls ($ (v0) + (240:3) $) and ($ (r1) + (150:3) $) .. (r1); 

\foreach \x/\word in {(r0)/r0, (r1)/r1, (r2)/r2, (r3)/r3}
{
\draw [red, fill=red] \x circle  (10pt);
}

\foreach \x/\word in {(e0)/e0, (e1)/e1, (e2)/e2, (e3)/e3}
{
\draw [green!50!black, fill=green!50!black] \x circle  (10pt);
}

\foreach \x/\word in {(v0)/v0, (v1)/v1, (v2)/v2, (v3)/v3, (v4)/v4}
{
\draw [blue, fill=blue] \x circle (10pt);
}
\end{tikzpicture}
\begin{tikzpicture}[scale = 0.2]
\coordinate (r3) at (1,6);
\coordinate (e1) at (-1,4); 
\coordinate (v1) at (-1.5, 2);
\coordinate (e2) at (-2,0); 
\coordinate (v2) at (-2.3,-2); 
\coordinate (e3) at (-2.5,-4); 
\coordinate (r0) at (-3,-7); 
\coordinate (v4) at (3,4);
\coordinate (r2) at (2.5,0);
\coordinate (v3) at (4,-3);
\coordinate (e0) at (7,3);
\coordinate (v0) at (-6,6); 
\coordinate (r1) at (-6, 0); 

\draw [ultra thick, red] (e1) to [out = 250, in = 70] (v1); 
\draw [ultra thick, red] (v1) to (e2); 
\draw [ultra thick, red] (e2) to (v2); 
\draw [ultra thick, red] (v2) to (e3); 
\draw [ultra thick, red] (e0) to [out=165, in=0] (v4); 
\draw [ultra thick, red] (v4) to [out=180, in=10] (e1); 
\draw [ultra thick, red] (v0)
.. controls ($ (v0) + (210:8) $) and ($ (e3) + (200:8) $) .. (e3); 
\draw [ultra thick, red] (e1) to [out=150, in=330] (v0); 
\draw [ultra thick, red] (e3) to [out=300, in=210] (v3); 
\draw [ultra thick, red] (v3) to [out=30, in=300] (e0); 

\draw [ultra thick, red] (e0) 
.. controls ($ (e0) + (60:7) $) and ($ (v0) + (60:7) $) .. (v0); 

\foreach \x/\word in {(r0)/r0, (r1)/r1, (r2)/r2, (r3)/r3}
{
\draw [red, fill=red] \x circle  (10pt);
}

\foreach \x/\word in {(e0)/e0, (e1)/e1, (e2)/e2, (e3)/e3}
{
\draw [green!50!black, fill=green!50!black] \x circle  (10pt);
}

\foreach \x/\word in {(v0)/v0, (v1)/v1, (v2)/v2, (v3)/v3, (v4)/v4}
{
\draw [blue, fill=blue] \x circle (10pt);
}

\end{tikzpicture}
\end{center}
\caption{Another trinity, and triple of bipartite plane graphs $G_V, G_E, G_R$, shown together with vertices in their complementary regions.}
\label{fig:triple_of_graphs}
\label{fig:trinity}
\end{figure}
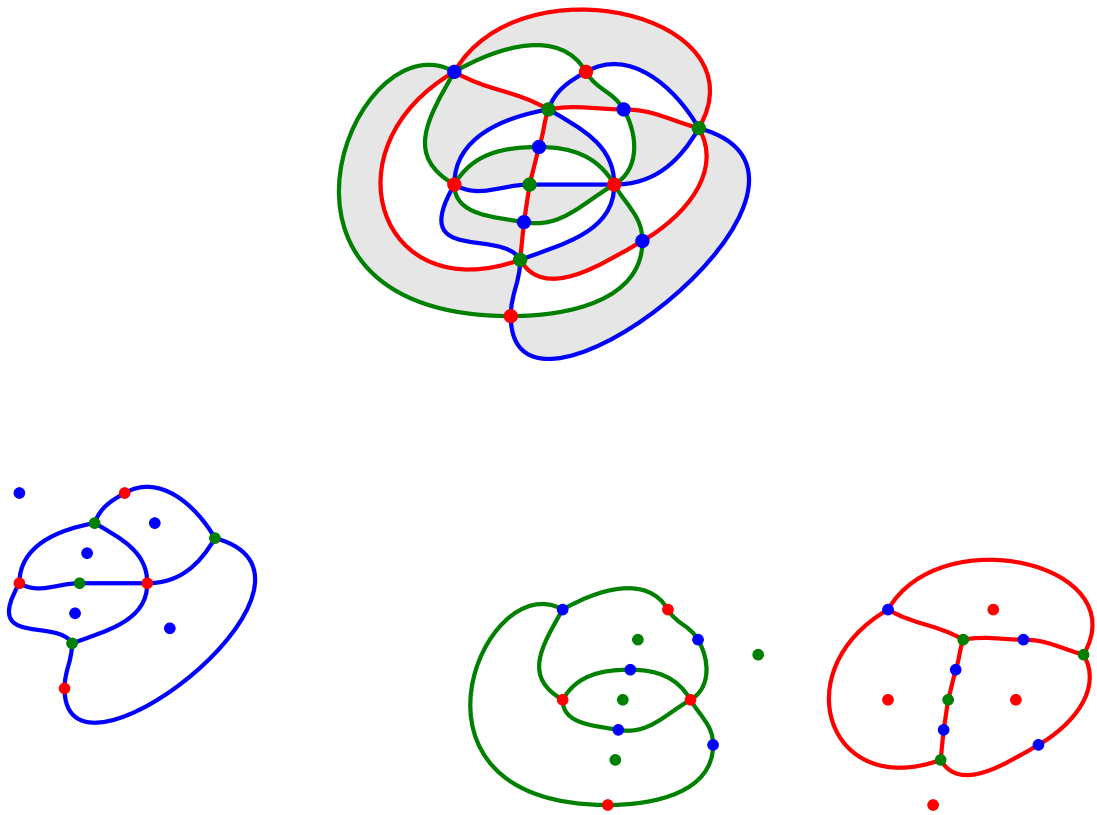

As discussed in section \ref{sec:bipartite_hypergraphs} above, a bipartite graph can be regarded as a hypergraph in two different ways, related by abstract duality. The three bipartite graphs $G_V$, $G_E$, $G_R$ thus yield \emph{six} hypergraphs. Writing $(V,E)$ for the hypergraph with vertices $V$ and hyperedges $E$, wee that $G=G_R$ yields the two hypergraphs $(V,E)$ and $(E,V)$; $G_E$ yields the hypergraphs $(R,V)$ and $(V,R)$; and $G_V$ yields $(E,R)$ and $(R,E)$.

The hypergraphs $(R,E)$ and $(V,E)$ are, in a certain sense, ``planar duals", although in a slightly different sense from the standard construction of section \ref{sec:planar_duals}. The two hypergraphs have the same hyperedges, and the vertices of one correspond to the complementary regions of the other. (However the bipartite graph of $(V,E)$  is $G_R$, while the bipartite graph of $(R,E)$ is $G_V \neq G_R^*$.) See figure \ref{fig:six_hypergraphs}.

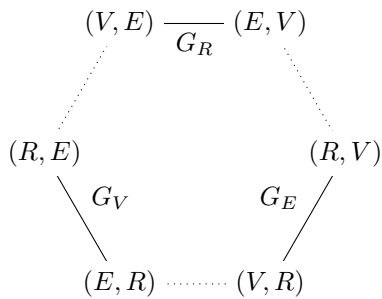
\begin{figure}
\begin{center}
\begin{tikzpicture}
\draw (60:2) -- node[midway, below] {$G_R$} (120:2);
\draw (180:2) -- node[midway, above right] {$G_V$} (240:2);
\draw (-60:2) -- node[midway, above left] {$G_E$} (0:2);
\draw [dotted] (0:2) -- (60:2);
\draw [dotted] (120:2) -- (180:2);
\draw [dotted] (-60:2) -- (-120:2);
\draw (120:2)  node[fill=white] {$(V,E)$};
\draw (60:2) node[fill=white] {$(E,V)$};
\draw (180:2) node[fill=white] {$(R,E)$};
\draw (240:2) node[fill=white] {$(E,R)$};
\draw (-60:2) node[fill=white] {$(V,R)$};
\draw (0:2) node[fill=white] {$(R,V)$};
\end{tikzpicture}
\caption{The six hypergraphs associated to a trinity. Abstract duality is marked by solid lines; ``planar duality" is marked by dotted lines.}
\label{fig:six_hypergraphs}
\end{center}
\end{figure}

\subsection{Planar duals of bipartite plane graphs and arborescences}
\label{sec:planar_duals_bipartite_plane_graphs}

Each of the three bipartite plane graphs $G_V, G_E, G_R$ in a trinity $\Tr$ has a planar dual $G_V^*, G_E^*, G_R^*$. Such a planar dual can be drawn on $\Tr$ so that each edge passes through precisely two triangles, one black and one white. We can therefore naturally orient each edge to run from the black to the white triangle, making $G_V^*, G_E^*, G_R^*$ into \emph{directed} plane graphs. Moreover, around each vertex of $\Tr$, triangles alternate in colour, so at each vertex of $G_V^*$, $G_E^*$ or $G_R^*$, edges are alternately incoming and outgoing. In particular, at every vertex, the in-degree and out-degree are equal. A directed graph where in- and out-degrees are equal at each vertex is called \emph{balanced}.

For the original example of figure \ref{fig:graph_G}, the planar duals $G_V^*, G_E^*, G_R^*$ are shown in figure \ref{fig:easy_trinity_with_dual}. For the example of figure \ref{fig:trinity}, the violet graph $G_V$ and its planar dual $G_V^*$ are shown in figure \ref{fig:trinity_with_dual}.

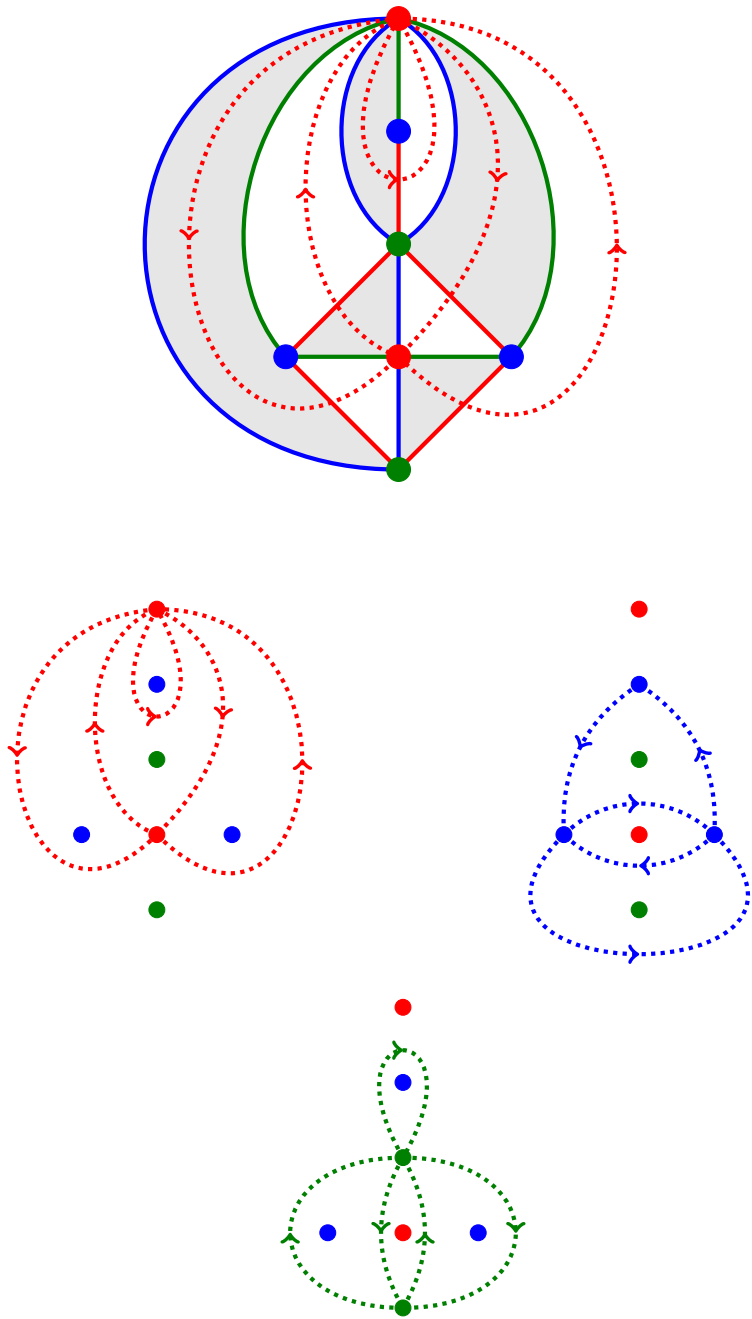
\begin{figure}
\begin{center}
\begin{tikzpicture}[scale = 1.5]
\coordinate (e1) at (0,0);
\coordinate (e2) at (0,-2);
\coordinate (v1) at (0,1);
\coordinate (v2) at (1,-1);
\coordinate (v3) at (-1,-1);
\coordinate (r1) at (0,-1);
\coordinate (r2) at (0,2);

\fill [black!10!white] (r2) to [bend right=60] (e1) -- (v1) -- cycle;
\fill [black!10!white] (r2) .. controls (-3,2) and (-3,-2) .. (e2) -- (v3) to [bend left=60] (r2);
\fill [black!10!white] (r1) -- (e1) -- (v3) -- cycle;
\fill [black!10!white] (r1) -- (e2) -- (v2) -- cycle;
\fill [black!10!white] (r2) to [bend left=60] (e1) -- (v2) to [bend right=60] (r2);

\draw [ultra thick, red] (e1) -- (v1);
\draw [ultra thick, red] (e1) -- (v2);
\draw [ultra thick, red] (e1) -- (v3);
\draw [ultra thick, red] (e2) -- (v2);
\draw [ultra thick, red] (e2) -- (v3);

\draw [ultra thick, blue] (r1) -- (e1);
\draw [ultra thick, blue] (r1) -- (e2);
\draw [ultra thick, blue] (r2) to [bend right=60] (e1); 
\draw [ultra thick, blue] (r2) to [bend left=60] (e1);
\draw [ultra thick, blue] (r2) .. controls (-3,2) and (-3,-2) .. (e2);

\draw [ultra thick, green!50!black] (r1) -- (v2);
\draw [ultra thick, green!50!black] (r1) -- (v3);
\draw [ultra thick, green!50!black] (r2) -- (v1);
\draw [ultra thick, green!50!black] (r2) to [bend left=60] (v2);
\draw [ultra thick, green!50!black] (r2) to [bend right=60] (v3);

\begin{scope}[dotted, ultra thick, decoration={
    markings,
    mark=at position 0.5 with {\arrow{>}}}
    ] 
\draw [red, postaction={decorate}] (r2) .. controls ($ (r2) + (-120:2.2) $) and ($ (r2) + (-60:2.2) $) .. (r2);
\draw [red, postaction={decorate}] (r2) .. controls ($ (r2) + (-25:1) $) and ($ (r1) + (45:2) $) .. (r1);
\draw [red, postaction={decorate}] (r2) .. controls ($ (r2) + (-175:3) $) and ($ (r1) + (-135:2.7) $) .. (r1);
\draw [red, postaction={decorate}] (r1) to [bend left=70] (r2);
\draw [red, postaction={decorate}] (r1) .. controls ($ (r1) + (-45:3) $) and ($ (r2) + (0:3) $) .. (r2);
\end{scope}

\foreach \x in {(e1), (e2)}
{
\draw [green!50!black, fill=green!50!black] \x circle  (3pt);
}

\foreach \x in {(v1), (v2), (v3)}
{
\draw [blue, fill=blue] \x circle  (3pt);
}

\foreach \x in {(r1), (r2)}
{
\draw [red, fill=red] \x circle  (3pt);
}

\end{tikzpicture}

\begin{tikzpicture}[scale = 1]
\coordinate (e1) at (0,0);
\coordinate (e2) at (0,-2);
\coordinate (v1) at (0,1);
\coordinate (v2) at (1,-1);
\coordinate (v3) at (-1,-1);
\coordinate (r1) at (0,-1);
\coordinate (r2) at (0,2);





\begin{scope}[dotted, ultra thick, decoration={
    markings,
    mark=at position 0.5 with {\arrow{>}}}
    ] 
\draw [red, postaction={decorate}] (r2) .. controls ($ (r2) + (-120:2.2) $) and ($ (r2) + (-60:2.2) $) .. (r2);
\draw [red, postaction={decorate}] (r2) .. controls ($ (r2) + (-25:1) $) and ($ (r1) + (45:2) $) .. (r1);
\draw [red, postaction={decorate}] (r2) .. controls ($ (r2) + (-175:3) $) and ($ (r1) + (-135:2.7) $) .. (r1);
\draw [red, postaction={decorate}] (r1) to [bend left=70] (r2);
\draw [red, postaction={decorate}] (r1) .. controls ($ (r1) + (-45:3) $) and ($ (r2) + (0:3) $) .. (r2);
\end{scope}

\foreach \x in {(e1), (e2)}
{
\draw [green!50!black, fill=green!50!black] \x circle  (3pt);
}

\foreach \x in {(v1), (v2), (v3)}
{
\draw [blue, fill=blue] \x circle  (3pt);
}

\foreach \x in {(r1), (r2)}
{
\draw [red, fill=red] \x circle  (3pt);
}

\end{tikzpicture}
\begin{tikzpicture}[scale = 1]
\coordinate (e1) at (0,0);
\coordinate (e2) at (0,-2);
\coordinate (v1) at (0,1);
\coordinate (v2) at (1,-1);
\coordinate (v3) at (-1,-1);
\coordinate (r1) at (0,-1);
\coordinate (r2) at (0,2);





\begin{scope}[dotted, ultra thick, decoration={
    markings,
    mark=at position 0.5 with {\arrow{>}}}
    ] 
\draw [blue, postaction={decorate}] (v1) to [bend right=30] (v3);
\draw [blue, postaction={decorate}] (v2) to [bend right=30] (v1);
\draw [blue, postaction={decorate}] (v3) to [bend left=45] (v2);
\draw [blue, postaction={decorate}] (v2) to [bend left=45] (v3);
\draw [blue, postaction={decorate}] (v3) .. controls ($ (v3) + (-135:3) $) and ($ (v2) + (-45:3) $) .. (v2);
\end{scope}

\foreach \x in {(e1), (e2)}
{
\draw [green!50!black, fill=green!50!black] \x circle  (3pt);
}

\foreach \x in {(v1), (v2), (v3)}
{
\draw [blue, fill=blue] \x circle  (3pt);
}

\foreach \x in {(r1), (r2)}
{
\draw [red, fill=red] \x circle  (3pt);
}

\end{tikzpicture}
\begin{tikzpicture}[scale = 1]
\coordinate (e1) at (0,0);
\coordinate (e2) at (0,-2);
\coordinate (v1) at (0,1);
\coordinate (v2) at (1,-1);
\coordinate (v3) at (-1,-1);
\coordinate (r1) at (0,-1);
\coordinate (r2) at (0,2);





\begin{scope}[dotted, ultra thick, decoration={
    markings,
    mark=at position 0.5 with {\arrow{>}}}
    ] 
\draw [green!50!black, postaction={decorate}] (e1) .. controls ($ (e1) + (120:2.2) $) and ($ (e1) + (60:2.2) $) .. (e1);
\draw [green!50!black, postaction={decorate}] (e1) to [bend right=30] (e2);
\draw [green!50!black, postaction={decorate}] (e2) to [bend right=30] (e1);
\draw [green!50!black, postaction={decorate}] (e1) .. controls ($ (e1) + (0:2) $) and ($ (e2) + (0:2) $) .. (e2);
\draw [green!50!black, postaction={decorate}] (e2) .. controls ($ (e2) + (180:2) $) and ($ (e1) + (180:2) $) .. (e1);

\end{scope}

\foreach \x in {(e1), (e2)}
{
\draw [green!50!black, fill=green!50!black] \x circle  (3pt);
}

\foreach \x in {(v1), (v2), (v3)}
{
\draw [blue, fill=blue] \x circle  (3pt);
}

\foreach \x in {(r1), (r2)}
{
\draw [red, fill=red] \x circle  (3pt);
}

\end{tikzpicture}
\end{center}
\caption{The trinity of figure \ref{fig:easy_trinity}, together with the planar duals of its three bipartite plane graphs. ($G_R^*$ is also shown on the trinity.)}
\label{fig:easy_trinity_with_dual}
\end{figure}

\begin{figure}
\begin{center}
\begin{tikzpicture}[scale = 0.2]
\coordinate (r3) at (1,6);
\coordinate (e1) at (-1,4); 
\coordinate (v1) at (-1.5, 2);
\coordinate (e2) at (-2,0); 
\coordinate (v2) at (-2.3,-2); 
\coordinate (e3) at (-2.5,-4); 
\coordinate (r0) at (-3,-7); 
\coordinate (v4) at (3,4);
\coordinate (r2) at (2.5,0);
\coordinate (v3) at (4,-3);
\coordinate (e0) at (7,3);
\coordinate (v0) at (-6,6); 
\coordinate (r1) at (-6, 0); 

\draw [ultra thick, blue] (r3) to[out=210,in=70] (e1); 
\draw [ultra thick, blue] (e3) to [out=270, in=90] (r0); 
\draw [ultra thick, blue] (r0)
.. controls ($ (r0) + (270:8) $) and ($ (e0) + (345:10) $) .. (e0); 
\draw [ultra thick, blue] (e1) to [out=190, in=90] (r1); 
\draw [ultra thick, blue] (e3) to [out=20, in=270] (r2); 
\draw [ultra thick, blue] (r2) to [out=90, in=330] (e1); 
\draw [ultra thick, blue] (r1)
.. controls ($ (r1) + (240:5) $) and ($ (e3) + (120:2) $) .. (e3); 
\draw [ultra thick, blue] (e0) to [out=120, in=30] (r3); 

\draw [ultra thick, blue] (r1) to [out=330, in=180] (e2); 
\draw [ultra thick, blue] (e2) to [out=0, in=180] (r2); 
\draw [ultra thick, blue] (r2) to [out=0, in=240] (e0); 

\foreach \x/\word in {(r0)/r0, (r1)/r1, (r2)/r2, (r3)/r3}
{
\draw [red, fill=red] \x circle  (10pt);
}

\foreach \x/\word in {(e0)/e0, (e1)/e1, (e2)/e2, (e3)/e3}
{
\draw [green!50!black, fill=green!50!black] \x circle  (10pt);
}

\foreach \x/\word in {(v0)/v0, (v1)/v1, (v2)/v2, (v3)/v3, (v4)/v4}
{
\draw [blue, fill=blue] \x circle (10pt);
}

\end{tikzpicture}
\begin{tikzpicture}[scale = 0.25]
\coordinate (r3) at (1,6);
\coordinate (e1) at (-1,4); 
\coordinate (v1) at (-1.5, 2);
\coordinate (e2) at (-2,0); 
\coordinate (v2) at (-2.3,-2); 
\coordinate (e3) at (-2.5,-4); 
\coordinate (r0) at (-3,-7); 
\coordinate (v4) at (3,4);
\coordinate (r2) at (2.5,0);
\coordinate (v3) at (4,-3);
\coordinate (e0) at (7,3);
\coordinate (v0) at (-6,6); 
\coordinate (r1) at (-6, 0); 

\fill [black!10!white] (e0)
.. controls ($ (e0) + (60:7) $) and ($ (v0) + (60:7) $) .. (v0)
to [out=30, in=120] (r3)
to [out=30, in=120] (e0);

\fill [black!10!white] (r3)
to[out=210,in=70] (e1)
to [out=10, in=180] (v4)
to [out=120, in=300] (r3);

\fill [black!10!white] (v0)
.. controls ($ (v0) + (240:3) $) and ($ (r1) + (150:3) $) .. (r1)
to [out=90, in=190] (e1)
to [out=150, in=330] (v0);

\fill [black!10!white] (e1)
to [out = 250, in = 70] (v1)
to [out=0, in=120] (r2)
to [out=90, in=330] (e1);

\fill [black!10!white] (v1)
to (e2)
to [out=180, in=330] (r1)
to [out=60, in=180] (v1);

\fill [black!10!white] (e2)
to (v2)
to [out=350, in=210] (r2)
to [out=180, in=0] (e2);

\fill [black!10!white] (v2)
to (e3)
.. controls ($ (e3) + (120:2) $)  and ($ (r1) + (240:5) $) .. (r1)
to [out=270, in=170] (v2);

\fill [black!10!white] (v0) 
.. controls ($ (v0) + (210:8) $) and ($ (e3) + (200:8) $) .. (e3)
to [out=270, in=90] (r0)
.. controls ($ (r0) + (180:15) $) and ($ (v0) + (150:6) $) .. (v0);

\fill [black!10!white] (v4)
to [out=0, in=165] (e0)
to [out=240, in=0] (r2)
to [out=30, in=300] (v4);

\fill [black!10!white] (v3) 
to [out=90, in=300] (r2)
to [out=270, in=20] (e3)
to [out=300, in=210] (v3);

\fill [black!10!white] (v3) 
to [out=30, in=300] (e0)
.. controls ($ (e0) + (345:10) $)  and ($ (r0) + (270:8) $) .. (r0)
to [out=0, in=270] (v3);

\begin{scope}[dotted, ultra thick, decoration={
    markings,
    mark=at position 0.5 with {\arrow{>}}}
    ] 
\draw [blue, postaction={decorate}] (v4) to [bend right=10] (v0);
\draw [blue, postaction={decorate}] (v0) to (v1);
\draw [blue, postaction={decorate}] (v1) to (v4);
\draw [blue, postaction={decorate}] (v0) .. controls ($ (v0) + (50:4) $) and ($ (v4) + (45:7) $) .. (v4);
\draw [blue, postaction={decorate}] (v4) to [bend left=30] (v3);
\draw [blue, postaction={decorate}] (v3) to (v2);
\draw [blue, postaction={decorate}] (v2) .. controls ($ (v2) + (225:8) $) and ($ (v0) + (225:6) $) ..  (v0);
\draw [blue, postaction={decorate}] (v0) .. controls ($ (v0) + (205:15) $) and ($ (v3) + (235:11) $) .. (v3);
\draw [blue, postaction={decorate}] (v2) to [bend right=80] (v1);
\draw [blue, postaction={decorate}] (v1) to [bend right=80] (v2);
\draw [blue, postaction={decorate}] (v3) .. controls ($ (v3) + (20:15) $) and ($ (v0) + (70:15) $) .. (v0);
\end{scope}

\draw [ultra thick, blue] (r3) to[out=210,in=70] (e1); 
\draw [ultra thick, red] (e1) to [out = 250, in = 70] (v1); 
\draw [ultra thick, red] (v1) to (e2); 
\draw [ultra thick, red] (e2) to (v2); 
\draw [ultra thick, red] (v2) to (e3); 
\draw [ultra thick, blue] (e3) to [out=270, in=90] (r0); 
\draw [ultra thick, blue] (r0)
.. controls ($ (r0) + (270:8) $) and ($ (e0) + (345:10) $) .. (e0); 
\draw [ultra thick, red] (e0) to [out=165, in=0] (v4); 
\draw [ultra thick, red] (v4) to [out=180, in=10] (e1); 
\draw [ultra thick, blue] (e1) to [out=190, in=90] (r1); 
\draw [ultra thick, green!50!black] (r1) to [out=270, in=170] (v2); 
\draw [ultra thick, green!50!black] (v2) to [out=350, in=210] (r2); 
\draw [ultra thick, green!50!black] (r2) to [out=30, in=300] (v4); 
\draw [ultra thick, green!50!black] (v4) to [out=120, in=300] (r3); 
\draw [ultra thick, green!50!black] (r3) to [out=120, in=30] (v0); 
\draw [ultra thick, red] (v0)
.. controls ($ (v0) + (210:8) $) and ($ (e3) + (200:8) $) .. (e3); 
\draw [ultra thick, blue] (e3) to [out=20, in=270] (r2); 
\draw [ultra thick, blue] (r2) to [out=90, in=330] (e1); 
\draw [ultra thick, red] (e1) to [out=150, in=330] (v0); 
\draw [ultra thick, green!50!black] (v0)
.. controls ($ (v0) + (150:6) $) and ($ (r0) + (180:15) $) .. (r0); 
\draw [ultra thick, green!50!black] (r0) to [out=0, in=270] (v3); 
\draw [ultra thick, green!50!black] (v3) to [out=90, in=300] (r2); 
\draw [ultra thick, green!50!black] (r2) to [out=120, in=0] (v1); 
\draw [ultra thick, green!50!black] (v1) to [out=180, in=60] (r1); 
\draw [ultra thick, blue] (r1)
.. controls ($ (r1) + (240:5) $) and ($ (e3) + (120:2) $) .. (e3); 
\draw [ultra thick, red] (e3) to [out=300, in=210] (v3); 
\draw [ultra thick, red] (v3) to [out=30, in=300] (e0); 
\draw [ultra thick, blue] (e0) to [out=120, in=30] (r3); 

\draw [ultra thick, red] (e0) 
.. controls ($ (e0) + (60:7) $) and ($ (v0) + (60:7) $) .. (v0); 
\draw [ultra thick, green!50!black] (v0)
.. controls ($ (v0) + (240:3) $) and ($ (r1) + (150:3) $) .. (r1); 
\draw [ultra thick, blue] (r1) to [out=330, in=180] (e2); 
\draw [ultra thick, blue] (e2) to [out=0, in=180] (r2); 
\draw [ultra thick, blue] (r2) to [out=0, in=240] (e0); 

\foreach \x/\word in {(r0)/r0, (r1)/r1, (r2)/r2, (r3)/r3}
{
\draw [red, fill=red] \x circle  (10pt);
}

\foreach \x/\word in {(e0)/e0, (e1)/e1, (e2)/e2, (e3)/e3}
{
\draw [green!50!black, fill=green!50!black] \x circle  (10pt);
}

\foreach \x/\word in {(v0)/v0, (v1)/v1, (v2)/v2, (v3)/v3, (v4)/v4}
{
\draw [blue, fill=blue] \x circle (10pt);
}
\end{tikzpicture}
\begin{tikzpicture}[scale = 0.25]
\coordinate (r3) at (1,6);
\coordinate (e1) at (-1,4); 
\coordinate (v1) at (-1.5, 2);
\coordinate (e2) at (-2,0); 
\coordinate (v2) at (-2.3,-2); 
\coordinate (e3) at (-2.5,-4); 
\coordinate (r0) at (-3,-7); 
\coordinate (v4) at (3,4);
\coordinate (r2) at (2.5,0);
\coordinate (v3) at (4,-3);
\coordinate (e0) at (7,3);
\coordinate (v0) at (-6,6); 
\coordinate (r1) at (-6, 0); 

\begin{scope}[dotted, ultra thick, decoration={
    markings,
    mark=at position 0.5 with {\arrow{>}}}
    ] 
\draw [blue, postaction={decorate}] (v4) to [bend right=10] (v0);
\draw [blue, postaction={decorate}] (v0) to (v1);
\draw [blue, postaction={decorate}] (v1) to (v4);
\draw [blue, postaction={decorate}] (v0) .. controls ($ (v0) + (50:4) $) and ($ (v4) + (45:7) $) .. (v4);
\draw [blue, postaction={decorate}] (v4) to [bend left=30] (v3);
\draw [blue, postaction={decorate}] (v3) to (v2);
\draw [blue, postaction={decorate}] (v2) .. controls ($ (v2) + (225:8) $) and ($ (v0) + (225:6) $) ..  (v0);
\draw [blue, postaction={decorate}] (v0) .. controls ($ (v0) + (205:15) $) and ($ (v3) + (235:11) $) .. (v3);
\draw [blue, postaction={decorate}] (v2) to [bend right=80] (v1);
\draw [blue, postaction={decorate}] (v1) to [bend right=80] (v2);
\draw [blue, postaction={decorate}] (v3) .. controls ($ (v3) + (20:15) $) and ($ (v0) + (70:15) $) .. (v0);
\end{scope}

\foreach \x/\word in {(r0)/r0, (r1)/r1, (r2)/r2, (r3)/r3}
{
\draw [red, fill=red] \x circle  (10pt);
}

\foreach \x/\word in {(e0)/e0, (e1)/e1, (e2)/e2, (e3)/e3}
{
\draw [green!50!black, fill=green!50!black] \x circle  (10pt);
}

\foreach \x/\word in {(v0)/v0, (v1)/v1, (v2)/v2, (v3)/v3, (v4)/v4}
{
\draw [blue, fill=blue] \x circle (10pt);
}
\end{tikzpicture}
\end{center}
\caption{The trinity of figure \ref{fig:trinity}, together with its violet graph $G_V$ and its dual.}
\label{fig:trinity_with_dual}
\end{figure}
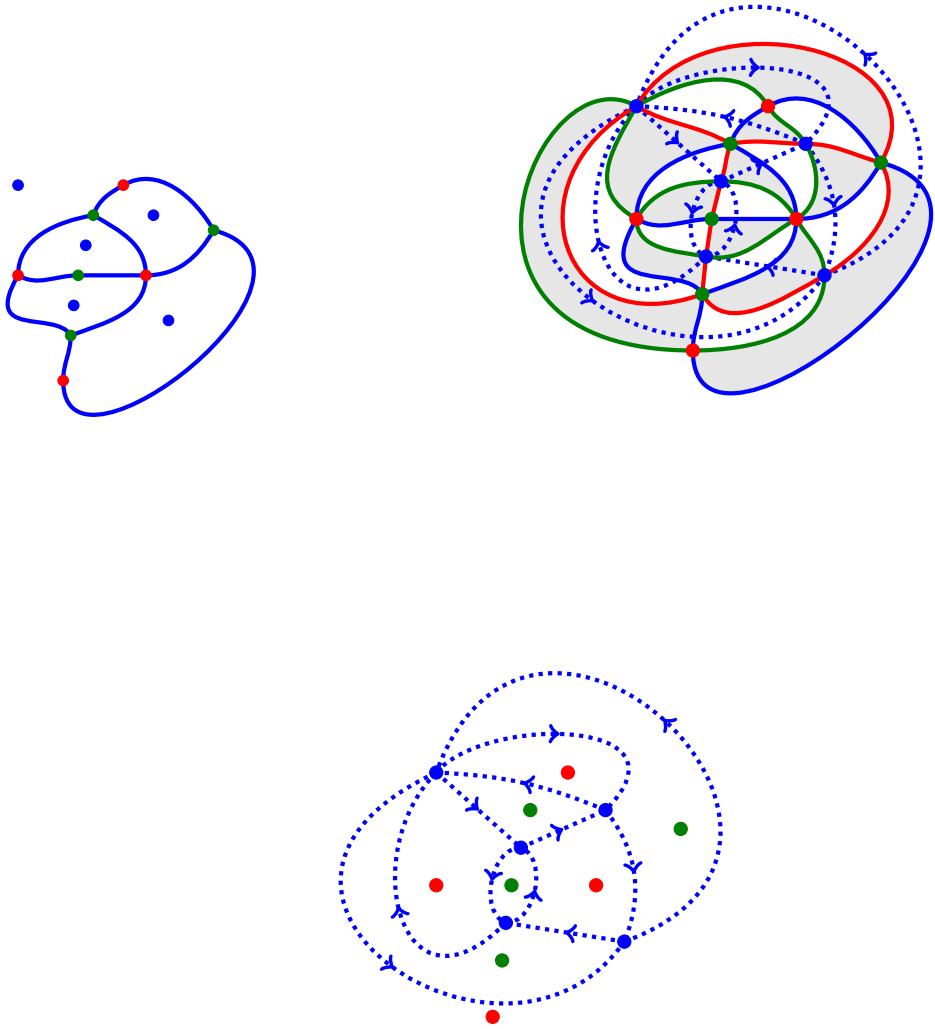

When we have a directed graph such as $G_V^*, G_E^*, G_R^*$, we can consider spanning trees which behave nicely with respect to the orientations on edges. 
\begin{defn}
Let $D$ be a directed graph, with a distinguished \emph{root} vertex $r_0$. An \emph{arborescence} of $D$ rooted at $r_0$ is a subgraph $A$ of $D$ such that
\begin{enumerate}
\item the connected components of $A$ not containing $r_0$ are isolated points, and
\item the connected component of $A$ containing $r_0$ (the \emph{root component}) is a tree in which there is a directed path from $r_0$ to any other vertex.
\end{enumerate}
A \emph{spanning arborescence} is an arborescence without isolated points.
\end{defn}
Thus, an arobrescence is a subgraph which consists of a tree, in which every edge points away from the root vertex, together with isolated vertices. A spanning arborescence is a spanning tree in which every edge points away from the root.

It is known \cite{Tutte48, vA-E_dB51} that the number of spanning arborescences in a balanced finite directed graph $D$ does not depend on the choice of root vertex; hence the following definition makes sense.
\begin{defn}
Let $D$ be a balanced finite directed graph. The number of spanning arborescences in $D$ (for any choice of root vertex) is called the \emph{arborescence number} of $D$ and denoted $\rho(D)$.
\end{defn}

It is interesting to consider spanning arborescences in the balanced directed plane graphs $G_V^*$, $G_E^*$, $G_R^*$ of a trinity. In figure \ref{fig:easy_trinity_with_dual} we observe that an arborescence in $G_R^*$ consists of a single edge from one red vertex to the other; if we fix the root vertex, then the direction of this edge is also fixed. So for either choice of root, there are 2 spanning arborescences, and $\rho(G_R^*) = 2$. It is also not difficult to check that, for this example, that $\rho(G_V^*) = 2$ and $\rho(G_E^*) = 2$ also. This is not a coincidence; it is \emph{Tutte's tree trinity theorem}, which we now state.
\begin{thm}[Tutte \cite{Tutte48}]
\label{thm:Tutte_tree_trinity_thm}
For the three bipartite plane graphs $G_V, G_E, G_R$ of a trinity,
\[
\rho(G_V^*) = \rho(G_E^*) = \rho(G_R^*).
\]
\end{thm}

There is a beautiful proof of this result, due to Tutte, illustrating the triality relationship between the three graphs. We refer to \cite{Tutte75} and \cite[thm. 9.9]{Kalman13_Tutte} for details, and sketch the proof here.

First however we establish some notation and conventions. Let $n$ denote the number of white triangles in the trinity $\Tr$ . Each edge of $\Tr$ has a black triangle on one side and a white triangle on the other, and each triangle has its three sides of distinct colours. Thus the map which sends each red edge to the adjacent white triangle is a bijection between red edges and white triangles. It follows that the number of red edges is $n$. Similarly there are $n$ violet and $n$ emerald edges; and similarly again, the number of black triangles is also $n$. 

Thus, $\Tr$ is a triangulation of the sphere with $3n$ edges and $2n$ triangular faces, so $|V| + |E| + |R| - 3n + 2n = 2$. Fix a white triangle as the \emph{root} or \emph{outer} triangle, and fix its vertices as the \emph{root} violet, emerald, and red vertices. Then there are $n-1$ non-outer white triangles, and $|V|+|E|+|R|-3=n-1$ non-root vertices altogether (of all colours). In our examples, we will always take the root triangle to be the exterior region, hence the name \emph{outer}.

\begin{proof}[Proof sketch of theorem \ref{thm:Tutte_tree_trinity_thm}]
Given a spanning arborescence $A$ of $G_R^*$, we can obtain a map from non-root red vertices to white triangles. Each non-root red vertex $r$ has precisely one edge of $A$ pointing into it; this edge of $A$ enters $r$ via an adjacent white triangle $t$. Our map assigns $r$ to $t$. In a similar fashion, a spanning arborescence in $G_V^*$ or $G_E^*$ yields a map from non-root vertices of the appropriate colour to adjacent white triangles.

The key step in the proof is that, given a spanning arborescence $A$ of $G_R^*$, there exist unique spanning arborescences of $G_V^*$ and $G_E^*$  (with respect to the root vertices around the outer triangle), such that the union of the maps obtained forms a bijection from non-root vertices of $\Tr$, to white triangles of $\Tr$. Moreover, every bijection from non-root vertices to white triangles which assigns vertices to adjacent triangles yields such a triple of arborescences.

From $A$ (and the associated assignment of red vertices to white triangles), we construct the spanning arborescences of $G_V^*$ and $G_E^*$ explicitly. The edges of $G_R$ whose duals are not included in $A$ form a \emph{spanning tree} $A^*$ of $G_R$. (See figure \ref{fig:spanning_arborescences_with_duals} for an illustration of this phenomenon.)
Adjacent to each non-root degree-1 vertex $v$ (i.e. leaf) of $A^*$ there is a unique white triangle $t$ not assigned yet; we assign $v$ to $t$. Deleting the leaf $v$ from $A^*$ and repeating the process eventually yields a bijection from non-root vertices to adjacent white triangles. One can show that violet vertices are sent to adjacent white triangles so as to yield an arborescence of $G_V^*$, and emerald vertices are sent to adjacent white triangles so as to yield an arborescence of $G_E^*$. See figure \ref{fig:easy_trinity_arb_matching} for an illustration of this construction.

There is of course nothing special about starting with $G_R$. So the arborescence numbers $\rho(G_V^*)$, $\rho(G_E^*)$, $\rho(G_R^*)$ are all equal to the number of bijections discussed above.
\end{proof}

\begin{figure}
\begin{center}
\begin{tikzpicture}[scale = 1.5]
\coordinate (e1) at (0,0);
\coordinate (e2) at (0,-2);
\coordinate (v1) at (0,1);
\coordinate (v2) at (1,-1);
\coordinate (v3) at (-1,-1);
\coordinate (r1) at (0,-1);
\coordinate (r2) at (0,2);

\fill [black!10!white] (r2) to [bend right=60] (e1) -- (v1) -- cycle;
\fill [black!10!white] (r2) .. controls (-3,2) and (-3,-2) .. (e2) -- (v3) to [bend left=60] (r2);
\fill [black!10!white] (r1) -- (e1) -- (v3) -- cycle;
\fill [black!10!white] (r1) -- (e2) -- (v2) -- cycle;
\fill [black!10!white] (r2) to [bend left=60] (e1) -- (v2) to [bend right=60] (r2);

\draw [ultra thick, red] (e1) -- (v1);
\draw [red] (e1) -- (v2);
\draw [ultra thick, red] (e1) -- (v3);
\draw [ultra thick, red] (e2) -- (v2);
\draw [ultra thick, red] (e2) -- (v3);

\draw [blue] (r1) -- (e1);
\draw [blue] (r1) -- (e2);
\draw [blue] (r2) to [bend right=60] (e1); 
\draw [blue] (r2) to [bend left=60] (e1);
\draw [blue] (r2) .. controls (-3,2) and (-3,-2) .. (e2);

\draw [green!50!black] (r1) -- (v2);
\draw [green!50!black] (r1) -- (v3);
\draw [green!50!black] (r2) -- (v1);
\draw [green!50!black] (r2) to [bend left=60] (v2);
\draw [green!50!black] (r2) to [bend right=60] (v3);

\begin{scope}[dotted, ultra thick, decoration={
    markings,
    mark=at position 0.5 with {\arrow{>}}}
    ] 
\draw [red, postaction={decorate}] (r2) .. controls ($ (r2) + (-25:1) $) and ($ (r1) + (45:2) $) .. (r1);
\end{scope}

\draw [black, fill=black] ($ (r1) + (45:0.3) $) circle (2pt);
\draw [black, fill=black] ($ (v1) + (0:0.3) $) circle (2pt);
\draw [black, fill=black] ($ (v3) + (-22.5:0.3) $) circle (2pt);
\draw [black, fill=black] ($ (e1) + (180:0.3) $) circle (2pt);

\foreach \x in {(e1), (e2)}
{
\draw [green!50!black, fill=green!50!black] \x circle  (3pt);
}

\foreach \x in {(v1), (v2), (v3)}
{
\draw [blue, fill=blue] \x circle  (3pt);
}

\foreach \x in {(r1), (r2)}
{
\draw [red, fill=red] \x circle  (3pt);
}

\end{tikzpicture}
\end{center}
\caption{A Tutte matching on the trinity of figure \ref{fig:easy_trinity}. The matching can be constructed from the arborescence $A$ of $G_R^*$ shown in dotted red, which is dual to the spanning tree $A^*$ of $G_R$ shown in heavy red. The assignments of the matching are indicated by black dots.}
\label{fig:easy_trinity_arb_matching}
\end{figure}
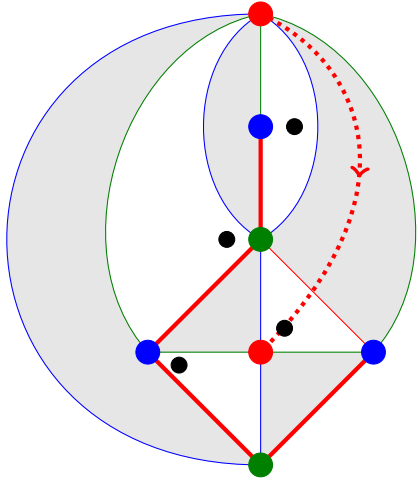

The bijections considered in the above proof are useful to consider and so we make the following definition, following \cite{Kalman13_Tutte}.
\begin{defn}
Let $\Tr$ be a trinity, with a distinguished outer triangle, whose vertices are root vertices of the three colours. A \emph{Tutte matching} on $\Tr$ is a bijection from the set of  non-root points of $\Tr$, to non-outer white triangles of $\Tr$, such that every vertex is sent to an adjacent triangle. The set of Tutte matchings on $\mathfrak{T}$ is denoted $\M_{\mathfrak{T}}$.
\end{defn}

The proof of theorem \ref{thm:Tutte_tree_trinity_thm} shows that $\rho(G_V^*) = \rho(G_E^*) = \rho(G_R^`*) = |\M_{\mathfrak{T}}|$.

\subsection{Tutte matchings and Berman's theorem}

Tutte matchings also arise as the terms in the determinant of a certain adjacency matrix. Given a trinity $\mathfrak{T}$ as above, let us form the \emph{adjacency matrix} $M_\Tr$ as follows. Its rows correspond to the non-root vertices of $\Tr$, and its columns correspond to the non-outer white triangles of $\Tr$. For a non-root vertex $v$ and non-outer white triangle $t$ of $\Tr$, the $(v,t)$ entry of $M_\Tr$ is $1$, if $v$ is adjacent to $t$, otherwise it is $0$. The ordering of the non-root vertices in the rows is arbitrary, as is the ordering of the non-outer white triangles in the columns; so $M_\Tr$ is only well-defined up to permutations of the rows and columns. Nonetheless, its determinant is well-defined up to sign.

For example, in the trinity $\Tr$ of figure \ref{fig:easy_trinity} (derived from the bipartite plane graph of figure \ref{fig:graph_G}), the adjacency matrix is as follows. We show the labels on vertices and white triangles in figure \ref{fig:easy_trinity_with_labels}

\[
M_\Tr =
\bordermatrix{
													& t_1 & t_2 & t_3 & t_4 \cr
\text{\scriptsize{Red } } \widetilde{1}  & 0 	& 	1 & 	0 & 1  \cr
\text{\scriptsize{Emerald } } 1          & 1 	& 	0 & 	1 & 1  \cr
\text{\scriptsize{Violet } } \overline{1} &0 	& 	0 & 	1 & 0 \cr
\text{\scriptsize{Violet } } \overline{3} &1 	& 	1 & 	0 & 0 
}
\]

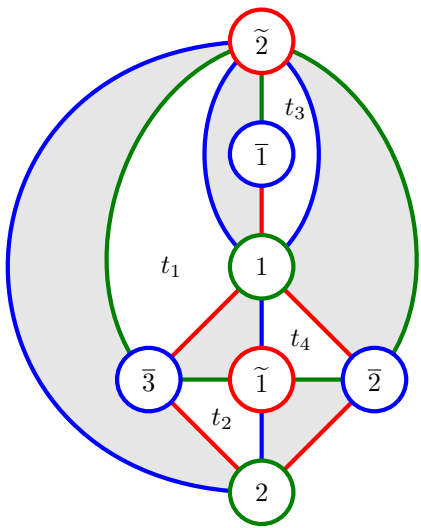
\begin{figure}
\begin{center}
\begin{tikzpicture}[scale = 1.5]
\coordinate (e1) at (0,0);
\coordinate (e2) at (0,-2);
\coordinate (v1) at (0,1);
\coordinate (v2) at (1,-1);
\coordinate (v3) at (-1,-1);
\coordinate (r1) at (0,-1);
\coordinate (r2) at (0,2);

\fill [black!10!white] (r2) to [bend right=60] (e1) -- (v1) -- cycle;
\fill [black!10!white] (r2) .. controls (-3,2) and (-3,-2) .. (e2) -- (v3) to [bend left=60] (r2);
\fill [black!10!white] (r1) -- (e1) -- (v3) -- cycle;
\fill [black!10!white] (r1) -- (e2) -- (v2) -- cycle;
\fill [black!10!white] (r2) to [bend left=60] (e1) -- (v2) to [bend right=60] (r2);

\draw (-0.8,0) node {$t_1$};
\draw (-0.35,-1.35) node {$t_2$};
\draw (0.3,1.4) node {$t_3$};
\draw (0.35,-0.65) node {$t_4$};

\draw [ultra thick, red] (e1) -- (v1);
\draw [ultra thick, red] (e1) -- (v2);
\draw [ultra thick, red] (e1) -- (v3);
\draw [ultra thick, red] (e2) -- (v2);
\draw [ultra thick, red] (e2) -- (v3);

\draw [ultra thick, blue] (r1) -- (e1);
\draw [ultra thick, blue] (r1) -- (e2);
\draw [ultra thick, blue] (r2) to [bend right=60] (e1); 
\draw [ultra thick, blue] (r2) to [bend left=60] (e1);
\draw [ultra thick, blue] (r2) .. controls (-3,2) and (-3,-2) .. (e2);

\draw [ultra thick, green!50!black] (r1) -- (v2);
\draw [ultra thick, green!50!black] (r1) -- (v3);
\draw [ultra thick, green!50!black] (r2) -- (v1);
\draw [ultra thick, green!50!black] (r2) to [bend left=60] (v2);
\draw [ultra thick, green!50!black] (r2) to [bend right=60] (v3);

\foreach \x/\word in {(e1)/{$1$}, (e2)/{$2$}}
{
\draw [green!50!black, ultra thick, fill=white] \x circle  (8pt);
\draw \x node {$\word$};
}

\foreach \x/\word in {(r1)/{$\widetilde{1}$}, (r2)/{$\widetilde{2}$}}
{
\draw [red, ultra thick, fill=white] \x circle  (8pt);
\draw \x node {\word};
}

\foreach \x/\word in {(v1)/{$\overline{1}$}, (v2)/{$\overline{2}$}, (v3)/{$\overline{3}$}}
{
\draw [blue, ultra thick, fill=white] \x circle  (8pt);
\draw \x node {\word}; gherkin
}
\end{tikzpicture}
\end{center}
\caption{The trinity of figure \ref{fig:easy_trinity}, with labellings of vertices and non-outer white triangles.}
\label{fig:easy_trinity_with_labels}
\end{figure}

It can easily be seen, expanding out the determinant in a naive fashion, that there are two terms, both contributing $+1$ to the determinant, so the determinant is $2$. These terms correspond to the matchings
\[
\begin{array}{c}
\tilde{1} \leftrightarrow t_2 \\
1 \leftrightarrow t_4 \\
\overline{1} \leftrightarrow t_3 \\
\overline{3} \leftrightarrow t_1
\end{array}
\quad 
\text{and} 
\quad
\begin{array}{c}
\tilde{1} \leftrightarrow t_4 \\
1 \leftrightarrow t_1 \\
\overline{1} \leftrightarrow t_3 \\
\overline{3} \leftrightarrow t_2 .
\end{array}
\]
which correspond to the two spanning arborescences of $G_V^*$ (or $G_E^*$ or $G_R^*$). The second matching is the one illustrated in figure \ref{fig:easy_trinity_arb_matching}

The terms in an expansion of $\det M_\Tr$ correspond to Tutte matchings, and so we obtain a $\pm 1$ for each Tutte matching. As it turns out, each $\pm 1$ always has the same sign: 
Berman in \cite{Berman80} showed that the determinant always gives the number of Tutte matchings, giving the following theorem.
\begin{thm}[Berman \cite{Berman80}]
\label{thm:Berman_determinant}
The adjacency matrix $M_\Tr$ of a trinity $\Tr$ satisfies
\[
| \det M_{\mathfrak{T}} | = \rho(G_V^*) = \rho(G_E^*) = \rho(G_R^*) = | \M_{\mathfrak{T}} |.
\]
\qed
\end{thm}

Because so many quantities associated to a trinity are equal to the same number --- and we will shortly see more --- we will give this number a name.
\begin{defn}
For a trinity $\mathfrak{T}$, the common arborescence number of $G_V^*, G_E^*, G_R^*$ is called the \emph{magic number} of $\mathfrak{T}$.
\end{defn}

\section{Trinities, arborescences and polytopes}
\label{sec:trinities}

\subsection{The story so far}
\label{sec:arborescences_hypertrees}

Let us briefly recap what we have found so far. In section \ref{sec:hypergraphs_polytope_dualities} we saw that, associated to a bipartite graph $G$ and the corresponding abstract dual hypergraphs $\HH, \overline{\HH}$, are several polytopes: the GP polytope $\P_\HH$, trimmed GP polytope $\P_\HH^-$, hypertree polytope $\S_\HH$, and root polytope $\QQ_\HH$. These polytopes obey various interesting properties, such as: $\P_\HH^- = \S_{\overline{\HH}}$; $\P_\HH$ and $\P_{\overline{\HH}}$ can be obtained as slices of $\QQ_G$; and $| S_\HH | = | S_{\overline{\HH}} |$.

Then, in section \ref{sec:plane_graphs_dualities_trinities}, we considered \emph{plane} bipartite graphs $G$, which naturally give rise to trinities $\Tr$ and triples $G_V, G_E, G_R = G$ in a relationship of triality: the planar duals $G_V^*, G_E^*, G_R^*$ all have the same arborescence number, which is the same as the number of Tutte matchings $|\M_\Tr|$, and also equal to $|\det M_\Tr|$, where $M_\Tr$ is an adjacency matrix.

Let us now combine these ideas and consider the polytopes associated to the graphs of a trinity. We begin by taking up the example of section \ref{sec:polytope_example}.

\subsection{Back to the example}
\label{sec:more_example}

We return to the bipartite plane graph $G$ of figure \ref{fig:graph_G}, which yields the trinity $\Tr$ and bipartite plane graphs $G_V, G_E, G_R=G$ of figure \ref{fig:easy_trinity}, together with the six corresponding hypergraphs $(V,E)$, $(E,V)$, $(R,V)$, $(V,R)$, $(E,R)$, $(R,E)$. Various aspects of $\Tr$ were illustrated throughout section \ref{sec:plane_graphs_dualities_trinities}, in figures \ref{fig:easy_trinity_with_dual} (dual graphs $G_V^*, G_E^*, G_R^*$), \ref{fig:easy_trinity_arb_matching} (Tutte matching and arborescence) and \ref{fig:easy_trinity_with_labels} (full labelling).

In section \ref{sec:polytope_example} we calculated the polytopes associated to $G=G_R$. We denoted blue/violet vertices by $V$ and green/emerald vertices by $U$; we now write $E$ for the emerald vertices. We denoted the two hypergraphs of $G=G_R$ by $\HH$ and $\overline{\HH}$; we now recognise these hypergraphs as $\HH = (V,E)$ and $\overline{\HH} = (E,V)$. We found that the hypergraphs $(V,E)$ and $(E,V)$ each have 2 hypertrees, which are the lattice points $S_{(V,E)}, S_{(E,V)}$ of the polytopes $\S_{(V,E)} = \P_{(E,V)}^-$ and $\S_{(E,V)}= \P_{(V,E)}^-$ respectively.

Let us then turn to the two other bipartite graphs of the trinity, $G_V$ and $G_E$; see figure \ref{fig:easy_green_graph_big_nodes}. We observe that $G_E$ is isomorphic to $G_R$ as a bipartite graph, and we have isomorphisms of hypergraphs $(V,R) = (V,E)$ and $(R,V) = (E,V)$. So the polytopes associated to $G_E$ are isomorphic to those of $G_R$; in particular,
\begin{align*}
\P_{(V,R)} &= 
\Conv\{ (1,1,0), (1,0,1), (0,2,0), (0,0,2) \} \subset \R^V, \\
\P_{(V,R)}^- = \S_{(R,V)} &= 
\Conv\{ (0,1,0),(0,0,1) \} \subset \R^V, \\
\P_{(R,V)} &= 
\Conv \{ (0,3),(2,1) \} \subset \R^R, \\
\P_{(R,V)}^- = \S_{(V,R)} &= 
\Conv\{ (0,2), (1,1) \} \subset \R^R, \\
\end{align*}

\begin{figure}
\begin{center}
\begin{tikzpicture}[scale = 1]
\coordinate (e1) at (0,0);
\coordinate (e2) at (0,-2);
\coordinate (v1) at (0,1);
\coordinate (v2) at (1,-1);
\coordinate (v3) at (-1,-1);
\coordinate (r1) at (0,-1);
\coordinate (r2) at (0,2);

\draw [ultra thick, blue] (r1) -- (e1);
\draw [ultra thick, blue] (r1) -- (e2);
\draw [ultra thick, blue] (r2) to [bend right=60] (e1); 
\draw [ultra thick, blue] (r2) to [bend left=60] (e1);
\draw [ultra thick, blue] (r2) .. controls (-3,2) and (-3,-2) .. (e2);

\foreach \x/\word in {(e1)/{$1$}, (e2)/{$2$}}
{
\draw [green!50!black, ultra thick, fill=white] \x circle  (8pt);
\draw \x node {\word};
}


\foreach \x/\word in {(r1)/{$\widetilde{1}$}, (r2)/{$\widetilde{2}$}}
{
\draw [red, ultra thick, fill=white] \x circle  (8pt);
\draw \x node {\word};
}

\end{tikzpicture}
\begin{tikzpicture}[scale = 1]
\coordinate (e1) at (0,0); 
\coordinate (e2) at (0,-2);
\coordinate (v1) at (0,1);
\coordinate (v2) at (1,-1);
\coordinate (v3) at (-1,-1);
\coordinate (r1) at (0,-1);
\coordinate (r2) at (0,2);

\draw [ultra thick, green!50!black] (r1) -- (v2);
\draw [ultra thick, green!50!black] (r1) -- (v3);
\draw [ultra thick, green!50!black] (r2) -- (v1);
\draw [ultra thick, green!50!black] (r2) to [bend left=60] (v2);
\draw [ultra thick, green!50!black] (r2) to [bend right=60] (v3);

\foreach \x/\word in {(r1)/{$\widetilde{1}$}, (r2)/{$\widetilde{2}$}}
{
\draw [red, ultra thick, fill=white] \x circle  (8pt);
\draw \x node {\word};
}

\foreach \x/\word in {(v1)/{$\overline{1}$}, (v2)/{$\overline{2}$}, (v3)/{$\overline{3}$}}
{
\draw [blue, ultra thick, fill=white] \x circle  (8pt);
\draw \x node {\word};
}
\end{tikzpicture}
\end{center}
\caption{The violet and emerald graphs of the trinity in figure \ref{fig:easy_trinity}, derived from the graph in figure \ref{fig:graph_G}, with vertices labelled.}
\label{fig:easy_green_graph_big_nodes}
\end{figure}
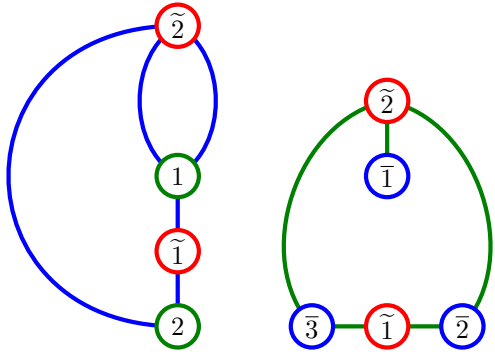

We next calculate the polytopes of $G_V$. First consider the hypergraph $(E,R)$. Its GP polytopes are
\begin{align*}
\P_{(E,R)} &= \Delta_{12} + \Delta_{12} = 2 \Delta_{12} = \Conv \{ (2,0),(0,2) \} \subset \R^E, \\
\P_{(E,R)}^- &= \P_{(E,R)} - \Delta_{12} = \Delta_{12} = \Conv \{(1,0),(0,1) \} \subset \R^E,
\end{align*}
and its spanning trees have $R$-degrees $(1,2)$ and $(2,1)$, with hypertrees $(0,1)$ and $(1,0)$, so
\[
\S_{(E,R)} = \Conv \{ (1,0),(0,1) \} \subset \R^R.
\]
For the abstract dual $(R,E)$ in fact we obtain the same polytopes, with $1,2$ replaced by $\widetilde{1}, \widetilde{2}$
\begin{align*}
\P_{(R,E)} &= 2 \Delta_{\widetilde{1} \widetilde{2}} = \Conv \{ (2,0), (0,2) \} \subset \R^R, \\
\P_{(R,E)}^- &= \Conv \{ (1,0),(0,1) \} \subset \R^R, \\
\S_{(R,E)} &= \Conv \{ (1,0),(0,1) \} \subset \R^E.
\end{align*}

Thus, all six hypergraphs of $\Tr$ have 2 hypertrees; all six hypertree polytopes contain 2 integer points. We found in section \ref{sec:planar_duals_bipartite_plane_graphs} that all three dual graphs $G_V^*, G_E^* G_R^*$ have arborescence number $2$ as well. It's a small example, but again, it is no coincidence.

While all six hypertree polytopes have the same number of integer points, they are not the same polytope. Indeed, the polytopes lie in different spaces: $\S_{(R,V)}$ and $\S_{(E,V)}$ both lie in $\R^V$, while $\S_{(R,E)}, \S_{(V,E)} \subset \R^E$ and $\S_{(V,R)}, \S_{(E,R)} \subset \R^R$. And those polytopes which lie in the same space may not be identical: for instance,
\[
\S_{(V,R)} = \Conv \{ (0,2),(1,1) \} \subset \R^E
\quad \text{but} \quad
\S_{(E,R)} = \Conv \{ (1,0),(0,1) \} \subset \R^E.
\]
While these polytopes are not identical, they are isometric. Again, this is a small example, but no coincidence.

We discuss these non-coincidences in the next section.

\subsection{Hypertrees and arborescences}
\label{sec:hypertrees_arborescences}

We saw in theorems \ref{thm:trimmed_GP_spanning_tree_polytopes} and \ref{thm:hypertrees_in_abstract_duals} that abstract dual hypergraphs have the same number of hypertrees, correspnding to triangulations of the associated root polytope. Thus for a trinity $\Tr$, the abstract dual hypergraphs $(E,R),(R,E)$ of $G_V$ must have the same number of hypertrees, as must the hypergraphs $(R,V),(V,R)$ of $G_E$, and the hypergraphs $(V,E),(E,V)$ of $G_R$. But in fact more is true, as the example suggests: \emph{all six} hypergraphs have the same number of hypertrees, which is the magic number of $\Tr$. This triality relationship on hypergraphs follows from the following theorem, together with theorem \ref{thm:hypertrees_in_abstract_duals} and Tutte's tree trinity theorem \ref{thm:Tutte_tree_trinity_thm}, 

\begin{thm}[K\'{a}lm\'{a}n \cite{Kalman13_Tutte}, thm. 10.1]
\label{thm:Kalman_hypertree_arborescence}
Let $G$ be a plane bipartite graph corresponding to the hypergraph $\HH$. The number of hypertrees in $\HH$ is equal to the arborescence number of the planar dual $G^*$ of $G$:
\[
| S_\HH | = \rho(G^*).
\]
\end{thm}

Together with $\P_\HH^- = \S_{\overline{\HH}}$ (theorem \ref{thm:trimmed_GP_spanning_tree_polytopes}) and Berman's theorem \ref{thm:Berman_determinant}, we now have an even longer string of equalities:
\begin{align*}
|\det M_\Tr| &= \rho(G_V^*) = \rho(G_E^*) = \rho(G_R^*) = |\M_{\mathfrak{T}}| \\
&= | S_{(V,E)} | = | S_{(E,V)} | = | S_{(E,R)} | = | S_{(R,E)} | = | S_{(R,V)} | = | S_{(V,R)} | \\
&= | P_{(V,E)}^- | = | P_{(E,V)}^- | = | P_{(E,R)}^- | = | P_{(R,E)}^- | = | P_{(R,V)}^- | = | P_{(V,R)}^- |
\end{align*}

The proof of theorem \ref{thm:Kalman_hypertree_arborescence} uses the relationship between spanning trees of $G_V^*$ and $G_V$, mentioned in the proof of \ref{thm:Tutte_tree_trinity_thm}: the spanning trees of a plane graph $G$ and its planar dual $G^*$ are naturally in bijection. If $T$ is a spanning tree of $G$, the edges of $G^*$ whose duals are not included in $T$ form a spanning tree $T^*$ of $G^*$, and the correspondence $T \leftrightarrow T^*$ is bijective. We very briefly sketch the proof and refer to \cite{Kalman13_Tutte} for details.

\begin{proof}[Sketch of proof of theorem \ref{thm:Kalman_hypertree_arborescence}]
Given an arborescence $A$ of $G^*$, we forget the orientations on its edges and regard it as a spanning tree of $G^*$, so its planar dual $A^*$ is naturally a spanning tree of $G$. The spanning tree $A^*$ then yields a hypertree $f_A$ of $\HH$. This gives a map from arborescences of $G^*$ to hypertrees of $\HH$. With some effort this map can be shown to be a bijection. Since $G^*$ has $\rho(G^*)$ arborescences, and $\HH$ has $|S_\HH|$ hypertrees, the result follows.
\end{proof}

The proof of K\'{a}lm\'{a}n's theorem \ref{thm:Kalman_hypertree_arborescence} thus proves the following stronger statement.
\begin{cor}
\label{cor:hypertree_arborescence_bijection}
The map from arborescences of $G^*$ to hypertrees of $\HH$, which sends an arborescence $A$ to the hypertree $f_A$ of $A^*$, is a bijection.
\qed
\end{cor}

As observed in section \ref{sec:more_example}, while the six hypertree polytopes all contain the same number of integer points, they do not all lie in the same space: rather, $\S_{(R,V)}, \S_{(E,V)} \subset \R^V$, while $\S_{(R,E)}, \S_{(V,E)} \subset \R^E$ and $\S_{(V,R)}, \S_{(E,R)} \subset \R^R$. We also observed that while the pairs of polytopes lying in the same space were not always identical, they were always isometric. This is true in general --- another duality relationship.

\begin{thm}[K\'{a}lm\'{a}n \cite{Kalman13_Tutte}, thm. 8.3]
\label{thm:polytope_reflections}
The hypertree polytopes $\S_{(V,E)}$ and $\S_{(R,E)}$ in $\R^E$ are reflections of each other in a point.
\end{thm}

Said another way, $-\S_{(V,E)}$ is a translation of $\S_{(R,E)}$.
By symmetry, $\S_{(R,V)}, \S_{(E,V)} \subset \R^V$ are also related by a reflection, as are $\S_{(V,R)}, \S_{(E,R)} \subset \R^R$.

We again give a rough sketch of the proof, and refer to \cite{Kalman13_Tutte} for details.
\begin{proof}[Sketch of proof]
Given a hypertree $f_T \in \Z^E$ of $(V,E)$, represented by a spanning tree $T$ of $G_R$, one can construct a spanning tree $T^*$ of $G_V$, such that for each $e \in E$, we have $\deg_e T + \deg_e T^* = |e| + 1$, where $|e|$ is the size of the hyperedge $e$ in $(V,E)$ and $(R,E)$ (these two sizes are equal since violet and red edges alternate around $e$). This $T^*$ can be constructed by deforming the planar dual of $T$.  The hypertree $f_{T^*} \in \Z^E$ of $T^*$ then satisfies $f_T + f_{T^*} = \deg_E G_R - (1,1,\ldots,1)$, so that $f_{T^*}$ is obtained from $f_T$ by reflection in a constant point in $\Z^E$.
\end{proof}

\subsection{Arborescences and the root polytope}
\label{sec:arborescences_root_polytope}

The bijection of corollary \ref{cor:hypertree_arborescence_bijection} means that arborescences are especially nice spanning trees. A hypertree $f$ of $\HH$ (which, recall, is essentially the degree sequence of a spanning tree of $G$ at the hyperedges of $\HH$) arises from at least one, possibly many, spanning trees of $G$, and thus corresponds to at least one, possibly many, spanning trees of $G^*$. The bijection of corollary \ref{cor:hypertree_arborescence_bijection} however says that precisely one of these spanning trees of $G^*$ is an arborescence. So among the spanning trees of $G$ representing a given hypertree $f$, one among them is the ``nicest", namely the one dual to an arborescence. (This choice of ``nicest" does however depend on a choice of root vertex in $G^*$.)

We illustrate this phenomenon in figure \ref{fig:spanning_arborescences_with_duals} for the example graph $G$ of section \ref{sec:hypergraphs_polytope_dualities} (also the red graph $G_R$ of the trinity of figure \ref{fig:easy_trinity}). As usual we choose root vertices to be ``outer". The two arborescences of $G^*$ have duals which are shown as $T_1$ and $T_4$ in figure \ref{fig:spanning_trees_in_G}. But we saw in figure \ref{fig:spanning_trees_in_G} that $G_R$ has $4$ spanning trees. The two trees $T_1, T_4$, the duals of arborescences, are the ``nicest" spanning trees and represent all hypertrees exactly once.

\begin{figure}
\begin{center}
\begin{tikzpicture}[scale = 1]
\coordinate (e1) at (0,0);
\coordinate (e2) at (0,-2);
\coordinate (v1) at (0,1);
\coordinate (v2) at (1,-1);
\coordinate (v3) at (-1,-1);
\coordinate (r1) at (0,-1);
\coordinate (r2) at (0,2);


\draw [ultra thick, red] (e1) -- (v1);
\draw [ultra thick, red] (e1) -- (v2);
\draw [ultra thick, red] (e1) -- (v3);
\draw [ultra thick, red] (e2) -- (v2);
\draw [ultra thick, red] (e2) -- (v3);



\begin{scope}[dotted, ultra thick, decoration={
    markings,
    mark=at position 0.5 with {\arrow{>}}}
    ] 
\draw [red, postaction={decorate}] (r2) .. controls ($ (r2) + (-120:2.2) $) and ($ (r2) + (-60:2.2) $) .. (r2);
\draw [red, postaction={decorate}] (r2) .. controls ($ (r2) + (-25:1) $) and ($ (r1) + (45:2) $) .. (r1);
\draw [red, postaction={decorate}] (r2) .. controls ($ (r2) + (-175:3) $) and ($ (r1) + (-135:2.7) $) .. (r1);
\draw [red, postaction={decorate}] (r1) to [bend left=70] (r2);
\draw [red, postaction={decorate}] (r1) .. controls ($ (r1) + (-45:3) $) and ($ (r2) + (0:3) $) .. (r2);
\end{scope}

\foreach \x in {(e1), (e2)}
{
\draw [green!50!black, fill=green!50!black] \x circle  (3pt);
}

\foreach \x in {(v1), (v2), (v3)}
{
\draw [blue, fill=blue] \x circle  (3pt);
}

\foreach \x in {(r1), (r2)}
{
\draw [red, fill=red] \x circle  (3pt);
}

\end{tikzpicture}
\begin{tikzpicture}[scale = 1]
\coordinate (e1) at (0,0);
\coordinate (e2) at (0,-2);
\coordinate (v1) at (0,1);
\coordinate (v2) at (1,-1);
\coordinate (v3) at (-1,-1);
\coordinate (r1) at (0,-1);
\coordinate (r2) at (0,2);


\draw [ultra thick, red] (e1) -- (v1);
\draw [ultra thick, red] (e1) -- (v3);
\draw [ultra thick, red] (e2) -- (v2);
\draw [ultra thick, red] (e2) -- (v3);



\begin{scope}[dotted, ultra thick, decoration={
    markings,
    mark=at position 0.5 with {\arrow{>}}}
    ] 
\draw [red, postaction={decorate}] (r2) .. controls ($ (r2) + (-25:1) $) and ($ (r1) + (45:2) $) .. (r1);
\end{scope}

\foreach \x in {(e1), (e2)}
{
\draw [green!50!black, fill=green!50!black] \x circle  (3pt);
}

\foreach \x in {(v1), (v2), (v3)}
{
\draw [blue, fill=blue] \x circle  (3pt);
}

\foreach \x in {(r1), (r2)}
{
\draw [red, fill=red] \x circle  (3pt);
}

\end{tikzpicture}
\begin{tikzpicture}[scale = 1]
\coordinate (e1) at (0,0);
\coordinate (e2) at (0,-2);
\coordinate (v1) at (0,1);
\coordinate (v2) at (1,-1);
\coordinate (v3) at (-1,-1);
\coordinate (r1) at (0,-1);
\coordinate (r2) at (0,2);


\draw [ultra thick, red] (e1) -- (v1);
\draw [ultra thick, red] (e1) -- (v2);
\draw [ultra thick, red] (e1) -- (v3);
\draw [ultra thick, red] (e2) -- (v2);



\begin{scope}[dotted, ultra thick, decoration={
    markings,
    mark=at position 0.5 with {\arrow{>}}}
    ] 
\draw [red, postaction={decorate}] (r2) .. controls ($ (r2) + (-175:3) $) and ($ (r1) + (-135:2.7) $) .. (r1);
\end{scope}

\foreach \x in {(e1), (e2)}
{
\draw [green!50!black, fill=green!50!black] \x circle  (3pt);
}

\foreach \x in {(v1), (v2), (v3)}
{
\draw [blue, fill=blue] \x circle  (3pt);
}

\foreach \x in {(r1), (r2)}
{
\draw [red, fill=red] \x circle  (3pt);
}

\end{tikzpicture}

\end{center}
\caption{The red graph of the trinity of figure \ref{fig:easy_trinity}, together with its dual $G_R^*$; and the two spanning arborescences of $G_R^*$, together with planar dual spanning trees of $G_R$. The spanning trees of $G_R$ are $T_1$ and $T_4$ from figure \ref{fig:spanning_trees_in_G}.}
\label{fig:spanning_arborescences_with_duals}
\end{figure}
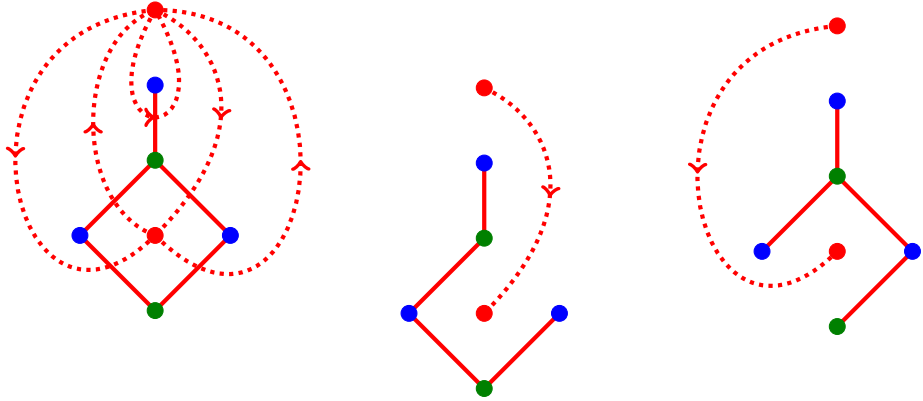

Now, recall from section \ref{sec:polytope_example} that $T_1$ and $T_4$ also gave nice decompositions of polytopes. We found that the root polytope of $G$ decomposed into two tetrahedra corresponding to the root polytopes of $T_1$ and $T_4$, and the GP polytopes of $\HH = (V,E)$ and $\overline{\HH} = (E,V)$ decomposed into GP polytopes for $\TT_1$ and $\TT_4$:
\[
\QQ_G = \QQ_{T_1} \cup \QQ_{T_4}, \quad
\P_{\HH} = \P_{\TT_1} \cup \P_{\TT_4}, \quad
\P_{\overline{\HH}} = \P_{\overline{\TT_1}} \cup \P_{\overline{\TT_4}}.
\]

Again, this is no coincidence. We know by corollary \ref{cor:hypertree_arborescence_bijection} tells us that the  spanning trees dual to arborescences will cover all hypertrees precisely once. And we have seen (theorems \ref{thm:trimmed_GP_spanning_tree_polytopes} and \ref{thm:hypertrees_in_abstract_duals}) that spanning trees which yield a triangulation of $\QQ_G$ also cover all hypertrees precisely once. So it might be expected that the spanning trees dual to arborescences \emph{always} yield a triangulation of $\QQ_G$. This was proved by K\'{a}lm\'{a}n and Murakami. 

\begin{thm}[K\'{a}lm\'{a}n--Murakami \cite{K-Murakami}, thm. 1.1]
\label{thm:KM_root_polytope_triangulation}
Let $G$ be a connected plane bipartite graph, with planar dual $G^*$ and a fixed root vertex. The root polytopes of the spanning trees of $G$ dual to arborescences of $G^*$ form a triangulation of the root polytope $\QQ_G$ of $G$.
\end{thm}

Using the Cayley trick, we also obtain nice decompositions of the GP polytopes.

Note that, if we take a different choice of root vertex, we obtain a different set of spanning arborescences of $G^*$, and hence a different set of spanning trees in $G$ whose root polytopes triangulate $\QQ_G$. So there are in general many triangulations of $\QQ_G$. In our example, if we change the root red vertex from $\widetilde{2}$ to $\widetilde{1}$, then the spanning arborescences of $G^*$ become those dual to spanning trees $T_2$ and $T_3$ of figure \ref{fig:spanning_trees_in_G}. And indeed we found that $\QQ_G$ also decomposes as $\QQ_{T_2} \cup \QQ_{T_3}$.

K\'{a}lm\'{a}n--Murakami go further, and show that the triangulation of $\QQ_G$ obtained is \emph{shellable} \cite[prop. 1.2]{K-Murakami}. Roughly this means that the simplices of the triangulation can be removed one at a time to leave a nice polytope at each stage. The \emph{$h$-vector} of this triangulation turns out to be closely related to the HOMFLY-PT polynomial of a link related to $G$, which we discuss next.

\section{Knots, links and polytopes}
\label{sec:knots_links}

We have now seen, in some detail, the structure which emerges from considering bipartite plane graphs and their polytopes --- they form trinities, and there are interesting relationships between the associated spanning trees, arborescences, Tutte matchings, and polytopes.

We now take a step further in our story and consider some 3-dimensional topological constructions associated to such graphs. We consider a \emph{knot} or \emph{link} associated to a bipartite plane graph.

\subsection{From graph to link: the median construction}
\label{sec:median_construction}

Let $G$ be a finite plane graph, which we regard as lying in $\R^2 \subset S^2 \subset S^3$. We construct a link $L_G$ in $S^3$ as follows. Take a regular neighbourhood $U$ of $G$ in $\R^2$, which can be regarded as consisting of a disc around each vertex of $G$, together with a band along each edge of $G$. Insert a negative half twist in each band of $U$ and call the resulting surface $F_G$. Then $L_G = \partial F_G$ is an alternating link in $S^3$. This construction is called \emph{the median construction}. See figures \ref{fig:easy_median_construction} and \ref{fig:median_construction} for examples, using the red graphs of the trinities of figures \ref{fig:easy_trinity} and \ref{fig:trinity} respectively.

\begin{figure}
\begin{center}
\begin{tikzpicture}[scale = 1.5]
\coordinate (e1) at (0,0);
\coordinate (e2) at (0,-2);
\coordinate (v1) at (0,1);
\coordinate (v2) at (1,-1);
\coordinate (v3) at (-1,-1);
\coordinate (r1) at (0,-1);
\coordinate (r2) at (0,2);

\draw [ultra thick, red] (e1) -- (v1);
\draw [ultra thick, red] (e1) -- (v2);
\draw [ultra thick, red] (e1) -- (v3);
\draw [ultra thick, red] (e2) -- (v2);
\draw [ultra thick, red] (e2) -- (v3);

\foreach \x in {(e1), (e2)}
{
\draw [green!50!black, fill=green!50!black] \x circle  (3pt);
}

\foreach \x in {(v1), (v2), (v3)}
{
\draw [blue, fill=blue] \x circle  (3pt);
}
\end{tikzpicture}
\begin{tikzpicture}[scale = 1.5]
\coordinate (e1) at (0,0);
\coordinate (e2) at (0,-2);
\coordinate (v1) at (0,1);
\coordinate (v2) at (1,-1);
\coordinate (v3) at (-1,-1);
\coordinate (r1) at (0,-1);
\coordinate (r2) at (0,2);

\coordinate (e1a) at ($ (e1) + (30:0.3) $);
\coordinate (e1b) at ($ (e1) + (150:0.3) $);
\coordinate (e1c) at ($ (e1) + (270:0.3) $);

\coordinate (e2a) at ($ (e2) + (90:0.3) $);
\coordinate (e2b) at ($ (e2) + (270:0.3) $);

\coordinate (v1a) at ($ (v1) + (90:0.3) $);

\coordinate (v2a) at ($ (v2) + (0:0.3) $);
\coordinate (v2b) at ($ (v2) + (180:0.3) $);

\coordinate (v3a) at ($ (v3) + (0:0.3) $);
\coordinate (v3b) at ($ (v3) + (180:0.3) $);

\draw [draw=none, fill=black!10!white]
(v1a) 
to [out = 180, in = 135] (e1a)
to (e1b)
to [out=45, in=0] (v1a)
to (e1b);

\draw [draw=none, fill=black!10!white]
(e1b)
to [out = 225, in = 90] (v3a)
to (v3b)
to [out = 90, in = 180] (e1c)
to (e1b);

\draw [draw=none, fill=black!10!white]
(v3b)
to [out = 270, in = 180] (e2a)
to (e2b)
to [out = 180, in = 270] (v3a)
to (v3b);

\draw [draw=none, fill=black!10!white]
(e2b)
to [out = 0, in = 270] (v2b)
to [out = 90, in = 315] (e1a)
to (e1c)
to [out = 0, in = 90] (v2a)
to [out = 270, in = 0] (e2a)
to (e2b);

\draw [draw=none, fill=black!10!white]
(e1a) to (e1b) to (e1c) to (e1a);

\draw [draw=none, fill=black!10!white]
(e2a) to (e2b) to (e2a);

\draw [draw=none, fill=black!10!white]
(v2a) to (v2b) to (v2a);

\draw [draw=none, fill=black!10!white]
(v3a) to (v3b) to (v3a);

\draw [ultra thick, red] (e1) -- (v1);
\draw [ultra thick, red] (e1) -- (v2);
\draw [ultra thick, red] (e1) -- (v3);
\draw [ultra thick, red] (e2) -- (v2);
\draw [ultra thick, red] (e2) -- (v3);

\begin{knot}[consider self intersections, 
clip width=2,
flip crossing=2,
flip crossing=4
]
\strand [thick]
(v1a)
to [out = 180, in = 135] (e1a)
to [out = 315, in = 90] (v2b)
to [out = 270, in = 0] (e2b)
to [out = 180, in = 270] (v3a)
to [out = 90=, in = 225] (e1b)
to [out = 45, in = 0] (v1a);
\strand[thick]
(v2a)
to [out = 90, in = 0] (e1c)
to [out = 180, in = 90] (v3b)
to [out = 270, in = 180] (e2a)
to [out = 0, in = 270] (v2a);
\end{knot}

\foreach \x in {(e1), (e2)}
{
\draw [green!50!black, fill=green!50!black] \x circle  (3pt);
}

\foreach \x in {(v1), (v2), (v3)}
{
\draw [blue, fill=blue] \x circle  (3pt);
}


\end{tikzpicture}
\begin{tikzpicture}[scale = 1.5]
\coordinate (e1) at (0,0);
\coordinate (e2) at (0,-2);
\coordinate (v1) at (0,1);
\coordinate (v2) at (1,-1);
\coordinate (v3) at (-1,-1);
\coordinate (r1) at (0,-1);
\coordinate (r2) at (0,2);

\coordinate (e1a) at ($ (e1) + (30:0.3) $);
\coordinate (e1b) at ($ (e1) + (150:0.3) $);
\coordinate (e1c) at ($ (e1) + (270:0.3) $);

\coordinate (e2a) at ($ (e2) + (90:0.3) $);
\coordinate (e2b) at ($ (e2) + (270:0.3) $);

\coordinate (v1a) at ($ (v1) + (90:0.3) $);

\coordinate (v2a) at ($ (v2) + (0:0.3) $);
\coordinate (v2b) at ($ (v2) + (180:0.3) $);

\coordinate (v3a) at ($ (v3) + (0:0.3) $);
\coordinate (v3b) at ($ (v3) + (180:0.3) $);

\draw [draw=none, fill=black!10!white]
(v1a) 
to [out = 180, in = 135] (e1a)
to (e1b)
to [out=45, in=0] (v1a)
to (e1b);

\draw [draw=none, fill=black!10!white]
(e1b)
to [out = 225, in = 90] (v3a)
to (v3b)
to [out = 90, in = 180] (e1c)
to (e1b);

\draw [draw=none, fill=black!10!white]
(v3b)
to [out = 270, in = 180] (e2a)
to (e2b)
to [out = 180, in = 270] (v3a)
to (v3b);

\draw [draw=none, fill=black!10!white]
(e2b)
to [out = 0, in = 270] (v2b)
to [out = 90, in = 315] (e1a)
to (e1c)
to [out = 0, in = 90] (v2a)
to [out = 270, in = 0] (e2a)
to (e2b);

\draw [draw=none, fill=black!10!white]
(e1a) to (e1b) to (e1c) to (e1a);

\draw [draw=none, fill=black!10!white]
(e2a) to (e2b) to (e2a);

\draw [draw=none, fill=black!10!white]
(v2a) to (v2b) to (v2a);

\draw [draw=none, fill=black!10!white]
(v3a) to (v3b) to (v3a);

\begin{knot}[consider self intersections, 
clip width=2,
flip crossing=2,
flip crossing=4
]
\strand [thick]
(v1a)
to [out = 180, in = 135] (e1a)
to [out = 315, in = 90] (v2b)
to [out = 270, in = 0] (e2b)
to [out = 180, in = 270] (v3a)
to [out = 90=, in = 225] (e1b)
to [out = 45, in = 0] (v1a);
\strand[thick]
(v2a)
to [out = 90, in = 0] (e1c)
to [out = 180, in = 90] (v3b)
to [out = 270, in = 180] (e2a)
to [out = 0, in = 270] (v2a);
\end{knot}

\end{tikzpicture}
\end{center}
\caption{The median construction for the graph of figure \ref{fig:graph_G}, the red graph of the trinity of figure \ref{fig:easy_trinity}. Left: the graph. Centre: the construction. Right: the link and Siefert surface.}
\label{fig:easy_median_construction}
\end{figure}
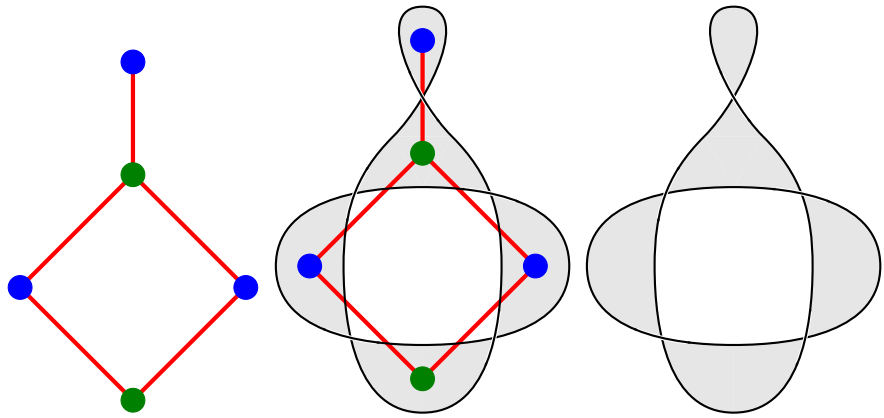

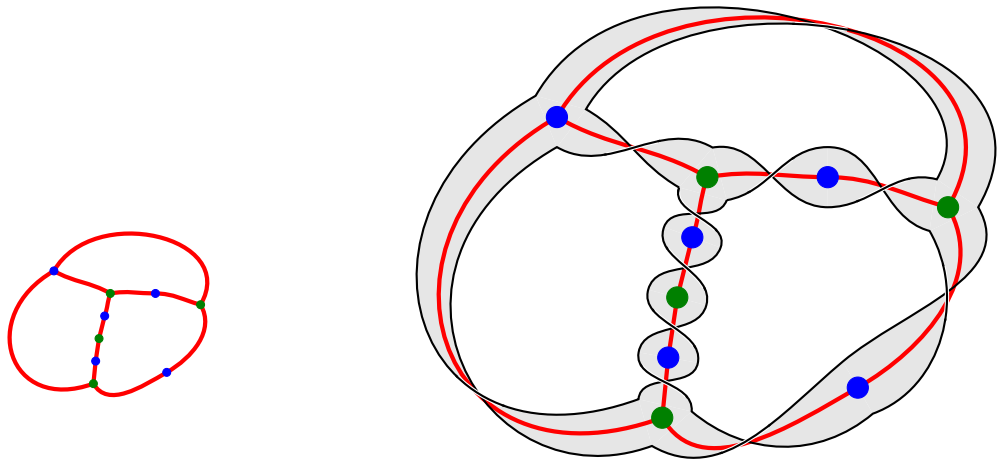
\begin{figure}
\begin{center}
\begin{tikzpicture}[scale = 0.15]
\coordinate (r3) at (1,6);
\coordinate (e1) at (-1,4); 
\coordinate (v1) at (-1.5, 2);
\coordinate (e2) at (-2,0); 
\coordinate (v2) at (-2.3,-2); 
\coordinate (e3) at (-2.5,-4); 
\coordinate (r0) at (-3,-7); 
\coordinate (v4) at (3,4);
\coordinate (r2) at (2.5,0);
\coordinate (v3) at (4,-3);
\coordinate (e0) at (7,3);
\coordinate (v0) at (-6,6); 
\coordinate (r1) at (-6, 0); 

\draw [ultra thick, draw=none] (r3) to[out=210,in=70] (e1); 
\draw [ultra thick, red] (e1) to [out = 250, in = 70] (v1); 
\draw [ultra thick, red] (v1) to (e2); 
\draw [ultra thick, red] (e2) to (v2); 
\draw [ultra thick, red] (v2) to (e3); 
\draw [ultra thick, draw=none] (e3) to [out=270, in=90] (r0); 
\draw [ultra thick, draw=none] (r0)
.. controls ($ (r0) + (270:8) $) and ($ (e0) + (345:10) $) .. (e0); 
\draw [ultra thick, red] (e0) to [out=165, in=0] (v4); 
\draw [ultra thick, red] (v4) to [out=180, in=10] (e1); 
\draw [ultra thick, draw=none] (e1) to [out=190, in=90] (r1); 
\draw [ultra thick, draw=none] (r1) to [out=270, in=170] (v2); 
\draw [ultra thick, draw=none] (v2) to [out=350, in=210] (r2); 
\draw [ultra thick, draw=none] (r2) to [out=30, in=300] (v4); 
\draw [ultra thick, draw=none] (v4) to [out=120, in=300] (r3); 
\draw [ultra thick, draw=none] (r3) to [out=120, in=30] (v0); 
\draw [ultra thick, red] (v0)
.. controls ($ (v0) + (210:8) $) and ($ (e3) + (200:8) $) .. (e3); 
\draw [ultra thick, draw=none] (e3) to [out=20, in=270] (r2); 
\draw [ultra thick, draw=none] (r2) to [out=90, in=330] (e1); 
\draw [ultra thick, red] (e1) to [out=150, in=330] (v0); 
\draw [ultra thick, draw=none] (v0)
.. controls ($ (v0) + (150:6) $) and ($ (r0) + (180:15) $) .. (r0); 
\draw [ultra thick, draw=none] (r0) to [out=0, in=270] (v3); 
\draw [ultra thick, draw=none] (v3) to [out=90, in=300] (r2); 
\draw [ultra thick, draw=none] (r2) to [out=120, in=0] (v1); 
\draw [ultra thick, draw=none] (v1) to [out=180, in=60] (r1); 
\draw [ultra thick, draw=none] (r1)
.. controls ($ (r1) + (240:5) $) and ($ (e3) + (120:2) $) .. (e3); 
\draw [ultra thick, red] (e3) to [out=300, in=210] (v3); 
\draw [ultra thick, red] (v3) to [out=30, in=300] (e0); 
\draw [ultra thick, draw=none] (e0) to [out=120, in=30] (r3); 

\draw [ultra thick, red] (e0) 
.. controls ($ (e0) + (60:7) $) and ($ (v0) + (60:7) $) .. (v0); 
\draw [ultra thick, draw=none] (v0)
.. controls ($ (v0) + (240:3) $) and ($ (r1) + (150:3) $) .. (r1); 
\draw [ultra thick, draw=none] (r1) to [out=330, in=180] (e2); 
\draw [ultra thick, draw=none] (e2) to [out=0, in=180] (r2); 
\draw [ultra thick, draw=none] (r2) to [out=0, in=240] (e0); 


\foreach \x/\word in {(e0)/e0, (e1)/e1, (e2)/e2, (e3)/e3}
{
\draw [green!50!black, fill=green!50!black] \x circle  (10pt);
}

\foreach \x/\word in {(v0)/v0, (v1)/v1, (v2)/v2, (v3)/v3, (v4)/v4}
{
\draw [blue, fill=blue] \x circle (10pt);
}

\end{tikzpicture}
\begin{tikzpicture}[scale = 0.4]
\coordinate (r3) at (1,6);
\coordinate (e1) at (-1,4); 
\coordinate (v1) at (-1.5, 2);
\coordinate (e2) at (-2,0); 
\coordinate (v2) at (-2.3,-2); 
\coordinate (e3) at (-2.5,-4); 
\coordinate (r0) at (-3,-7); 
\coordinate (v4) at (3,4);
\coordinate (r2) at (2.5,0);
\coordinate (v3) at (4,-3);
\coordinate (e0) at (7,3);
\coordinate (v0) at (-6,6); 
\coordinate (r1) at (-6, 0); 

\coordinate (v0a) at ($ (v0) + (15:1) $);
\coordinate (v0b) at ($ (v0) + (135:1) $);
\coordinate (v0c) at ($ (v0) + (270:1) $);

\coordinate (v1a) at ($ (v1) + (340:1) $);
\coordinate (v1b) at ($ (v1) + (160:1) $);

\coordinate (v2a) at ($ (v2) + (350:1) $);
\coordinate (v2b) at ($ (v2) + (170:1) $);

\coordinate (v3a) at ($ (v3) + (120:1) $);
\coordinate (v3b) at ($ (v3) + (300:1) $);

\coordinate (v4a) at ($ (v4) + (90:1) $);
\coordinate (v4b) at ($ (v4) + (270:1) $);

\coordinate (e0a) at ($ (e0) + (0:1) $);
\coordinate (e0b) at ($ (e0) + (112:1) $);
\coordinate (e0c) at ($ (e0) + (232:1) $);

\coordinate (e1a) at ($ (e1) + (80:1) $);
\coordinate (e1b) at ($ (e1) + (200:1) $);
\coordinate (e1c) at ($ (e1) + (310:1) $);

\coordinate (e2a) at ($ (e2) + (345:1) $);
\coordinate (e2b) at ($ (e2) + (165:1) $);

\coordinate (e3a) at ($ (e3) + (12:1) $);
\coordinate (e3b) at ($ (e3) + (142:1) $);
\coordinate (e3c) at ($ (e3) + (250:1) $);

\draw [draw=none, fill=black!10!white]
(e1b) 
to [out = 250, in = 70] (v1a)
to (v1b)
to [out=70, in=250] (e1c)
to (e1b);

\draw [draw=none, fill=black!10!white]
(v1a)
to [out=250, in=75] (e2b)
to (e2a)
to [out=75, in=250] (v1b)
to (v1a);

\draw [draw=none, fill=black!10!white]
(e2b)
to [out=255, in=80] (v2a)
to (v2b)
to [out=80, in=255] (e2a)
to (e2b);

\draw [draw=none, fill=black!10!white]
(v2a)
to [out=260, in=85] (e3b) 
to (e3a)
to [out=85, in=260] (v2b) 
to (v2a);

\draw [draw=none, fill=black!10!white]
(e3b)
.. controls ($ (e3b) + (200:8) $) and ($ (v0b) + (210:8) $) .. (v0b)
to (v0c)
.. controls ($ (v0c) + (210:8) $) and ($ (e3c) + (200:6) $) .. (e3c)
to (e3b);

\draw [draw=none, fill=black!10!white]
(v0b)
.. controls ($ (v0b) + (60:7) $) and ($ (e0b) + (60:5) $) .. (e0b) 
to (e0a)
.. controls ($ (e0a) + (60:7) $) and ($ (v0a) + (60:6) $) .. (v0a)
to (v0b);

\draw [draw=none, fill=black!10!white]
(e0b)
to [out=165, in=0] (v4b)
to (v4a)
to [out=0, in=165] (e0c)
to (e0c);

\draw [draw=none, fill=black!10!white]
(v4b)
to [out=180, in=10] (e1a)
to (e1c)
to [out=10, in=180] (v4a)
to (v4b);

\draw [draw=none, fill=black!10!white]
(e1a)
to [out=150, in=330] (v0c)
to (v0a)
to [out=330, in=150] (e1b)
to (e1a);

\draw [draw=none, fill=black!10!white]
(e3c)
to [out=330, in=220] (v3a)
to (v3b)
to [out=220, in=320] (e3a)
to (e3c);

\draw [draw=none, fill=black!10!white]
(v3a)
to [out=40, in=300] (e0a)
to (e0c)
to [out=300, in=20] (v3b)
to (v3a);

\draw [draw=none, fill=black!10!white]
(v0a) to (v0b) to (v0c) to (v0a);

\draw [draw=none, fill=black!10!white]
(e0a) to (e0b) to (e0c) to (e0a);

\draw [draw=none, fill=black!10!white]
(e1a) to (e1b) to (e1c) to (e1a);

\draw [draw=none, fill=black!10!white]
(e3a) to (e3b) to (e3c) to (e3a);

\draw [ultra thick, red] (e1) to [out = 250, in = 70] (v1); 
\draw [ultra thick, red] (v1) to (e2); 
\draw [ultra thick, red] (e2) to (v2); 
\draw [ultra thick, red] (v2) to (e3); 
\draw [ultra thick, red] (e0) to [out=165, in=0] (v4); 
\draw [ultra thick, red] (v4) to [out=180, in=10] (e1); 
\draw [ultra thick, red] (v0)
.. controls ($ (v0) + (210:8) $) and ($ (e3) + (200:8) $) .. (e3); 
\draw [ultra thick, red] (e1) to [out=150, in=330] (v0); 
\draw [ultra thick, red] (e3) to [out=300, in=210] (v3); 
\draw [ultra thick, red] (v3) to [out=30, in=300] (e0); 
\draw [ultra thick, red] (e0) 
.. controls ($ (e0) + (60:7) $) and ($ (v0) + (60:7) $) .. (v0); 

\begin{knot}[consider self intersections, 
clip width=2,
flip crossing=2,
flip crossing=4,
flip crossing=6,
flip crossing=8,
flip crossing=11
]
\strand [thick]
(e1b)
to [out = 250, in = 70] (v1a)
to [out=250, in=75] (e2b) 
to [out=255, in=80] (v2a) 
to [out=260, in=85] (e3b) 
.. controls ($ (e3b) + (200:8) $) and ($ (v0b) + (210:8) $) .. (v0b)
.. controls ($ (v0b) + (60:7) $) and ($ (e0b) + (60:5) $) .. (e0b) 
to [out=165, in=0] (v4b)
to [out=180, in=10] (e1a)
to [out=150, in=330] (v0c)
.. controls ($ (v0c) + (210:8) $) and ($ (e3c) + (200:6) $) .. (e3c) 
to [out=330, in=220] (v3a) 
to [out=40, in=300] (e0a) 
.. controls ($ (e0a) + (60:7) $) and ($ (v0a) + (60:6) $) .. (v0a) 
to [out=330, in=150] (e1b);
\strand[thick]
(e1c)
to [out=10, in=180] (v4a)
to [out=0, in=165] (e0c)
to [out=300, in=20] (v3b) 
to [out=220, in=320] (e3a) 
to [out=85, in=260] (v2b) 
to [out=80, in=255] (e2a) 
to [out=75, in=250] (v1b) 
to [out=70, in=250] (e1c); 
\end{knot}


\foreach \x/\word in {(e0)/e0, (e1)/e1, (e2)/e2, (e3)/e3}
{
\draw [green!50!black, fill=green!50!black] \x circle  (10pt);
}

\foreach \x/\word in {(v0)/v0, (v1)/v1, (v2)/v2, (v3)/v3, (v4)/v4}
{
\draw [blue, fill=blue] \x circle (10pt);
}


\end{tikzpicture}
\end{center}
\caption{The median construction for the red graph of the trinity of figure \ref{fig:trinity}.}
\label{fig:median_construction}
\end{figure}

Viewing the plane $\R^2$ from above, one side of $F_G$ faces up at each vertex of $G$. Because of the half twist along an edge, if $v_0, v_1$ are vertices connected by an edge then opposite sides of $F_G$ face up at $v_0$ and $v_1$. So $F_G$ is orientable if and only if $G$ is bipartite. 

Thus, if $G$ is a plane bipartite graph, then the median construction yields an orientable surface $F_G$ bounding a link $L_G$. After orienting $F_G$, $L_G$ can be oriented as the boundary of $F_G$. Then $L_G$ is an oriented link with Seifert surface $F_G$. Indeed, $F_G$ is obtained by applying Seifert's algorithm to $L_G$.

The diagram obtained for $L_G$ is a \emph{special} alternating link diagram. A special alternating diagram is one in which every Seifert circle is innermost. Indeed, the Seifert circles of $L_G$ run around the vertices of $G$ in $S^2$, so are all innermost; there is no nesting. 

It is known that applying Seifert's algorithm to an alternating diagram yields a minimal genus Seifert surface \cite{Crowell59, Gabai86_genera, Murasugi58}. Further, $L_G$ is non-split. Thus the median construction on a plane bipartite graph yields a minimal genus Seifert surface $F_G$ for a non-split special alternating link $L_G$. A converse is also true: any minimal genus Seifert surface of a non-split prime special alternating link arises from the median construction on a connected bipartite plane graph \cite{Banks11, Hirasawa-Sakuma96}.

\subsection{The HOMFLY-PT polynomial and the root polytope}

K\'{a}lm\'{a}n--Murakami showed in \cite{K-Murakami} that the link $L_G$, in particular its HOMFLY-PT polynomial, is closely related to the root polytope $\QQ_G$ of the plane bipartite graph $G$.

Denote the HOMFLY-PT polynomial of an oriented link $L$ by $P_L (v,z)$. 
Roughly, K\'{a}lm\'{a}n--Murakami showed that the ``top"-degree terms of $P_{L_G} (v,z)$ correspond to the \emph{$h$-vector} of $\QQ_G$.

The $h$-vector is one of several vectors associated to a simplicial complex, encoding its combinatorial structure. The root polytope $\QQ_G$ is not itself a simplicial complex, just a polytope. But a \emph{triangulation} of $\QQ_G$ provides it with the structure of a simplicial complex. And we have seen above how the root polytope $\QQ_G$ has certain natural triangulations, deriving from spanning trees and arborescences. According to theorem \ref{thm:KM_root_polytope_triangulation}, the root polytopes of spanning trees of $G$ dual to arborescences of $G^*$ triangulate $\QQ_G$. This triangulation of $\QQ_G$ is a simplicial complex, and hence has an $h$-vector.

To define the $h$-vector, we first define the \emph{$f$-vector}. The $f$-vector of a $d$-dimensional simplicial complex can be given as a polynomial
\[
f(y) = y^{d+1} + f_0 y^d + f_1 y^{d-1} + \cdots + f_{d-1} y + f_d,
\]
where $f_k$ is the number of $k$-dimensional simplices in the complex. The $h$-vector is then defined by $h(x) = f(x-1)$. So the $f$- and $h$-vectors encode numbers of simplices of each dimension. 

K\'{a}lm\'{a}n--Murakami proved that the $h$-vector of this natural triangulation of $\QQ_G$ is essentially the ``top" of $P_{L_G} (v,z)$. Precisely, we define the \emph{top} of $P_L (v,z)$ to be the polynomial in $v$ consisting of terms that contribute to the leading term of $P_L (1,z)$ . (That is, the top of $P_L (v,z)$ consists of those terms with highest $z$-degree which don't cancel out upon setting $v=1$, with the $z$'s then removed.) Setting $v=1$ in $P_L (v,z)$ is a natural substitution to make: it yields the Alexander-Conway polynomial.

\begin{thm}[K\'{a}lm\'{a}n--Murakami \cite{K-Murakami} thm. 1.3]
\label{thm:HOMFLY_h-vector}
Let $G$ be a connected plane bipartite graph with $V$ vertices and $E$ edges. Then the top of the HOMFLY-PT polynomial $P_{L_G} (v,z)$ is
\[
v^{E+V-1} h(v^{-2}),
\]
where $h$ is the $h$-vector of a triangulation of the root polytope $\QQ_G$ described in theorem \ref{thm:KM_root_polytope_triangulation}.
\end{thm}

Thus, fixing a root vertex of $G^*$, and taking spanning trees of $G$ dual to arborescences of $G^*$, their root polytopes triangulate $\QQ_G$, and the $h$-vector of this triangulation is, up to a shift in powers of $v$, the top of the HOMFLY-PT polynomial $P_{L_G} (v,z)$.

(It follows that the $h$-vector of the triangulation of $\QQ_G$ does not depend on the choice of root vertex of $G^*$. For each such choice, the corresponding triangulation of $\QQ_G$ has the same $h$-vector.)

We illustrate with the example of the bipartite plane graph $G$ of figure \ref{fig:graph_G}. The median construction on $G$ yields a torus link $L_G$ of $2$ components, as shown in figure \ref{fig:easy_median_construction}. (The degree 1 vertex of  $G$ creates a loop in $L_G$, which can be removed by a Reidemeister I move.) The HOMFLY-PT polynomial of $L_G$ is easily calculated as
\[
P_{L_G} (v,z) = v^4 + v^3 z + vz.
\]
Setting $v=1$ yields $1+2z$, and the leading term $2z$ arises from the terms $v^3 z + v z$ of $P_{L_G} (v,z)$. Thus 
\[
\text{Top of } P_{L_G} = v^3 + v.
\]

On the other hand, as discussed above in section \ref{sec:arborescences_root_polytope}, the triangulation of $\QQ_G$ given by theorem \ref{thm:KM_root_polytope_triangulation} is the triangulation $\QQ_{T_1} \cup \QQ_{T_4}$, where $T_1, T_4$ are as in figure \ref{fig:spanning_trees_in_G}. This triangulation is shown in figure \ref{fig:polytope_slices} (left). This is for the standard ``outer" choice of root vertex of $G^*$.

We observe that this triangulation has $5$ vertices, $9$ edges, $7$ triangles and $2$ tetrahedra, so its $f$-vector is
\[
f(y) = y^4 + 5y^3 + 9y^2 + 7y + 2.
\]
Hence its $h$-vector is
\[
h(x) = f(x-1) = x^4 + x^3.
\]
Since $G$ has $5$ vertices and $5$ edges, we verify the claim of theorem 
\ref{thm:HOMFLY_h-vector}:
\[
v^{E+V-1} h(v^{-1} ) = v^9 (v^{-8} + v^{-6}) = v + v^3 = \text{Top of} P_{L_G}.
\]

Incidentally, as mentioned in section \ref{sec:arborescences_root_polytope}, if we take the other choice of root vertex for $G^*$, then we obtain the triangulation $\QQ_{T_2} \cup \QQ_{T_3}$ of $\QQ_G$, depicted in figure \ref{fig:polytope_slices} (right). This triangulation has the same $f$- and $h$-vectors, and yields the same result.

\section{Sutured manifolds and Floer homology}
\label{sec:sutured}

The discussion so far --- graphs, trees, links, polytopes --- now moves into 3-manifold topology and holomorphic curves: in particular, to sutured manifolds, and Floer homology.

\subsection{Sutured manifolds}
\label{sec:sutured_manifolds}

Roughly, for present purposes, a \emph{sutured surface} is a surface with some oriented curves drawn on its boundary, which behave nicely with respect to orientations; and a \emph{sutured 3-manifold} is a 3-manifold whose boundary is a sutured surface. 

More precisely, we define a \emph{sutured surface} $(S, \Gamma)$ to consist of a smooth closed oriented surface $S$, together with a smooth oriented 1-submanifold $\Gamma$ of $S$, satisfying the following condition: $S \backslash \Gamma = R_+ \sqcup R_-$, where $R_+, R_-$ are oriented subsurfaces of $S$, and $\partial R_+ = -\partial R_- = \Gamma$ as oriented 1-manifolds. Additionally, we require that $\Gamma$ has nonempty intersection with each component of $S$, i.e. every component of $S$ contains a curve of $\Gamma$. The curves of $\Gamma$ are called \emph{sutures}. 

Effectively the orientation condition on sutures just means that we can label the complementary regions of $\Gamma$ in $S$ as positive ($R_+$) or negative ($R_-$) such that, whenever we cross over a curve of $\Gamma$, we pass from a positive to a negative region, or vice versa. Note that the orientations on $S$ and $\Gamma$ determine $R_+$ and $R_-$ uniquely. See figure \ref{fig:sutured_nonsutured_surfaces}.

\begin{figure}
\begin{center}
\begin{tikzpicture}[scale=0.8]

\draw (-40mm,0mm) to[out=90,in=180] (0mm,20mm) to[out=0,in=90] (40mm,0mm);
\draw (-40mm,0mm) to[out=-90,in=180] (0mm,-20mm) to[out=0,in=-90] (40mm,0mm);

\draw (-8mm,0mm) to[out=40,in=180] (0mm,4mm) to[out=0,in=140] (8mm,0mm);
\draw (-8mm,0mm) to[out=-40,in=180] (0mm,-4mm) to[out=0,in=-140] (8mm,0mm);
\draw (-8mm,0mm) to[out=140,in=130] (-10mm,2mm);
\draw (8mm,0mm) to[out=40,in=50] (10mm,2mm);

\draw[red, thick, ->] (-30mm, 0mm) to[out=90,in=180] (0mm,15mm) to[out=0,in=90] (30mm,0mm);
\draw[red,thick,->] (30mm,0mm) to[out=-90,in=0] (0mm,-15mm) to[out=180,in=-90] (-30mm,0mm);

\draw[red, thick, <-] (-20mm, 0mm) to[out=90,in=180] (0mm,10mm) to[out=0,in=90] (20mm,0mm);
\draw[red,thick,<-] (20mm,0mm) to[out=-90,in=0] (0mm,-10mm) to[out=180,in=-90] (-20mm,0mm);

\draw (15mm,0mm) node {$+$};
\draw (25mm,0mm) node {$-$};
\draw (35mm,0mm) node {$+$};

\end{tikzpicture}
\hspace{2 cm}
\begin{tikzpicture}[scale=0.8]

\draw (-40mm,0mm) to[out=90,in=180] (0mm,20mm) to[out=0,in=90] (40mm,0mm);
\draw (-40mm,0mm) to[out=-90,in=180] (0mm,-20mm) to[out=0,in=-90] (40mm,0mm);

\draw (-8mm,0mm) to[out=40,in=180] (0mm,4mm) to[out=0,in=140] (8mm,0mm);
\draw (-8mm,0mm) to[out=-40,in=180] (0mm,-4mm) to[out=0,in=-140] (8mm,0mm);
\draw (-8mm,0mm) to[out=140,in=130] (-10mm,2mm);
\draw (8mm,0mm) to[out=40,in=50] (10mm,2mm);

\draw[red, thick, <-] (-25mm, 0mm) to[out=90,in=180] (0mm,12.5mm) to[out=0,in=90] (25mm,0mm);
\draw[red,thick,<-] (25mm,0mm) to[out=-90,in=0] (0mm,-12.5mm) to[out=180,in=-90] (-25mm,0mm);

\draw (15mm,0mm) node {$+?$};
\draw (35mm,0mm) node {$-?$};

\end{tikzpicture}
\end{center}
\caption{An example and a non-example of a sutured surface. Left: there are two sutures, and coherent orientations $+,-$ are assigned to the two annular complementary regions. Right: there is only one complementary region, which cannot be assigned a coherent orientation.}
\label{fig:sutured_nonsutured_surfaces}
\end{figure}
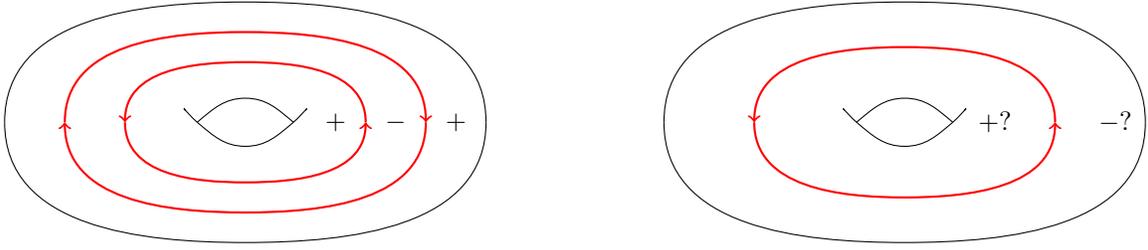

We can then define a \emph{sutured 3-manifold} $(M, \Gamma)$ to be a smooth oriented 3-manifold $M$, together with a smooth oriented 1-submanifold $\Gamma$ of $\partial M$, such that $(\partial M, \Gamma)$ is a sutured surface.

\subsection{From graph to sutured manifold}
\label{sec:graph_to_sutured_manifold}

We now show how to construct a sutured 3-manifold $(M_G, L_G)$ from a bipartite plane graph $G$. The construction is a slight extension of the median construction of section \ref{sec:median_construction}; indeed, the oriented link $L_G$ forms the set of sutures. 

Recall that $L_G$, and the Seifert surface $F_G$, are constructed by taking a neighbourhood of $G$ in $\R^2$, consisting of a disc around each vertex and a band along each edge, and inserting a negative half-twist in each band.

Let $N(G)$ be a regular neighbourhood of $G$ in $S^3$. We can take $N(G)$ so that the link $L_G$ is contained in its boundary $\partial N(G)$. Then we define $M_G = S^3 - N(G)$. So $M_G$ is a 3-manifold with boundary, and $L_G$ is an oriented 1-submanifold of $\partial M_G$. See figures \ref{fig:easy_M_G} and \ref{fig:M_G} for the construction of $(M_G, L_G)$ in the examples of figure \ref{fig:easy_median_construction}
 and \ref{fig:median_construction} respectively.

We can alternatively view $(M_G, L_G)$ as constructed by splitting $S^3$ open along the Seifert surface $F_G$. We see that $\partial M_G$ consists of two copies $F^+_G, F^-_G$ of $F_G$ meeting along $L_G$. Orienting $F^+_G, F^-_G$ via the boundary orientation of $M_G \subset S^3$, we have $\partial F^+_G = - \partial F^-_G = L_G$. Thus $(M_G, L_G)$ is a sutured 3-manifold, and the two copies $F^+_G, F^-_G$ of $F_G$ form the positive and negative subsurfaces $R_+, R_-$ respectively.

In our first example (figure \ref{fig:easy_M_G}), the boundary $\partial M_G = \partial N(G)$ is a torus, and in our picture, the manifold $M_G$ consists of everything in $S^3$ ``outside" the torus. As $S^3$ decomposes into two solid tori, $M_G$ is a solid torus with a pair of sutures along its boundary. In our second example (figure \ref{fig:M_G}), $\partial M_G$ is a genus $3$ closed orientable surface, and $M_G$ is everything ``outside" it in $S^3$. Since $\partial M_G$ splits $S^3$ into two genus-3 handlebodies, $M_G$ is a genus-3 handlebody.

In general, for a connected bipartite plane graph $G$, the manifold $M_G$ is a handlebody, of genus equal to the first Betti number $b_1 (G)$ of $G$.

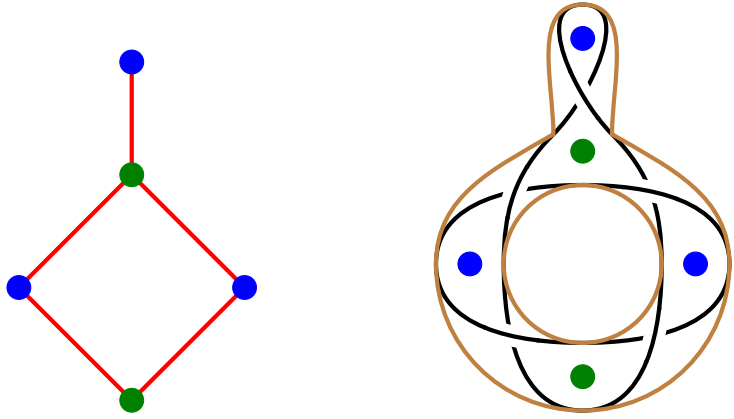
\begin{figure}
\begin{center}
\begin{tikzpicture}[scale = 1.5]
\coordinate (e1) at (0,0);
\coordinate (e2) at (0,-2);
\coordinate (v1) at (0,1);
\coordinate (v2) at (1,-1);
\coordinate (v3) at (-1,-1);
\coordinate (r1) at (0,-1);
\coordinate (r2) at (0,2);

\draw [ultra thick, red] (e1) -- (v1);
\draw [ultra thick, red] (e1) -- (v2);
\draw [ultra thick, red] (e1) -- (v3);
\draw [ultra thick, red] (e2) -- (v2);
\draw [ultra thick, red] (e2) -- (v3);

\foreach \x in {(e1), (e2)}
{
\draw [green!50!black, fill=green!50!black] \x circle  (3pt);
}

\foreach \x in {(v1), (v2), (v3)}
{
\draw [blue, fill=blue] \x circle  (3pt);
}
\end{tikzpicture}
\hspace{2 cm}
\begin{tikzpicture}[scale = 1.5]
\coordinate (e1) at (0,0);
\coordinate (e2) at (0,-2);
\coordinate (v1) at (0,1);
\coordinate (v2) at (1,-1);
\coordinate (v3) at (-1,-1);
\coordinate (r1) at (0,-1);
\coordinate (r2) at (0,2);

\coordinate (e1a) at ($ (e1) + (30:0.3) $);
\coordinate (e1b) at ($ (e1) + (150:0.3) $);
\coordinate (e1c) at ($ (e1) + (270:0.3) $);

\coordinate (e2a) at ($ (e2) + (90:0.3) $);
\coordinate (e2b) at ($ (e2) + (270:0.3) $);

\coordinate (v1a) at ($ (v1) + (90:0.3) $);

\coordinate (v2a) at ($ (v2) + (0:0.3) $);
\coordinate (v2b) at ($ (v2) + (180:0.3) $);

\coordinate (v3a) at ($ (v3) + (0:0.3) $);
\coordinate (v3b) at ($ (v3) + (180:0.3) $);


\begin{knot}[consider self intersections, 
clip width=5,
flip crossing=2,
flip crossing=4
]
\strand [ultra thick]
(v1a)
to [out = 180, in = 135] (e1a)
to [out = 315, in = 90] (v2b)
to [out = 270, in = 0] (e2b)
to [out = 180, in = 270] (v3a)
to [out = 90=, in = 225] (e1b)
to [out = 45, in = 0] (v1a);
\strand[ultra thick]
(v2a)
to [out = 90, in = 0] (e1c)
to [out = 180, in = 90] (v3b)
to [out = 270, in = 180] (e2a)
to [out = 0, in = 270] (v2a);
\end{knot}

\foreach \x in {(e1), (e2)}
{
\draw [green!50!black, fill=green!50!black] \x circle  (3pt);
}

\foreach \x in {(v1), (v2), (v3)}
{
\draw [blue, fill=blue] \x circle  (3pt);
}


\draw [ultra thick, brown] (v1a) 
to [out = 180, in = 90] (e1b)
to [out = 210, in = 90] (v3b)
to [out = 270, in = 180] (e2b)
to [out = 0, in = 270] (v2a)
to [out = 90, in = 330] (e1a)
to [out = 90, in = 0] (v1a);

\draw [ultra thick, brown] (e1c)
to [out = 180, in = 90] (v3a)
to [out = 270, in = 180] (e2a)
to [out = 0, in = 270] (v2b)
to [out = 90, in = 0] (e1c);

\end{tikzpicture}
\end{center}
\caption{Left: The bipartite plane graph $G$ from figure \ref{fig:graph_G}, whose median construction was seen in figure \ref{fig:easy_median_construction}. Right: The sutured manifold $(M_G, L_G)$. The boundary $\partial M_G = \partial N(G)$, indicated in brown, is a torus. The sutures $L_G$ are drawn in black, and form a torus link.}
\label{fig:easy_M_G}
\end{figure}

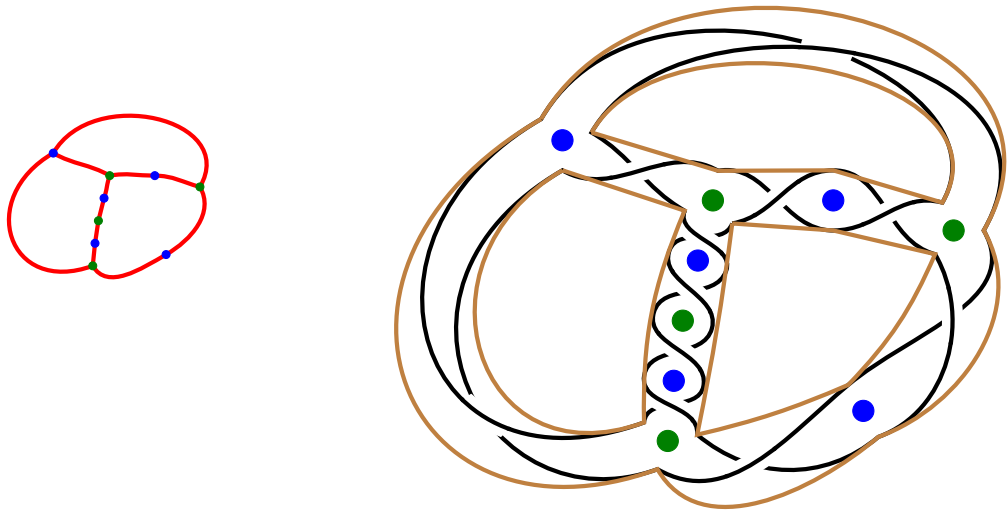
\begin{figure}
\begin{center}
\begin{minipage}{0.3\textwidth}
\begin{tikzpicture}[scale = 0.15]
\coordinate (r3) at (1,6);
\coordinate (e1) at (-1,4); 
\coordinate (v1) at (-1.5, 2);
\coordinate (e2) at (-2,0); 
\coordinate (v2) at (-2.3,-2); 
\coordinate (e3) at (-2.5,-4); 
\coordinate (r0) at (-3,-7); 
\coordinate (v4) at (3,4);
\coordinate (r2) at (2.5,0);
\coordinate (v3) at (4,-3);
\coordinate (e0) at (7,3);
\coordinate (v0) at (-6,6); 
\coordinate (r1) at (-6, 0); 

\draw [ultra thick, draw=none] (r3) to[out=210,in=70] (e1); 
\draw [ultra thick, red] (e1) to [out = 250, in = 70] (v1); 
\draw [ultra thick, red] (v1) to (e2); 
\draw [ultra thick, red] (e2) to (v2); 
\draw [ultra thick, red] (v2) to (e3); 
\draw [ultra thick, draw=none] (e3) to [out=270, in=90] (r0); 
\draw [ultra thick, draw=none] (r0)
.. controls ($ (r0) + (270:8) $) and ($ (e0) + (345:10) $) .. (e0); 
\draw [ultra thick, red] (e0) to [out=165, in=0] (v4); 
\draw [ultra thick, red] (v4) to [out=180, in=10] (e1); 
\draw [ultra thick, draw=none] (e1) to [out=190, in=90] (r1); 
\draw [ultra thick, draw=none] (r1) to [out=270, in=170] (v2); 
\draw [ultra thick, draw=none] (v2) to [out=350, in=210] (r2); 
\draw [ultra thick, draw=none] (r2) to [out=30, in=300] (v4); 
\draw [ultra thick, draw=none] (v4) to [out=120, in=300] (r3); 
\draw [ultra thick, draw=none] (r3) to [out=120, in=30] (v0); 
\draw [ultra thick, red] (v0)
.. controls ($ (v0) + (210:8) $) and ($ (e3) + (200:8) $) .. (e3); 
\draw [ultra thick, draw=none] (e3) to [out=20, in=270] (r2); 
\draw [ultra thick, draw=none] (r2) to [out=90, in=330] (e1); 
\draw [ultra thick, red] (e1) to [out=150, in=330] (v0); 
\draw [ultra thick, draw=none] (v0)
.. controls ($ (v0) + (150:6) $) and ($ (r0) + (180:15) $) .. (r0); 
\draw [ultra thick, draw=none] (r0) to [out=0, in=270] (v3); 
\draw [ultra thick, draw=none] (v3) to [out=90, in=300] (r2); 
\draw [ultra thick, draw=none] (r2) to [out=120, in=0] (v1); 
\draw [ultra thick, draw=none] (v1) to [out=180, in=60] (r1); 
\draw [ultra thick, draw=none] (r1)
.. controls ($ (r1) + (240:5) $) and ($ (e3) + (120:2) $) .. (e3); 
\draw [ultra thick, red] (e3) to [out=300, in=210] (v3); 
\draw [ultra thick, red] (v3) to [out=30, in=300] (e0); 
\draw [ultra thick, draw=none] (e0) to [out=120, in=30] (r3); 

\draw [ultra thick, red] (e0) 
.. controls ($ (e0) + (60:7) $) and ($ (v0) + (60:7) $) .. (v0); 
\draw [ultra thick, draw=none] (v0)
.. controls ($ (v0) + (240:3) $) and ($ (r1) + (150:3) $) .. (r1); 
\draw [ultra thick, draw=none] (r1) to [out=330, in=180] (e2); 
\draw [ultra thick, draw=none] (e2) to [out=0, in=180] (r2); 
\draw [ultra thick, draw=none] (r2) to [out=0, in=240] (e0); 


\foreach \x/\word in {(e0)/e0, (e1)/e1, (e2)/e2, (e3)/e3}
{
\draw [green!50!black, fill=green!50!black] \x circle  (10pt);
}

\foreach \x/\word in {(v0)/v0, (v1)/v1, (v2)/v2, (v3)/v3, (v4)/v4}
{
\draw [blue, fill=blue] \x circle (10pt);
}

\end{tikzpicture}
\end{minipage}
\begin{minipage}{0.6\textwidth}
\begin{tikzpicture}[scale = 0.4]
\coordinate (r3) at (1,6);
\coordinate (e1) at (-1,4); 
\coordinate (v1) at (-1.5, 2);
\coordinate (e2) at (-2,0); 
\coordinate (v2) at (-2.3,-2); 
\coordinate (e3) at (-2.5,-4); 
\coordinate (r0) at (-3,-7); 
\coordinate (v4) at (3,4);
\coordinate (r2) at (2.5,0);
\coordinate (v3) at (4,-3);
\coordinate (e0) at (7,3);
\coordinate (v0) at (-6,6); 
\coordinate (r1) at (-6, 0); 

\coordinate (v0a) at ($ (v0) + (15:1) $);
\coordinate (v0b) at ($ (v0) + (135:1) $);
\coordinate (v0c) at ($ (v0) + (270:1) $);

\coordinate (v1a) at ($ (v1) + (340:1) $);
\coordinate (v1b) at ($ (v1) + (160:1) $);

\coordinate (v2a) at ($ (v2) + (350:1) $);
\coordinate (v2b) at ($ (v2) + (170:1) $);

\coordinate (v3a) at ($ (v3) + (120:1) $);
\coordinate (v3b) at ($ (v3) + (300:1) $);

\coordinate (v4a) at ($ (v4) + (90:1) $);
\coordinate (v4b) at ($ (v4) + (270:1) $);

\coordinate (e0a) at ($ (e0) + (0:1) $);
\coordinate (e0b) at ($ (e0) + (112:1) $);
\coordinate (e0c) at ($ (e0) + (232:1) $);

\coordinate (e1a) at ($ (e1) + (80:1) $);
\coordinate (e1b) at ($ (e1) + (200:1) $);
\coordinate (e1c) at ($ (e1) + (310:1) $);

\coordinate (e2a) at ($ (e2) + (345:1) $);
\coordinate (e2b) at ($ (e2) + (165:1) $);

\coordinate (e3a) at ($ (e3) + (12:1) $);
\coordinate (e3b) at ($ (e3) + (142:1) $);
\coordinate (e3c) at ($ (e3) + (250:1) $);


\begin{knot}[consider self intersections, 
clip width=5,
flip crossing=2,
flip crossing=4,
flip crossing=6,
flip crossing=8,
flip crossing=11
]
\strand [ultra thick]
(e1b)
to [out = 250, in = 70] (v1a)
to [out=250, in=75] (e2b) 
to [out=255, in=80] (v2a) 
to [out=260, in=85] (e3b) 
.. controls ($ (e3b) + (200:8) $) and ($ (v0b) + (210:8) $) .. (v0b)
.. controls ($ (v0b) + (60:7) $) and ($ (e0b) + (60:5) $) .. (e0b) 
to [out=165, in=0] (v4b)
to [out=180, in=10] (e1a)
to [out=150, in=330] (v0c)
.. controls ($ (v0c) + (210:8) $) and ($ (e3c) + (200:6) $) .. (e3c) 
to [out=330, in=220] (v3a) 
to [out=40, in=300] (e0a) 
.. controls ($ (e0a) + (60:7) $) and ($ (v0a) + (60:6) $) .. (v0a) 
to [out=330, in=150] (e1b);
\strand[ultra thick]
(e1c)
to [out=10, in=180] (v4a)
to [out=0, in=165] (e0c)
to [out=300, in=20] (v3b) 
to [out=220, in=320] (e3a) 
to [out=85, in=260] (v2b) 
to [out=80, in=255] (e2a) 
to [out=75, in=250] (v1b) 
to [out=70, in=250] (e1c); 
\end{knot}








\foreach \x/\word in {(e0)/e0, (e1)/e1, (e2)/e2, (e3)/e3}
{
\draw [green!50!black, fill=green!50!black] \x circle  (10pt);
}

\foreach \x/\word in {(v0)/v0, (v1)/v1, (v2)/v2, (v3)/v3, (v4)/v4}
{
\draw [blue, fill=blue] \x circle (10pt);
}

\foreach \x/\word in {(v0a)/a, (v0b)/b, (v0c)/c, (v1a)/a, (v1b)/b, (v2a)/a, (v2b)/b, (v3a)/a, (v3b)/b, (v4a)/a, (v4b)/b, (e0a)/a, (e0b)/b, (e0c)/c, (e1a)/a, (e1b)/b, (e1c)/c, (e2a)/a, (e2b)/b, (e3a)/a, (e3b)/b, (e3c)/c}
{
}


\draw [ultra thick, brown] (v0c) 
.. controls ($ (v0c) + (210:6) $) and ($ (e3b) + (200:6) $) .. (e3b)
to [bend left=12] (e1b)
to (v0c);

\draw [ultra thick, brown] (e1c)
to [bend left=2] (e3a)
to [bend right=5] (v3a)
to [bend right=10] (e0c)
to (v4b)
to (e1c);

\draw [ultra thick, brown] (v0a)
to (e1a)
to (v4a)
to (e0b)
.. controls ($ (e0b) + (60:5) $) and ($ (v0a) + (60:5) $) .. (v0a);

\draw [ultra thick, brown] (v0b)
.. controls ($ (v0b) + (210:10) $) and ($ (e3c) + (200:9) $) .. (e3c)
to [out=300, in=220] (v3b) 
to [out=20, in=300] (e0a)
.. controls ($ (e0a) + (60:8) $) and ($ (v0b) + (60:8) $) .. (v0b);

\end{tikzpicture}
\end{minipage}
\end{center}
\caption{Left: The red graph of the trinity of figure \ref{fig:trinity}, whose median construction was seen in figure \ref{fig:median_construction}. Right: The sutured manifold $(M_G, L_G)$. The boundary $\partial M_G = \partial N(G)$ is indicated in brown. The set of sutures/link $L_G$ is drawn in black.}
\label{fig:M_G}
\end{figure}

The median construction of $(M_G, L_G)$ is actually a special case of a more general construction. For any oriented link $L \subset S^3$ with a Seifert surface $F$, we may split $S^3$ along $F$ to obtain a sutured 3-manifold, which we denote $S^3 (F)$. The boundary of $S^3 (F)$ consists of two copies $F^+, F^-$ of $F$ meeting along $L$, and $S^3 (F)$ is naturally a sutured 3-manifold, with set of sutures $L$ and positive and negative regions $F^+, F^-$. We observe that $(M_G, L_G) = S^3 ( F_G )$, where $F_G$ is the Seifert surface from the median construction on $G$.

\subsection{Sutured Floer homology}
\label{sec:SFH}

Sutured Floer homology is an invariant of certain sutured 3-manifolds, introduced by Juh\'{a}sz in \cite{Ju06}. It extends Ozsv\'{a}th--Szab\'{o}'s theory of Heegaard Floer homology \cite{OS04Prop, OS04Closed, OS06} --- which provides invariants of closed 3-manifolds --- to the case of sutured manifolds. Heegaard Floer theory itself builds upon the work of Floer \cite{Floer_symplectic_fixed_points} and the pseudoholomorphic curve theory of Gromov  \cite{Grom}. 

Precisely, sutured Floer homology $SFH(M, \Gamma)$ is an invariant of \emph{balanced} sutured 3-manifolds $(M, \Gamma)$: the balanced condition means that $\chi(R_+) = \chi(R_-)$. 
It is defined in several steps, which we now discuss very roughly. The details of the construction in the rest of this section are, however, not needed for the sequel.
	
The first step is to find a Heegaard diagram $(\Sigma, \alpha, \beta)$ for $(M, \Gamma)$. As in a Heegaard diagram for a closed 3-manifold, $\Sigma$ is an oriented surface, and $\alpha$ and $\beta$ are collections of disjoint closed curves on $\Sigma$, with curves of $\alpha$ bounding discs on one side of $\Sigma$, and curves of $\beta$ bounding discs on the other side of $\Sigma$. However, in the sutured case $\Sigma$ has boundary. The number of curves in $\alpha$ and $\beta$ must be equal (this is equivalent to the balanced condition), but (unlike the closed case) this number may be less than the genus of $\Sigma$. The 3-manifold $M$ can be constructed by gluing thickened discs to a thickened $\Sigma$ (i.e. $\Sigma \times [0,1]$) along $\alpha \times \{0\}$ and $\beta \times \{1\}$. The sutures are given by $\Gamma = \partial \Sigma \times \{1/2\}$; the region $R_+$ consists of $\Sigma \times \{1\}$, surgered along $\beta \times \{1\}$, together with $\partial \Sigma \times (1/2,1)$; the region $R_-$ consists of $\Sigma \times \{0\}$, surgered along $\alpha \times \{0\}$, together with $\partial \Sigma \times (0,1/2)$.

The second step is, given the Heegaard diagram $(\Sigma, \alpha, \beta)$ for $(M, \Gamma)$, to consider pseudoholomorphic curves in $\Sigma \times [0,1] \times \R$. We consider such curves with boundary conditions $\alpha \times \{0\} \times \R$ and $\beta \times \{1\} \times \R$. At $\pm \infty$ in the $\R$ coordinate, such a curve approaches a point in $\alpha \cap \beta$. In particular we consider such curves which at $\pm \infty$ approach complete intersections of the $\alpha_i \cap \beta_j$. Setting up appropriate almost complex and symplectic structures, one can define moduli spaces of such holomorphic curves. These moduli spaces are finite-dimensional and their dimension is given by the index of the corresponding Fredholm problem.

The third step is to define a chain complex generated by the asymptotic conditions of these holomorphic curves, that is, generated by complete intersections of the $\alpha_i \cap \beta_j$. The differential in this chain complex is defined by counts of rigid holomorphic curves which have prescribed complete intersections as asymptotes at $\pm \infty$. The philosophy of this Floer chain complex is that it can be regarded as an infinite-dimensional version of a Morse complex.

One can show that the homology of this complex only depends on the original sutured manifold $(M, \Gamma)$, and does not depend on any of the other choices (Heegaard diagram, almost complex structure, etc.) made along the way. This homology is denoted $SFH(M, \Gamma)$ and is called \emph{sutured Floer homology}. 

\subsection{The spin-c grading}

Sutured Floer homology $SFH(M, \Gamma)$ is an abelian group which naturally has two gradings: the \emph{spin-c} grading and the \emph{Maslov} grading. For present purposes we only need the spin-c grading. 

A \emph{spin-c structure} on a 3-manifold $M$ is a homology class of nonvanishing vector fields; two nonvanishing vector fields are homologous is they are homotopic through nonvanishing vector fields in the complement of a 3-ball in $M$ \cite{OS04Closed, Turaev97_torsion}. A spin-c structure on a \emph{sutured} 3-manifold $(M, \Gamma)$ is a homology class of nonvanishing vector fields whose restriction to $\partial M$ is a fixed vector field adapted to the sutures $\Gamma$ and the regions $R_+, R_-$ \cite{Ju06}.

For present purposes, however, the precise definition of a spin-c structure is not so important. What is important is that each generator of the Floer chain complex has an associated spin-c structure, and the differential preserves it. Thus 
\[
SFH(M, \Gamma) = \bigoplus_{\s \in \Spin^c (M, \Gamma)} SFH(M, \Gamma, \s)
\]
where $\Spin^c (M, \Gamma)$ is the set of spin-c structures on $M$, and $SFH(M, \Gamma, \s)$ is the homology of the subcomplex with spin-c structure $\s$. 

Spin-c structures on $(M, \Gamma)$ are in bijective correspondence with $H^2 (M, \partial M) \cong H_1 (M)$ (singular homology with $\Z$ coefficients); indeed, spin-c structures form an affine set over this group. Note that a priori, since $H_1 (M)$, and hence $\Spin^c (M, \Gamma)$ may in general be infinite, there may be infinitely many summands in the above direct sum. However, the chain complex is finitely generated (there are only finitely many intersection points of $\alpha \cap \beta$), so only finitely many of the summands are nonzero. 

Thus, for any balanced sutured manifold, $SFH(M, \Gamma)$ is a finitely generated group, which is graded by $\Spin^c (M, \Gamma)$, an affine set over $H_1 (M)$.

\subsection{From bipartite graph to sutured Floer homology}
\label{sec:bipartite_graph_SFH}

We saw in section \ref{sec:graph_to_sutured_manifold} how to construct a sutured 3-manifold $(M_G, L_G)$ from a bipartite plane graph $G$. The positive and negative regions $R_+, R_-$ of this sutured manifold are both  homeomorphic to the Seifert surface $F_G$, so $\chi(R_+) = \chi(R_-)$. Thus $(M_G, L_G)$ is balanced. In fact, the general construction of splitting $S^3$ along the Seifert surface $F$ of an oriented link $L$ also yields a balanced sutured 3-manifold $S^3 (F)$.

We may thus consider the sutured Floer homology $SFH(M_G, L_G)$, and more generally, we can consider $SFH(S^3(F))$ for any Seifert surface $F$ of an oriented link in $S^3$.

In \cite{FJR11}, Friedl--Juh\'{a}sz--Rasmussen define a \emph{sutured L-space} to be a balanced sutured 3-manifold $(M, \Gamma)$ such that $SFH(M, \Gamma)$ is torsion free and supported in a single Maslov grading. They showed \cite[cor. 1.7]{FJR11} that for a sutured L-space $(M, \Gamma)$, every spin-c summand $SFH(M, \Gamma, \s)$ is isomorphic to $\Z$ or is trivial. Thus, to know $SFH(M, \Gamma)$, it is sufficient to know the set of spin-c structures $\s$ for which the summand $SFH(M, \Gamma, \s)$ is nontrivial; this set is called the \emph{support} of $SFH(M, \Gamma)$ and is denoted $\Supp (M, \Gamma)$. Thus
\[
\Supp (M, \Gamma) = \left\{ \s \in \Spin^c (M, \Gamma) \; \mid \; SFH (M, \Gamma, \s) \neq 0 \right\}.
\]
Friedl--Juh\'{a}sz--Rasmussen showed \cite[cor. 6.11]{FJR11} that if $L$ is a non-split alternating oriented link in $S^3$, and $R$ is a minimal genus Seifert surface, then $S^3 (R)$ is a sutured L-space.

As discussed in section \ref{sec:median_construction}, when $G$ is a connected plane bipartite graph, the oriented link $L_G$ is non-split and alternating, and $F_G$ is a minimal genus Seifert surface. Thus the results of Friedl--Juh\'{a}sz--Rasmussen imply immediately that $(M_G, L_G)$ is a sutured L-space, and hence that $SFH(M_G, L_G)$ is determined by its support. For a spin-c structure $\s \in \Supp (M_G, L_G)$, we have $SFH(M, \Gamma, \s) = \Z$; for any other spin-c structure, $SFH(M_G, L_G, \s) = 0$. 

The support $\Supp (M_G, L_G)$ is a finite subset of $\Spin^c (M_G, L_G)$. And we have seen that $\Spin^c (M_G, L_G)$ is affine over $H_1 (M_G)$. Since $M_G$ is a handlebody of genus $b_1 (G)$, we have $H_1 (M_G) \cong \Z^{b_1 (G)}$. Hence the support can be regarded as a discrete integer lattice,
\[
\Supp (M_G, L_G) \subset \Spin^c (M_G, L_G) \cong \Z^{b_1(G)}.
\]
In turns out that $\Supp (M_G, L_G)$ in fact is the set of integer points of a \emph{polytope} --- a polytope that we have seen before.

\subsection{Sutured Floer homology and polytopes}

Juh\'{a}sz--K\'{a}lm\'{a}n--Rasmussen in \cite{Juhasz-Kalman-Rasmussen12}
described $\Supp (M_G, L_G)$ in terms of hypertree polytopes.

As in section \ref{sec:plane_bipartite_graphs_trinities} we may realise $G$ as the red graph $G_R$ of a trinity $\Tr$, so that the vertex classes of $G$ are violet $V$ and emerald $E$, and the red vertices $R$ correspond to the complementary regions of $G$. 

As a connected graph, $G = G_R$ is homotopy equivalent to a wedge of $b_1 (G)$ circles, and the first Betti number $b_1 (G)$ is one less than the number of complementary regions of $G$. 
Thus $H_1 (M_G) \cong \Z^{|R|-1}$, and $\Supp (M_G, L_G)$ is a finite subset of the affine $\Z^{|R|-1}$ space $\Spin^c (M_G, L_G)$.

On the other hand, each of the six hypergraphs of $\Tr$ has associated polytopes; we focus on the two hypertree polytopes in $\R^R$, namely
\[
\S_{(V,R)} \quad \text{and} \quad \S_{(E,R)}.
\]
By theorem \ref{thm:polytope_reflections}, these two polytopes are related by a reflection in $\R^R$.

The result of Juh\'{a}sz--K\'{a}lm\'{a}n--Rasmussen can be stated as follows.
\begin{thm}[\cite{Juhasz-Kalman-Rasmussen12} thm. 1.1]
\label{thm:SFH_support_polytope}
Let $G$ be a plane bipartite graph with vertex classes $V,E$ and complementary regions $R$. Then
\[
\Supp (M_G, L_G) \cong S_{(E,R)} \cong -S_{(V,R)}.
\]
\end{thm}
The ``$\cong$" symbols in the statement will be explained shortly. Essentially they denote equivalence as affine sets.
As noted in section \ref{sec:hypertrees_arborescences}, the fact that $\S_{(E,R)}, \S_{(V,R)}$ are related by a reflection means that $\S_{(E,R)}$ and $-\S_{(V,R)}$ are related by a translation in $\R^R$.

Note that while $S_{(E,R)} \subset \Z^R$, all hypertrees of $(E,R)$ actually lie in an $(|R|-1)$-dimensional affine subspace of $\Z^R$. Recall (definition \ref{def:hypertree}) that a hypertree $f_\TT$ of $(E,R)$ is given by $\deg_R \TT - (1, \ldots, 1) \in \Z^R$, where $\TT$ is a spanning tree of $(E,R)$. Thus the sum of the coordinates of $f_\TT$ is $\sum_{r \in R} \deg_r \TT - |R|$, which is $|R|$ less than the number of edges of $\TT$. But all spanning trees of a graph have the same number of edges, and hence for any hypertree $f_\TT \in \Z^R$, the sum of the coordinates of $f_\TT$ is a constant. The set of points in $\Z^R$ with coordinates summing to this constant is thus a codimension-1 affine subspace containing all the hypertrees of $(E,R)$. Similarly, all hypertrees of $(V,R)$ lie in a (generally distinct) $(|R|-1)$-dimensional affine subspace of $\R^R$.

Hence, all three objects $\Supp (M_G, L_G)$, $S_{(E,R)}$ and $-S_{(V,R)}$ can be regarded as lying in an affine spaces over $\Z^{|R|-1}$. The $\cong$ symbols in theorem \ref{thm:SFH_support_polytope}  mean that there are affine identifications of these spaces yielding bijections between $\Supp (M_G, L_G)$, $S_{(E,R)}$ and $-S_{(V,R)}$.

If we have a trinity $\Tr$, with its three bipartite graphs $G_V, G_E, G_R$ and three sutured manifolds $(M_{G_V}, L_{G_V})$, $(M_{G_E}, L_{G_E})$, $(M_{G_R}, L_{G_R})$, then the it follows that the three support sets
\[
\Supp (M_{G_V}, L_{G_V}), \quad
\Supp (M_{G_E}, L_{G_E}), \quad
\Supp (M_{G_R}, L_{G_R})
\]
have the same size, equal to the magic number of $\Tr$ --- despite lying in distinct spaces, in general of different dimensions. Indeed, theorem \ref{thm:SFH_support_polytope} says that the three support sets are related in the same way as the hypertree polytopes of $\Tr$. So the $SFH$ groups of these sutured manifolds satisfy analogous duality and triality relations --- and can be regarded as a ``categorified" version.

Since all three sutured manifolds are sutured L-spaces, for each manifold the size of the support is equal to the dimension of $SFH$. We deduce that the three sutured manifolds have $SFH$ free abelian of dimension given by the magic number of $\Tr$.

The idea viewing the support of $SFH$ as a polytope in $\Spin^c (M, \Gamma)$ is fruitful more generally. For instance, Juh\'{a}sz has shown that if we decompose a sutured manifold $(M, \Gamma)$ along a (sufficiently nice) surface into another sutured 3-manifold $(M', \Gamma')$, then the effect on $SFH$ is to restrict the support to an affine subspace of $\Spin^c (M, \Gamma)$. See \cite{Ju08, Juhasz10_polytope}.

\section{Contact structures}
\label{sec:contact}

Having discussed a wide range of results, we now turn to our recent work with K\'{a}lm\'{a}n \cite{Kalman-Mathews_Trinities}, which introduces \emph{contact structures} into the story.

\subsection{Contact topology}
\label{sec:contact_topology}

We give a very brief introduction to contact geometry and topology, and refer to \cite{Geiges_Introduction} for general background.

A \emph{contact structure} $\xi$ on a 3-manifold $M$ is a non-integrable 2-plane distribution on $M$; the pair $(M, \xi)$ forms a \emph{contact 3-manifold}. Such a $\xi$ consists of a choice of 2-plane in the tangent space of $M$ at each point; the non-integrability of $\xi$ means that there is no surface in $M$ whose tangent planes agree with $\xi$. Locally a contact structure can always be written as the kernel of a 1-form $\alpha$, and the non-integrability of $\xi$ is equivalent to the non-vanishing of the 3-form $\alpha \wedge d\alpha$, i.e. that this 3-form is a volume form. The \emph{standard} contact structure on $\R^3$ is given as the kernel of $dz - y \; dx$, illustrated in figure \ref{fig:standard_contact_structure}.

\begin{figure}
\begin{center}
\includegraphics{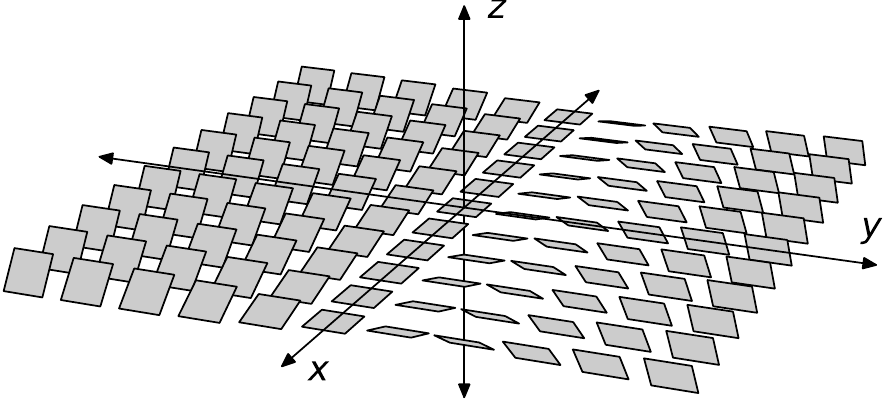}
\end{center}
\caption{The standard contact structure on $\R^3$. The contact planes are shown along $z=0$; translating these planes in the $z$-direction gives the full contact structure.}
\label{fig:standard_contact_structure}
\end{figure}

Contact structures exist in any \emph{odd} number of dimensions (as a \emph{maximally} non-integrable codimension-1 plane field), and form the odd-dimensional counterpart of symplectic geometry. The field goes back at least to Lie's work on differential equations, and arguably to Christiaan Huygens \cite{Geiges01, Geiges05}.

The non-integrability of contact structures might suggest that they are intractable, but in fact they possess a great deal of structure and symmetry. For instance, any contact structure $\xi$ has an infinite-dimensional space of \emph{contact vector fields}, that is, vector fields whose flow preserves $\xi$. Contact vector fields are infinitesimal symmetries of $\xi$, so the space of symmetries is enormous.

Given any 3-manifold $M$, one might ask: what are the contact structures on $M$? Martinet in 1971 showed that every 3-manifold has a contact structure \cite{Martinet71}, and since non-integrability is an open condition, every 3-manifold in fact has infinitely many distinct contact structures. A more reasonable question is to ask how many \emph{isotopy classes} of contact structures there are on $M$, where two contact structures on $M$ are \emph{isotopic} if they are related by a homotopy of 2-plane fields through contact structures.

Eliashberg in \cite{ElOT} illuminated the crucial distinction between two types of contact structures: \emph{tight} and \emph{overtwisted}. A contact structure is overtwisted if it contains a specific contact submanifold called an \emph{overtwisted disc}; otherwise it is tight. Eliashberg showed that the space of overtwisted contact structures on $M$ is weakly homotopy equivalent to the space of 2-plane distributions on $M$. Thus, the classification of overtwisted contact structures is reduced to a problem in homotopy theory, and can be understood by obstruction-theoretic methods (see e.g. \cite{Geiges_Introduction}).

The tight contact structures on a 3-manifold $M$ are much more subtle, and carry interesting information about its topology. Colin--Giroux--Honda showed that the number of isotopy classes of tight contact structures on a closed oriented irreducible 3-manifold $M$ is finite if and only if $M$ is atoroidal \cite{Colin_Giroux_Honda09}.

In general it is a difficult problem to classify all the tight contact structures on a 3-manifold. Results are known for various classes of 3-manifolds, such as lens spaces, and certain types of bundles (e.g. \cite{Gi00, GiBundles, Hon00I, Hon00II}). There exists a closed 3-manifold with no tight contact structure \cite{EH_nonexistence}. Part of the present work is a classification of tight contact structures on the manifolds $(M_G, L_G)$, as we will see shortly.

One useful technique to classify contact structures is to use the notion of \emph{convex surface} introduced by Giroux \cite{Gi91}. A surface $S$ in $(M, \xi)$ is \emph{convex} if there is a contact vector field $X$ transverse to $S$. This definition is not important for present purposes; more important are the following two facts about convex surfaces. Firstly, convex surfaces are \emph{generic}; any embedded surface in a contact 3-manifold is $C^\infty$ close to a convex surface. Secondly, a convex surface $S$ naturally has the structure of a \emph{sutured surface} (section \ref{sec:sutured_manifolds}): there is a natural set of sutures $\Gamma$ on $S$, dividing the surface into positive and negative regions $R_+, R_-$ in coherent fashion. The curve $\Gamma$ is called the \emph{dividing set} and consists of points on $S$ where $X \in \xi$; thinking of $S$ as ``horizontal" and $X$ as ``vertical", $\Gamma$ is the locus of points where the contact planes are ``vertical". If we imagine that the two sides of each contact plane are coloured black and white, then $R_+, R_-$ correspond to which side of the contact plane is visible ``from above". See figure \ref{fig:convex_disc}. In fact, the dividing set / set of sutures $\Gamma$ essentially determines the contact structure near $S$, up to isotopy, in a sense which can be made precise.

\begin{figure}
\begin{center}
\includegraphics{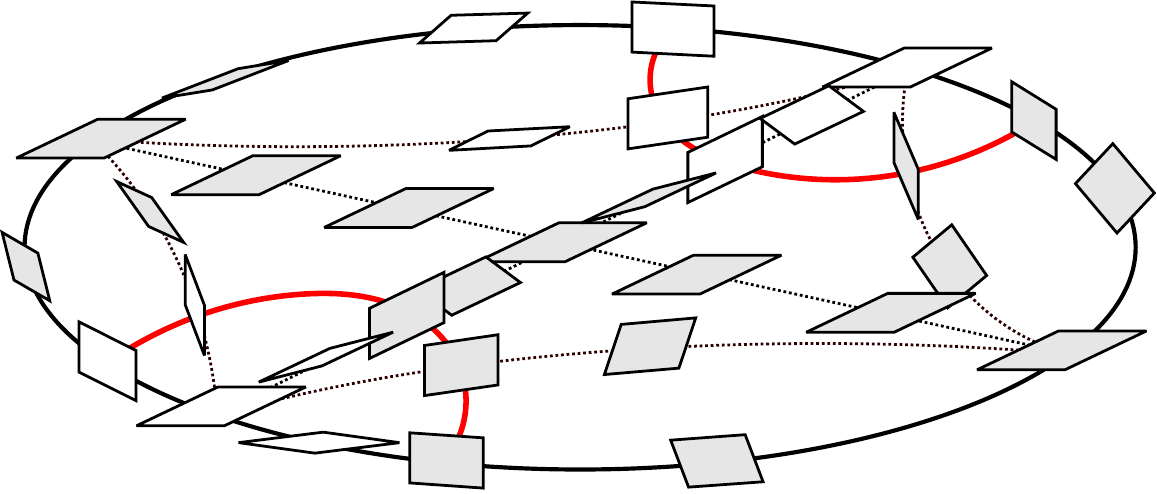}
\end{center}
\caption{A convex disc with contact structure. The dividing set / set of sutures is shown in red; the two sides of contact planes are coloured black and white.}
\label{fig:convex_disc}
\end{figure}

Since a set of sutures naturally describes a contact structure near a surface, a sutured 3-manifold $(M, \Gamma)$ has a naturally prescribed contact structure $\xi_\partial$ near its boundary. A \emph{contact structure} on $(M, \Gamma)$ is then a contact structure on $M$ which restricts to $\xi_\partial$ near $\partial M$. From this perspective, sets of sutures provide natural boundary conditions for contact structures.

\subsection{Contact invariants in sutured Floer homology}

Contact structures are closely related to sutured Floer homology. A contact structure $\xi$ on a balanced sutured 3-manifold $(M, \Gamma)$ yields a \emph{contact invariant} $c(\xi)$ in sutured Floer homology, well-defined up to sign \cite{HKM09, HKMContClass, OSContact}. Because of technical issues, there is an orientation reversal, so $c(\xi) \in SFH(-M, -\Gamma)$. When $\xi$ is overtwisted, $c(\xi) = 0$; so only tight contact structures have nontrivial contact invariants.

A contact structure $\xi$ on $(M, \Gamma)$ naturally has a spin-c structure $\s_\xi \in \Spin^c (M, \Gamma)$, given by the homology class of a vector field forming an orthogonal complement of $\xi$ (see e.g. \cite{OS04Closed, Turaev97_torsion}). Sutured Floer homology is also graded by spin-c structures, and the contact invariant $c(\xi)$ lies in the summand of $SFH(-M,-\Gamma)$ corresponding to the spin-c structure $\s_\xi$: that is, $c(\xi) \in SFH(-M, -\Gamma, \s_\xi)$ \cite{Juhasz16_functoriality}.

\subsection{Contact structures and trinities}

Let us then turn to the sutured 3-manifolds considered so far in our story, namely those of the form $(M_G, L_G)$. 
In the light of our discussion of contact structures, a natural question to to ask is: how many isotopy classes of tight contact structures are there on $(M_G, L_G)$?

More generally, we may consider a trinity $\Tr$, and the three bipartite plane graphs $G_V, G_E, G_R$ obtained from it. How many isotopy classes of tight contact structures are there on the three associated sutured manifolds?

As with all the magic in this subject, there is only one possible answer to this question.

\begin{thm}[\cite{Kalman-Mathews_Trinities}]
\label{thm:tight_ct_strs_trinities}
Let $G$ be a finite connected bipartite plane graph. The number of isotopy classes of tight contact structures on $(M_G, L_G)$ is equal to the number of hypertrees in either hypergraph of $G$.
\end{thm}
Of course, this number of hypertrees is the magic number of the trinity $\Tr$ associated to $G$, and so in fact all three sutured manifolds have the same number of tight contact structures.

As discussed in sections \ref{sec:median_construction} and \ref{sec:graph_to_sutured_manifold}, any minimal genus Seifert surface $F$ of a non-split prime special alternating link arises from the median construction on some connected bipartite plane graph $G$, and we then have $(M_G, L_G) = S^3 (F)$. Hence we obtain the following.
\begin{cor}[\cite{Kalman-Mathews_Trinities}]
Let $F$ be a minimal genus Seifert surface for a non-split prime special alternating link in $S^3$. Then the number of isotopy classes of tight contact structures on $S^3 (F)$ is equal to the number of hypertrees in a hypergraph giving the median construction of $F$.
\end{cor}

The proof of theorem \ref{thm:tight_ct_strs_trinities} relies heavily on Giroux's theory of convex surfaces \cite{Gi91}, and a theorem of Honda on gluing tight contact structures \cite{Hon02}. Using these results, we reduce the problem of classifying isotopy classes of tight contact structures on $(M_G, L_G)$ to a combinatorial problem about curves on discs in the complementary regions of $G$. The proof applies results from \cite{Kalman13_Tutte} and \cite{Me09Paper}.

We may also consider the spin-c structures of the tight contact structures on the sutured 3-manifold $(M_G, L_G)$. These lie in $\Spin^c (M_G, L_G)$ --- the same space in which $\Supp (M_G, L_G)$ lies. They are in fact the same set.

Indeed, an extension of the reasoning in the proof of theorem \ref{thm:tight_ct_strs_trinities}, together with results of Juh\'{a}sz--K\'{a}lm\'{a}n--Rasmussen \cite{Juhasz-Kalman-Rasmussen12}, gives the following result.
\begin{thm}[\cite{Kalman-Mathews_Trinities}]
\label{thm:tight_ct_str_spin-c}
Each isotopy class of tight contact structure on $(M_G, L_G)$ has a distinct spin-c structure. For a given $\s \in \Spin^c (M_G, L_G)$, a tight contact structure $\xi_\s$ with spin-c structure $\s$ exists on $(M_G, L_G)$ if and only if $SFH(-M_G, -L_G, \s)$ is nontrivial. 
\end{thm}
Denoting the set of spin-c structures of tight contact structures on $(M_G, L_G)$ by $\Spin^c\Tight(M_G,L_G)$, we then have
\[
\Spin^c\Tight(M_G, L_G) = \Supp (M_G, L_G).
\]
By theorem \ref{thm:SFH_support_polytope} then $\Spin^c \Tight (M_G, L_G)$ is, up to affine equivalence, the set of hypertrees of the two related hypergraphs, and inherits all their dualities.

Thus for the three sutured manifolds of a trinity, the spin-c structures of their tight contact structures obey all the duality and triality relationships of the corresponding hypertree polytopes --- and their number is the magic number.

We can also compute the contact invariants of the tight contact structures on $(M_G, L_G)$. As mentioned in section \ref{sec:contact_topology}, if a contact structure $\xi$ has spin-c structure $\s$, then its contact structure $c(\xi)$ lies in the corresponding spin-c summand $SFH(-M,-\Gamma,\s)$. By theorem \ref{thm:tight_ct_str_spin-c}, if $\xi$ is tight then this summand is nonzero; since $(M_G, L_G)$ is a sutured L-space, as discussed in section \ref{sec:bipartite_graph_SFH} we have $SFH(-M,-\Gamma,\s) \cong \Z$. The answer for $c(\xi)$ turns out to be the nicest possible one.
\begin{thm}[\cite{Kalman-Mathews_Trinities}]
For the tight contact structure $\xi_\s$ on $(M_G, L_G)$ with spin-c structure $\s$, the contact invariant $c(\xi)$ is a generator of $SFH(-M_G, -L_G, \s) \cong \Z$.
\end{thm}
(Note that as $c(\xi)$ is ambiguous up to sign, there is no preferred generator of the $\Z$ summand.)

This theorem follows from the topological quantum field theory properties of sutured Floer homology \cite{HKM08}, and showing that each tight contact structure on $(M_G, L_G)$ extends to a tight contact structure on $S^3$.

The list of objects equal to the magic number is now quite enormous. Denoting the set of isotopy classes of tight contact structures on a sutured 3-manifold $(M, \Gamma)$ by $\Tight (M, \Gamma)$, we have seen that all of the following quantities are equal to the magic number:
\begin{align*}
|\det M_\Tr| &= \rho(G_V^*) = \rho(G_E^*) = \rho(G_R^*) = |\M_{\mathfrak{T}}| \\
&= | S_{(V,E)} | = | S_{(E,V)} | = | S_{(E,R)} | = | S_{(R,E)} | = | S_{(R,V)} | = | S_{(V,R)} | \\
&= | P_{(V,E}^- | = | P_{(E,V)}^- | = | P_{(E,R)}^- | = | P_{(R,E)}^- | = | P_{(R,V)}^- | = | P_{(V,R)}^- | \\
&= | \Supp (M_{G_V}, L_{G_V}) | = | \Supp (M_{G_E}, L_{G_E}) |
= | \Supp (M_{G_R}, L_{G_R} | \\
&= \dim SFH(M_{G_V}, L_{G_V}) = \dim SFH(M_{G_E}, L_{G_E}) = \dim SFH(M_{G_R}, L_{G_R}) \\
&= | \Tight (M_{G_V}, L_{G_V}) | = | \Tight (M_{G_E}, L_{G_E}) | = | \Tight (M_{G_R}, L_{G_R}) | \\
&= | \Spin^c\Tight (M_{G_V}, L_{G_V}) | = | \Spin^c\Tight (M_{G_E}, L_{G_E})| = | \Spin^c\Tight (M_{G_R}, L_{G_R}) |.
\end{align*}

Moreover we have seen equalities, or at least affine equivalences, of many of these sets and the related polytopes, such as
\[
P_{(R,E)}^- = S_{(E,R)} \cong -S_{(V,R)} = -P_{(R,V)}^- \cong 
\Supp (M_{G_R}, L_{G_R}) = \Spin^c \Tight (M_{G_R}, L_{G_R}),
\]
and dualities and trialities between them.

However, this is still far from a list of objects given by the magic number. In \cite{Kalman-Mathews_Trinities} it is also shown that the number of states of a \emph{universe}, in the \emph{formal knot theory} of Kauffman \cite{Kauffman_FKT83}, is the magic number of a corresponding trinity. The leading coefficients of the Alexander polynomials of $L_{G_V}, L_{G_E}, L_{G_R}$ are also given by the magic number.

\addcontentsline{toc}{section}{References}

\small

\bibliography{danbib}
\bibliographystyle{amsplain}

\end{document}